\documentclass[12pt]{article}

\usepackage[ansinew]{inputenc}
\usepackage[UKenglish]{babel}
\usepackage{textcomp}
\usepackage{mathrsfs}
\usepackage{amsfonts}
\usepackage{mathtools}
\usepackage{amssymb}
\usepackage{amsmath}
\usepackage{amsthm}
\usepackage{amscd}
\usepackage{tensor}
\usepackage{stmaryrd}
\usepackage{rotating}
\usepackage[text={6.8in,8.5in},centering]{geometry}
\usepackage{tikz}
\usetikzlibrary{arrows, calc, decorations.markings, positioning}
\usepackage{tikz-cd}
\usepackage{relsize} 
\usepackage[none]{hyphenat}
\usepackage{hyperref}

\usepackage{graphicx}
\graphicspath{ {./Images/} }

\usepackage{hyperref,xcolor}
\usepackage{makeidx}
\definecolor{deepred}{rgb}{0.5,0,0}
\definecolor{deepblue}{rgb}{0,0,0.5}
\definecolor{deepgreen}{rgb}{0,0.5,0}

\hypersetup{
    colorlinks=true,
    linkcolor=deepred,
    filecolor=deepred,      
    urlcolor=deepred,
    citecolor=deepblue,
}

\newcommand{\Int}{\mathbb{Z}}

\newcommand{\Real}{\mathbb{R}}

\newcommand{\Cin}[1]{\text{C}^\infty(#1)}
\newcommand{\Tan}{\text{T}}
\newcommand{\Cot }{\text{T}^*}
\newcommand{\Sec}[1]{\Gamma(#1)}
\newcommand{\Der}{\text{D}}
\newcommand{\Jet}{\text{J}}
\newcommand{\Proj}{\text{pr}}
\newcommand{\Id}{\text{id}}

\newcommand{\Dim}[1]{\text{Dim}(#1)}
\newcommand{\Dr}[1]{\text{Der}(#1)}

\newcommand{\Ker}[1]{\text{ker}(#1)}

\newcommand{\LGrph}[1]{\text{Lgrph}(#1)}

\newcommand{\Man}{\textsf{Man}}

\newcommand{\Ring}{\textsf{Ring}}
\newcommand{\Vect}{\textsf{Vect}}
\newcommand{\LVect}{\textsf{LVect}}
\newcommand{\Line}{\textsf{Line}}

\newcommand{\Jac}{\textsf{Jac}}
\newcommand{\Acts}{\mathbin{\rotatebox[origin=c]{-90}{$\circlearrowright$}}}
\newcommand{\dtimes}{\mathbin{\rotatebox[origin=c]{90}{$\ltimes$}}}
\newcommand{\utimes}{\mathbin{\rotatebox[origin=c]{-90}{$\ltimes$}}}

\newtheorem{prop}{Proposition}
\numberwithin{prop}{subsection}

\newtheorem{thm}{Theorem}
\numberwithin{thm}{subsection}

\usepackage[
backend=biber,
style=alphabetic,
sorting=ynt
]{biblatex}
\begin{document}

\title{\textbf{Poly-Jacobi Manifolds}:\\ the Dimensioned Approach to Jacobi Geometry}

\author{Carlos Zapata-Carratala}
%\address{School of Mathematics, The University of Edinburgh, James Clerk Maxwell Building, Peter Guthrie Tait Road, Edinburgh EH9 3FD}
%\email{c.zapata.carratala@gmail.com}
\date{}

\maketitle

\sloppy

\begin{abstract}
    The standard formulation of Jacobi manifolds in terms of differential operators on line bundles, although effective at capturing most of the relevant geometric features, lacks a clear algebraic interpretation similar to how Poisson algebras are understood to be the algebraic counterpart of Poisson manifolds. We propose a formalism, based on the dimensioned algebra technology recently developed by the author, to capture the algebraic counterparts of Jacobi manifolds as dimensioned Poisson algebras. Particularly, we give a generalisation of the functor $\text{C}^\infty:\textsf{Smooth}\to \textsf{Ring}$ for line bundles and Jacobi manifolds, we show that coisotropic reduction of Jacobi manifolds provides an example of algebraic reduction of dimensioned Poisson algebras and we discuss the relation between products of Jacobi manifolds and the tensor products of their associated dimensioned Poisson algebras. These results motivate the definition of the category of poly-line bundles, defined as collections of line bundles over the same base manifold, and the corresponding generalisation of Jacobi structures into this context, the so-called poly-Jacobi manifolds. We present some preliminary results about these generalisations and discuss potential future lines of research. An interesting outcome of this research is a somewhat surprising connection between geometric mechanics and dimensional analysis. This is suggested by the fact that if one assumes a phase space to be a poly-Jacobi manifold, the observables -- which are the dimensioned analogue of real-valued functions -- carry a natural partial addition and total multiplication structure identical to the usual algebra of physical quantities in dimensional analysis.
\end{abstract}

\newpage

\tableofcontents

\newpage

\section{Introduction} \label{intro}

The field of contact geometry has enjoyed a resurgence in recent years after a long history of being shadowed by its even-dimensional symplectic sibling. With origins that can be traced to the work of Lie on differential equations in the 1890s, contact and Jacobi geometry developed from the seminal works of Kirillov \cite{kirillov1976local} and Lichnerowicz \cite{flato1976deformations}. See \cite{geiges2001brief} for a history of contact geometry. An important contribution to the recent rekindled interest in Jacobi geometry has been the line bundle approach to Jacobi manifolds, first developed by Marle \cite{marle1991jacobi} and championed by Vitagliano \cite{vitagliano2015dirac,vitagliano2019holo} and his students \cite{schnitzer2019normal} \cite{tortorella2017deformations}. To this day, the line bundle approach has proven the most effective at capturing the seemingly diverse zoology of Poisson-like structures that were identified in the literature in the last few decades. By considering a generic line bundle over a smooth manifold with its module of sections -- generalising the usual setting of a smooth manifold with its ring of functions -- it was possible to account for many previously known geometries: locally conformal Poisson (including ordinary Poisson and symplectic) \cite{lichnerowicz1986representation}, Jacobi \cite{ibanez1997coisotropic} and (pre)contact \cite{mrugala1991contact}.\newline

The aforementioned line bundle approach to Jacobi geometry has been encapsulated by the author as the Unit-Free formalism in the prequel paper \cite{zapata2020unitfree}. There, Poisson-like structures on line bundles were found to be closely related to a preliminary approach towards implementing physical dimension and units of measurement into geometric mechanics. In the unit-free approach to geometric mechanics one considers the observables of a physical system, represented by a smooth manifold $P$, to be given by the sections of some generic line bundle $L$ over $P$. This redefinition of observables, replacing the ring of functions $\Cin{P}$ by the module of sections $\Sec{L_P}$, proves very useful in translating usual Hamiltonian mechanics, which relies on a Poisson structure on functions, into the more general Jacobi setting. Despite the conceptual clarity and unification it provides, the unit-free formulation has one fundamental flaw: it removes -- or, at best, conceals -- the explicit algebraic structure of observables.\newline

This flaw of the unit-free formulation amounts to the sections of vector bundles being modules over rings of functions while lacking an internal multiplicative operation. This fact prompts the problem that we set out to solve in the present work. On the one hand, from a purely mathematical point of view, the lack of a ring-like structure limits the interpretation of the algebraic counterpart of Jacobi manifolds. It is well known that much in the same way that one considers rings of functions to be the algebraic counterparts of smooth manifolds, Poisson algebras are the counterparts of Poisson manifolds with natural correspondences of morphisms, tensor products and reduction \cite{fernandes2014lectures}; this is indeed just the restriction of the functor $\text{C}^\infty : \Man\to\Ring$ to the category of Poisson manifolds. Such a correspondence is not known for Jacobi manifolds, at least not in a directly analogous way: morphisms, products and reductions are well-understood geometrically but there is no known functor that will realise these constructions in a category of algebras.\newline

\textbf{Problem 1.} \emph{Can we find a category of algebraic structures and a functorial correspondence with the category of Jacobi manifolds that realises morphisms, reductions and products?}\newline

On the other hand, from a physical point of view, sections of a line bundle, without the presence of a multiplicative structure, are a dubious choice to represent observables. Physical quantities carry additive and multiplicative structures \cite{barenblatt1996scaling}; it is indeed commonplace to model observables in mathematical physics as algebras \cite{landsman2012mathematical}. Incidentally, an ordinary algebra structure seems insufficient for a realistic account of physical observables since it ignores all aspects of dimensional analysis. Physical quantities used in practice carry a partition or labelling that makes addition partially-defined and makes multiplication compatible with the labelling. Any formulation of geometric mechanics that aims to faithfully capture physical dimension and units of measurement would have to incorporate algebraic structures that share the aforementioned labelling, partial addition and multiplication.\newline

\textbf{Problem 2.} \emph{Is it possible to formulate a theory of Hamiltonian mechanics where observables carry the algebraic structure of the standard dimensional analysis of physical quantities?}\newline

Problems 1 and 2 are not only seemingly unrelated, they belong to entirely different branches of mathematical science. \textbf{However, in a serendipitous turn of events, the solution to 2 happens to lead to the solution to 1!} This is accomplished via the theory of Dimensioned Algebra introduced by the author in the prequel paper \cite{zapata2021dimensioned}. Dimensioned algebraic structures are sets with labels and binary operations that are either partially-defined on labelled subsets or are totally-defined but satisfy some natural compatibility with the labelling. If one attempts to develop the dimensioned analogue of commutative algebra, the notion of dimensioned ring, whose addition operation is partially defined but is otherwise defined via the same axioms, appears naturally, and it turns out that it exactly matches the structure of the dimensional analysis of physical quantities \cite{zapata2021dimensioned}. The key observation is that one can construct a dimensioned ring from the sections of tensor powers of a line bundle and its dual. Proposition \ref{PowerFunctorLineBundles} indeed shows that there is a functor $\odot: \Line_\Man \to \textsf{DimRing}$ mapping the category of line bundles to the category of dimensioned rings. This construction is the direct analogue\footnote{It is in fact a strict generalisation since one can regard manifolds as trivialised line bundles.} of the ordinary $\text{C}^\infty$ functor for the category of smooth manifolds.\newline

Interestingly, the dimensioned point of view suggests the definition of a natural class of structures generalising Jacobi manifolds that have not been previously identified in the literature. It turns out that the power functor $\odot$ can be defined not only for a single line bundle but for any ordered (finite) collection of line bundles over the same base manifold. Assuming the existence of a dimensioned Poisson algebra structure on the associated dimensioned ring gives a direct generalisation of a Jacobi manifold in this context. Furthermore, this is directly analogous to the notion that a Poisson manifold is a smooth manifold whose ring of functions carries a Poisson algebra structure. We call these structures \textbf{Poly-Jacobi Manifolds}.\newline

The contents are organized as follows: In Section \ref{Preliminaries} we cover the preliminaries summarizing both the unit-free and dimensioned formalisms and showing the dimensioned version of linear algebra in some detail to then proceed more swiftly when dealing with the differential geometry. Section \ref{LinesToDim} is dedicated to the power functor construction for general line bundles and associated notions such as derivations, jets and tensor fields. These constructions are then applied to the case of line bundles carrying a Jacobi bracket in Section \ref{JacToDim}: the dimensioned Poisson algebra of a Jacobi manifold is identified in \ref{dimPoiss}, the algebraic counter part of coisotropic reduction of Jacobi manifolds is found in \ref{Quotients} and the tensor product of dimensioned Poisson algebras is shown to capture the product of Jacobi manifolds in \ref{Products}. The generalised notion of a line bundle, called a poly-line bundle, is developed in \ref{polyLine}, we discuss the multiple generalisations of Jacobi manifods in \ref{GeneralisingJacobi} and Poly-Jacobi manifolds are defined and \ref{polyJacobiMan}. We conclude with some remarks and further commentary in Section \ref{commentary}.

\section{The Unit-Free and Dimensioned Formalisms} \label{Preliminaries}

\subsection{Lines and Line-Vector Spaces} \label{Lines}

In this section we begin by introducing the categories that ought to replace conventional fields of numbers and vector spaces. We are motivated by the \emph{unit-free intuition} by which we think of a $1$-dimensional vector space as a field of numbers where the unit has not been specified. A choice of unit will correspond to the choice of basis and thus it will give an isomorphism to a copy of the base field. The unit-free approach is developed by taking this simple idea into the context of algebra and geometry and replacing any instance the base field by a category of lines. It should be noted that most of the definitions and constructions in this section work for a generic field $\mathbb{F}$ but we will focus on the reals for concreteness.\newline

We identify \textbf{the category of lines}, $\Line$, as a subcategory of vector spaces $\Vect$. Objects are vector spaces over the field of real numbers $\mathbb{R}$ of dimension $1$. An object $L\in\Line$ will be called a \textbf{line}. A morphism in this category $B:L\to L'$, is an invertible linear map. Composition in the category $\Line$ is simply the composition of maps. If we think of $L$ and $L'$ as numbers without a choice of a unit, a morphism $B$ between them can be thought of as a unit-free conversion factor, for this reason we will often refer to a morphism of lines as a \textbf{factor}. We consider the field of real numbers, trivially a line when regarded as real a vector space, as a singled out object in the category of lines $\Real\in \Line$.\newline 

Note that all the morphisms in this category are, by definition, isomorphisms, thus making $\Line$ into a groupoid; however, one should not think of all the objects in the category as being equivalent. As we shall see below, there are times when one finds factors between lines (invertible morphisms by definition) in a canonical way, that is, without making any further choices beyond the information that specifies the lines, these will be called \textbf{canonical factors}. \newline

Note that the direct sum $\oplus$ is no longer defined in $\Line$, however, it is straightforward to check that $(\Line, \otimes, \mathbb{R})$ forms a symmetric monoidal category and that $*:\Line \to \Line$ is a duality contravariant autofunctor. The usual isomorphisms associated with the monoidal structure and the duality functor induce canonical factors between combinations of tensor products and duals of lines. In particular, $\Vect(L,L)\cong L^*\otimes L\in\Line$ has a distinguished non-zero element, the identity id$_L$, thus we find that it is canonically isomorphic to $\Real$ as lines. Therefore, for any line $L\in\Line$ we find a canonical factor
\begin{equation}
    p_L:L^*\otimes L\to \Real.
\end{equation}
This last result, under the intuition of lines as numbers without a choice of unit, allows us to reinterpret the singled out line $\Real$ informally as the set of procedures common to all lines by which a number gives any other number in a linear way (preserving ratios). This interpretation somewhat justifies the following adjustment in terminology: we will refer to the tensor unit $\Real\in\Line$ as \textbf{the patron line}. The term \textbf{unit} is reserved for non-vanishing elements of a line $u\in L^\bullet$, where we have denoted $L^\bullet:= L\backslash \{0\}$.\newline

In the same manner to how we have generalised the field of real numbers $\Real$ to the category of lines, we now generalise vector space into a larger category. We do so by identifying the category of \textbf{line-vector spaces} or \textbf{lvector spaces} defined as the product category
\begin{equation}
    \LVect:= \Vect \times \Line.
\end{equation}
Our notation for objects in this category will be $V^L:=(V,L)$ with $V\in\Vect$, $L\in\Line$, and similarly for morphisms $\psi^B:V^L\to W^{L'}$. Objects $V^L$ will be called \textbf{lvector spaces} and morphisms $\psi^B$ will be called \textbf{linear factors}. The \textbf{direct sum} of lvector spaces is defined when the line component is shared:
\begin{equation}
    V^L\oplus_L W^L:=(V\oplus W)^L.
\end{equation}
The notion of \textbf{subspace} and \textbf{quotient} are also similarly defined:
\begin{equation}
    U^L\subset V^L \text{ when } U \subset V \text{ is subspace, } \qquad V^L/U^L:= (V/U)^L.
\end{equation}
Since the line component of a lvector space $V^L$ plays the role of the unit-free field of scalars, the natural notion of dual space should be given by the space of linear maps from the vector space to the line, $\Vect(V,L)\cong V^*\otimes L$, we thus define \textbf{lduality} as follows: the ldual of a lvector space $V^L$ is
\begin{equation}
    V^{*L}:=(V^*\otimes L)^L
\end{equation}
and the ldual of a linear factor $\psi^B:V^L\to W^{L'}$ is
\begin{equation}
    \psi^{*B}:W^{*L'}\to V^{*L}
\end{equation}
with $\psi^{*B}(\beta^{l'})=\alpha^l$ such that
\begin{equation}
    \alpha=B^{-1}\circ \beta \circ \psi \qquad l=B^{-1}(l').
\end{equation}
There is also a natural notion of \textbf{lannihilator} of a subspace in a lvector space:
\begin{equation}
    U^{0L}:=\{\alpha\in V^{*L} | \quad \alpha(u)=0\in L, \forall u\in U\}\cong U^0\otimes L.
\end{equation}

\begin{prop}[$L$-Rooted Subcategories of $\LVect$]\label{LSumsDual}
By fixing a line $\normalfont L\in\Line$, the subcategory of lvector spaces sharing $L$ as line component form an abelian category with duality $\normalfont (\Vect^L,\oplus_L, ^{*L})$ that directly generalises the analogous categorical structures in ordinary vector spaces $\normalfont \Vect$. In particular, we have canonical isomorphic linear factors:
\begin{equation}
    (V^{*L})^{*L}\cong V^L, \qquad (V \oplus W)^{*L}\cong V^{*L}\oplus_L W^{*L}, \qquad V^{*L}/U^{0L}\cong U^{*L}.
\end{equation}
\end{prop}
\begin{proof}
This follows from simple linear algebra arguments exploiting the peculiarities of the category of lines $\Line$: linear maps between lines are also 1-dimensional vector spaces (hence also lines) and the endomorphisms of a line are canonically isomorphic to the patron line $L^*\otimes L\cong \Real$.
\end{proof}

We can continue to generalise linear algebra following the L-rooted approach to define the line version of tensors. The \textbf{covariant ltensors} of a lvector space $V^L$ are defined as the $L$-valued multilinear forms:
\begin{equation}
    \mathcal{T}_k(V^L) := (\Vect(V, \stackrel{k}{\dots} ,V,L))^L=(V^*\otimes \stackrel{k}{\dots} \otimes V^*\otimes L)^L.
\end{equation}
By taking the convention $\mathcal{T}_0(V^L)=L$ we can define the graded lvector spaces of covariant ltensors in the obvious way:
\begin{equation}
    \mathcal{T}_\bullet (V^L) := \bigoplus_{k=0}^\infty \mathcal{T}_k(V^L).
\end{equation}
The \textbf{pull-back} of a linear factor $\psi^B:V^L\to W^{L'}$ is obtained by extending the definition of ldual map above for an arbitrary number of vector arguments
\begin{equation}
    \psi^{*B}: \mathcal{T}_k(W^{L'}) \to \mathcal{T}_k(V^L).
\end{equation}
At degree $0$ the pull-back $\psi^{*B}: L' \to L$ is simply the inverse of the factor $B^{-1}$. This definition of ltensor is due to the peculiarity of the L-rooted approach which is further exacerbated when we attempt to identify a generalisation of the tensor product of lvector spaces. One could simply use the Cartesian product of the monoidal structures in $\Vect$ and $\Line$ to endow $\LVect$ with a natural monoidal structure, however, this will produce tensor products of lines, which we are trying to avoid. This will imply that the tensor algebra of an ordinary vector space doesn't generalise to an object of the same algebraic nature in the context of lvector spaces. In fact, the price we pay for enforcing the L-rooted approach is that we give up algebra structures for modules over algebras.\newline

We can proceed explicitly by using the vector space structure of lines: let $\alpha \in \mathcal{T}_p(V)$ and $\beta\in \mathcal{T}_q(V^L)$, the module product
\begin{equation}
    \alpha \cdot \beta \in \mathcal{T}_{p+q}(V^L)
\end{equation}
is defined by its action on vector arguments
\begin{equation}
    \alpha \cdot \beta (v_1,\dots v_p,w_1,\dots w_q):= \alpha(v_1,\dots v_p) \cdot \beta(w_1,\dots w_q)\in L.
\end{equation}
This product indeed allows us to define the \textbf{module of covariant ltensors}.

\begin{prop}[Modules of Covariant LTensors]\label{CoLTensor}
The space of covariant ltensors of a given lvector space $V^L$ with the product defined above $(\mathcal{T}_\bullet (V^L),+,\cdot)$ is a module over the associative algebra of covariant tensors of the vector component $(\mathcal{T}_\bullet (V),+,\otimes)$. A linear factor $\psi^B:V^L\to W^{L'}$ induces a morphism of modules covering a morphism of associative algebras via its pull-back $\psi^{*B}$, i.e. for all $\alpha\in \mathcal{T}_p(V)$ and $\beta,\gamma\in \mathcal{T}_q(V^L)$
\begin{equation}
    \psi^{*B}(\beta+\gamma)=\psi^{*B}\beta+\psi^{*B}\gamma, \qquad \psi^{B*}(\alpha\cdot \beta)= \psi^*\alpha \cdot \psi^{*B}\beta.
\end{equation}
\end{prop}
\begin{proof}
This follows trivially from basic linear algebra arguments noting that the module product is simply constructed from the vector space structure of lines and the fact that pull-backs are explicitly given by:
\begin{equation}
    \psi^{*B}\beta (v_1,\dots v_q):=B^{-1}(\beta(\psi(v_1),\dots \psi(v_q)), \qquad \psi^{*B}(l):=B^{-1}(l)
\end{equation}
for all $\beta\in \mathcal{T}_q(V^L)$ and $l\in L'$.
\end{proof}

Mirroring ordinary multilinear algebra, \textbf{contravariant ltensors} are defined as the covariant tensors of the ldual:
\begin{equation}
    \mathcal{T}^\bullet (V^L):=\mathcal{T}_\bullet (V^{*L}),
\end{equation}
and the \textbf{push-forward} of a linear factor $\psi^B:V^L\to W^{L'}$ is simply defined as de ldual of the ldual:
\begin{equation}
    \psi^B_*:=(\psi^{*B})^{*B^{-1}}.
\end{equation}

\begin{prop}[Modules of Contravariant LTensors]\label{ContraLTensor}
The space of contravariant ltensors of a given lvector space $V^L$ with the product defined above $(\mathcal{T}^\bullet (V^L),+,\cdot)$ is a module over the associative algebra of contravariant tensors of the vector component $(\mathcal{T}^\bullet (V),+,\otimes)$. A linear factor $\psi^B:V^L\to W^{L'}$ induces a morphism of modules covering a morphism of associative algebras via its push-forward $\psi^B_*$, i.e. for all $a\in \mathcal{T}^p(V)$ and $b,c\in \mathcal{T}^q(V^L)$
\begin{equation}
    \psi^B_*(b+c)=\psi^B_*b+\psi^B_*c, \qquad \psi^B_*(a\cdot b)= \psi_*a \cdot \psi^B_*b.
\end{equation}
\end{prop}
\begin{proof}
This is a direct consequence of Prop. \ref{CoLTensor} and the duality property $(V^{*L})^{*L}\cong V^L$.
\end{proof}

The module structures defined above admit natural symmetrisation and antisymmetrisation. In particular, we can identify the \textbf{exterior modules} of \textbf{lmultivectors} and \textbf{lforms} in the obvious way:
\begin{align}
    \wedge_\bullet (V^L) & := (\wedge^\bullet (V^*)\otimes L)^L,\\
    \wedge^\bullet (V^L) & := (\wedge^\bullet (V\otimes L^*)\otimes L)^L.
\end{align}
It is then a simple check to see that these are finite-dimensional lvector spaces with module structures over the ordinary exterior algebras of multivectors. Furthermore, push-forwards and pull-backs restrict to module morphisms of lmultivectors and lforms respectively.

\subsection{Line Bundles as Unit-Free Manifolds} \label{LineBundles}

As it is customary in differential geometry, now that we have identified some interesting structures at the linear level it is time to smoothly ``smear'' them on manifolds and develop the corresponding bundle generalisation. This will result in the identification of the category of line bundles $\Line_\Man$ and the category of lvector bundles $\LVect_\Man$ whose objects will be fibrations with bases in the category of smooth manifold $\Man$ and fibres in linear categories $\Line$ and $\LVect$, respectively. In this section we shall give precise definitions for these categories and show some elementary constructions in them.\newline

We define the \textbf{category of line bundles} $\Line_\Man$ as the subcategory of $\Vect_\Man$ whose objects are rank $1$ vector bundles $\lambda: L \to M$ and whose morphisms are regular, i.e. fibre-wise invertible, bundle morphisms covering general smooth maps
\begin{equation}
\begin{tikzcd}
L_1 \arrow[r, "B"] \arrow[d, "\lambda_1"'] & L_2 \arrow[d, "\lambda_2"] \\
M_1 \arrow[r, "\varphi"'] & M_2
\end{tikzcd}
\end{equation}
In the interest of brevity, we may sometimes refer to line bundles $L\in \Line_\Man$ as \textbf{lines} and regular line bundle morphisms $B\in\Line_\Man(L_1,L_2)$ as \textbf{factors}. A factor covering a diffeomorphism, i.e. a line bundle isomorphism, is called a \textbf{diffeomorphic factor}. Similarly, a factor covering an embedding or submersion is called an \textbf{embedding factor} or \textbf{submersion factor}, respectively.\newline

In the propositions that follow below we shall see that the usual structural constructions on smooth manifolds, such as submanifolds and products, find a natural generalisation within the category of line bundles. This then justifies the term \textbf{unit-free manifold} for a line bundle since we can regard individual fibres as fields of numbers without a choice of unit, as discussed at the start of Section \ref{Lines}. Unit-free manifolds generalise ordinary smooth manifolds in the sense that what are identical copies of $\Real$ at every point of a manifold in the ordinary case, i.e. any smooth manifold can be regarded as a globally-trivialised line bundle, become distinct generic 1-dimensional vector spaces.

\begin{prop}[Submanifolds of Line Bundles]\label{SubManLineBundle}
Let an embedded manifold $i:S\hookrightarrow M$ and a line bundle $\lambda: L \to M$, then restriction to the submanifold gives a canonical embedding factor
\begin{equation}
\begin{tikzcd}
L_S \arrow[r, "\iota"] \arrow[d] & L \arrow[d] \\
S \arrow[r,hook, "i"'] & M
\end{tikzcd}
\end{equation}
where $L_S:=i^*L$.
\end{prop}
\begin{proof}
This follows trivially by construction of pull-back bundle.
\end{proof}

Let us now define the analogue of the Cartesian product of manifolds in the category $\Line_\Man$. Consider two line bundles $\lambda_i:L_i\to M_i$, $i=1,2$, we will use the notations $L_i$ and $L_{M_i}$ indistinctly. We begin by defining the set of all linear invertible maps between fibres:
\begin{equation}
    M_1 \dtimes M_2:=\{B_{x_1x_2}:L_{x_1}\to L_{x_2}, \text{ factor }, (x_1,x_2)\in M_1\times M_2\},
\end{equation}
we call this set the \textbf{base product} of the line bundles. Let us denote by $p_i:M_1\dtimes M_2 \to M_i$ the obvious projections.

\begin{prop}[Base Product of Line Bundles]\label{BaseProductLineBundles}
The base product $M_1\dtimes M_2$ is a smooth manifold, furthermore, the natural $\Real^\times$-action given by fibre-wise multiplication makes $M_1\dtimes M_2$ into a principal bundle
\begin{equation}
    \begin{tikzcd}[column sep=0.1em]
        \Real^\times & \Acts & M_1\dtimes M_2  \arrow[d,"p_1\times p_2"]\\
         & & M_1\times M_2.
    \end{tikzcd}
\end{equation}
\end{prop}
\begin{proof}
This is shown by taking trivialisations $L|_{U_i}\cong U_i\times \Real$, $i=1,2$ that give charts of the form $U_1\times U_2\times \Real^\times$ for the open neighbourhoods $(p_1\times p_2)^{-1}(U_1 \times U_2)\subset M_1\dtimes M_2$. The cocycle condition of the transition functions of the two line bundles gives the local triviality condition for the $\Real^\times$-action to define a principal bundle.
\end{proof}

The construction of the base product of two line bundles $M_1\dtimes M_2$ allows to identify a factor $B:L_1\to L_2$ covering a smooth map $\varphi:M_1\to M_2$ with a submanifold that we may regard as the line bundle analogue of a graph:
\begin{equation}
    \LGrph{B}:=\{C_{x_1x_2}:L_{x_1}\to L_{x_2}|\quad x_2=\varphi(x_1), \quad C_{x_1x_2}=B_{x_1}\} \subset M_1\dtimes M_2,
\end{equation}
we call this submanifold the \textbf{lgraph} of the factor $B$.\newline

We define the \textbf{line product} of the line bundles as $L_1\utimes L_2 :=p_1^*L_1$, which is a line bundle over the base product 
\begin{equation}
    \lambda_{12}:L_1\utimes L_2\to M_1\dtimes M_2.
\end{equation}

\begin{prop}[Product of Line Bundles, \cite{schnitzer2019weakly}]\label{ProductLineBundle}
The line product construction $\utimes$ is a well-defined categorical product
\begin{equation}
    \normalfont\utimes:\Line_\Man \times \Line_\Man \to \Line_\Man
\end{equation}
with the corresponding projection factors fitting in the commutative diagram:
\begin{equation}\label{LineProductCommutativeDiagram}
\begin{tikzcd}
L_1 \arrow[d, "\lambda_1"'] & L_1\utimes L_2 \arrow[l,"P_1"']\arrow[d, "\lambda_{12}"]\arrow[r,"P_2"] & L_2\arrow[d,"\lambda_2"] \\
M_1 & M_1 \dtimes M_2\arrow[l,"p_1"]\arrow[r,"p_2"'] & M_2
\end{tikzcd}
\end{equation}
\end{prop}
\begin{proof}
See \cite[Prop. 2.2.4]{zapata2020unitfree}.
\end{proof}

If we consider a single line bundle $\lambda:L\to M$ the construction of the base product $M\dtimes M$ coincides with the usual definition of the \textbf{general linear groupoid} $p_1,p_2:\text{GL}(L)\to M$ with the projections $p_1$ and $p_2$ acting as the source and target maps respectively. This is clearly a transitive groupoid, with single orbit $M$, and the isotropies are $\text{GL}(L_x)$ which form a bundle of Lie groups isomorphic to the frame bundle $L^\times:= L^*\backslash \{0\}$. The set of diffeomorphic factors from $L$ to itself is called the \textbf{group of automorphisms} $\text{Aut}(L)$ and it is a simple check to verify that they correspond to the bisections of the general linear groupoid $\text{GL}(L)$. These objects represent the intrinsic symmetries of line bundles; indeed, the unit-free analogue of how the groups of diffeomorphisms capture the intrinsic symmetries of ordinary manifolds.\newline

Having shown that the category of line bundles is a direct generalisation of the category of smooth manifolds in terms of geometric constructions, we now move on to discuss how sections of line bundles represent a unit-free generalisation of the functions on a manifold. The facts discussed in the remainder of this section should convey the notion that modules of sections are the algebraic counterpart of line bundles in the same way that rings of functions are the algebraic counterpart of manifolds.\newline

Let a $B:L_1\to L_1$ be a factor covering the smooth map $\varphi:M_1\to M_2$, we can define the \textbf{factor pull-back} on sections as:
\begin{align}
B^*: \Sec{L_2} & \to \Sec{L_1}\\
s_2 & \mapsto B^*s_2
\end{align}
where
\begin{equation}
    B^*s_2(x_1):=B^{-1}_{x_1}s_2(\varphi(x_1))
\end{equation}
for all $x_1\in M_1$. Note that the fibre-wise invertibility of $B$ has critically been used for this to be a well-defined map of sections. The $\Real$-Linearity of $B$ implies the following interaction of factor pull-backs with the module structures:
\begin{equation}
    B^*(f\cdot s)=b^*f\cdot B^*s \quad \forall f\in\Cin{M_2}, s\in \Sec{L_{M_2}}.
\end{equation}
Then, if we consider $\textsf{RMod}$, the category of modules over rings with module morphisms covering ring morphisms, the assignment of sections becomes a contravariant functor
\begin{equation}
    \Gamma:\Line_\Man \to \textsf{RMod}.
\end{equation}
This is in contrast with the similar situation for sections of general vector bundles, since pull-backs of sections are not always defined for non-regular vector bundle morphisms.\newline

We may give an algebraic characterization of submanifolds of the base space of a line bundle as follows: for an embedded submanifold $i:S\hookrightarrow M$ we define its \textbf{vanishing submodule} as
\begin{equation}
    \Gamma_S:=\Ker{\iota^*}=\{s\in\Sec{L_M}|\quad s(x)=0\in L_x \quad \forall x\in S\}.
\end{equation}
This is the line bundle analogue to the characterization of submanifolds with their vanishing ideals. In fact, the two notions are closely related since (depending on the embedding $i$, perhaps only locally) we have
\begin{equation}
    \Gamma_S=I_S\cdot \Sec{L_M},
\end{equation}
where $I_S:=\Ker{i^*}$ is the multiplicative ideal of functions vanishing on the submanifold. This ideal gives a natural isomorphism of rings $\Cin{S}\cong \Cin{M}/I_S$ which, in turn, gives the natural isomorphism of $\Cin{S}$-modules
\begin{equation}
    \Sec{L_S}\cong \Sec{L}/\Gamma_S.
\end{equation}

The construction of the line product as a pull-back bundle over the base product indicates that the sections $\Sec{L_1\utimes L_2}$ are spanned by the sections of each factor, $\Sec{L_1}$, $\Sec{L_2}$, and the functions on the base product. More precisely, we have the isomorphisms of $\Cin{M_1\dtimes M_2}$-modules:
\begin{equation}
    \Cin{M_1\dtimes M_2}\cdot P_1^*\Sec{L_1}\cong \Sec{L_1 \utimes L_2} \cong \Cin{M_1\dtimes M_2}\cdot P_2^*\Sec{L_2}.
\end{equation}

Continuing with the theme of ``smearing'' linear spaces over a manifold, we define a \textbf{lvector bundle} $\epsilon:E^L\to M$ as a pair $(E,L)$ with $E$ a vector bundle and $L$ a line bundle over the same base $M$. The fibres are indeed identified with objects in the category of lvector spaces $E_x^{L_x}\in\LVect$ for all $x\in M$. A \textbf{morphism of lvector bundles} $F^B:E_1^{L_1}\to E_2^{L_2}$ is defined as a pair $(F,B)$ with $F:E_1\to E_2$ a vector bundle morphism and $B:L_1\to L_2$ a factor of lines covering the same smooth map between the bases. Clearly, the fibre-wise components of a morphism of lvector bundles are linear factors. We have thus identified the \textbf{the category of lvector bundles}, which will be denoted by $\LVect_\Man$.\newline

Lvector bundles appear naturally when derivations and jets, generalising vector fields and forms to the context of line bundles, are considered. Recall that the \textbf{derivations} $\Dr{L}$ of a line bundle $L$ are $\Real$-linear maps that interact with he module structure via the following Leibniz identity:
\begin{equation}
    D(f\cdot s)=f\cdot D(s) +X[f]\cdot s
\end{equation}
for $s\in \Sec{L}$ and $f\in\Cin{M}$. It is easy to see that derivations are the sections of an underlying vector bundle $\Der L$ that becomes a Lie algebroid with the commutator Lie bracket, this is called the \textbf{der bundle}. It turns out that the der bundle construction is functorial thus gives the unit-free generalization of the tangent functor:
\begin{equation}
    \Der : \Line_\Man \to \LVect_\Man.
\end{equation}
Furthermore, again in direct analogy with the ordinary tangent functor, the der construction is natural with respect to restriction to submanifolds and it preserves categorical products in the sense that there is an isomorphism of lvector bundles:
\begin{equation}
    \Der (L_1\utimes L_2)\cong \Der L_1 \boxplus \Der L_2.
\end{equation}
Similar to how differentials are local linear approximations of functions on ordinary manifolds, the unit-free counterpart of this idea is encapsulated in the standard definition of \textbf{jets} of sections: let a line bundle $\lambda: L\to M$, a point $x\in M$ and a neighbourhood $U\subset M$, the jet of a local section $s\in \Sec{L|_U}$ is the equivalence class of local sections:
\begin{equation}
    j_x^1s:=\{r\in \Sec{L|_U}\quad | \quad s(x)=r(x), \qquad \Tan_xs=\Tan_xr\}.
\end{equation}
Following standard constructions, it is easy to see that the \textbf{jet bundle} $\Jet^1 L$ corresponds to the ldual of the der bundle
\begin{equation}
    \Jet^1 L:=(\Der L)^*\otimes L = (\Der L)^{*L}.
\end{equation}
By construction, the assignment of the jet to a section gives a differential operator
\begin{equation}
    j^1:\Sec{L}\to \Sec{\Jet^1 L}
\end{equation}
called the \textbf{jet map}. The action of a section of the der bundle $a\in\Sec{\Der L}$ on a section $s\in\Sec{L}$ as a local derivation can now be rewritten as
\begin{equation}
    a[s]=j^1s(a)
\end{equation}
thus we see that the jet map $j^1$ gives the unit-free generalization of the exterior derivative in ordinary manifolds.

\subsection{Dimensioned Structures} \label{dimStruc}

The dimensioned formalism is motivated by the partial nature of the additive operation in the dimensional analysis of physical quantities; see \cite{zapata2021dimensioned} for a detailed discussion. Our goal is to define a theory of commutative algebra that incorporates sets where the additive operation is partially defined.\newline

A \textbf{dimensioned set} is simply a surjection of sets $\delta:A\to D$. We call $D$ the \textbf{set of dimensions}, $\delta$ the \textbf{dimensionality projection} and the preimages $A_d:=\delta^{-1}(d)\subset A$ the \textbf{dimension slices}. A \textbf{morphism of dimensioned sets} or \textbf{dimensioned map} $\Phi_\varphi:A_D\to B_E$ is simply a morphism of surjections, that is, a commutative diagram:
\begin{equation}
\begin{tikzcd}
A \arrow[r, "\Phi"] \arrow[d, "\delta"'] & B \arrow[d, "\epsilon"] \\
D \arrow[r, "\varphi"'] & E
\end{tikzcd}
\end{equation}
The \textbf{category of dimensioned sets} is denoted by $\textsf{DimSet}$. This category is entirely analogous to the ordinary category of sets \textsf{Set} with the obvious Cartesian monoidal structure.\newline

A \textbf{dimensional binar} structure $(A_D,*_D)$ is a partially-defined binary operation on $A$ with
\begin{equation}
    a *_d b \, \text{ defined only when } \, \delta(a)=\delta(a * b)=\delta(b)=d.
\end{equation}
This means that in the expression $a_d*_db_d$ all subscripts must agree for the product to be defined. In other words, a dimensional binar is a collection of ordinary binars indexed by the set of dimensions $\{(A_d,*_d), \, d\in D\}$. A \textbf{morphism of dimensional binars} $\Phi_\varphi:(A_D,*_D)\to (B_E,\circ_E)$ is a dimensioned map $\Phi_\varphi$ such that
\begin{equation}
    \forall \, d\in D, \, a,b\in A_d\qquad \Phi(a*_d b)=\Phi(a)\circ_{\varphi(d)} \Phi(b).
\end{equation}
The \textbf{category of dimensional binars} is denoted by $\textsf{DimBin}$ and the Cartesian monoidal structure extending from $\textsf{DimSet}$ in an obvious way.
Interestingly, in direct analogy with ordinary sets and binars, the sets of morphisms of dimensional binars carry natural dimensional binar structure: given two dimensioned maps $\Phi_\varphi, \Psi_\psi\in \textsf{DimSet}(A_D,B_E)$, for all $a_d\in A_d$ define
\begin{equation}
    \Phi_\varphi \circ \Psi_\psi (a_d):=\Phi(a)_{\varphi(d)} \circ_e \Psi(a)_{\psi(d)}
\end{equation}
which is indeed only possible when $\varphi(d)=e=\psi(d)$. Hence the point-wise operation so defined will endow the set of dimensioned maps with a dimensional binar structure:
\begin{equation}
    (\textsf{DimSet}(A_D,B_E)_{\textsf{Set}(D,E)},\,\circ_{\textsf{Set}(D,E)}).
\end{equation}

The theory of dimensional binars can be extended trivially to incorporate familiar algebraic notions: consider a dimensional binar $(A_D,*_D)$ and some abstract property of binars $P$ (e.g. associativity, commutativity, existence of identity, etc.), we will say that $(A_D,*_D)$ \textbf{satisfies property} $P$ simply when the binar slices $(A_d,*_d)$ satisfy the property $P$ for all $d\in D$. The theory of \textbf{dimensional groups} then follows naturally. All the familiar notions extend to the dimensioned context with the obvious caveats, for instance the identity element is now the identity subset of the identity elements of each slice. We will be focusing on \textbf{dimensional abelian groups} $(A_D,+_D)$. There is a natural notion of \textbf{product} of two dimensional groups $A_D$, $B_E$ given by the categorical product of dimensional binars $(A_D\times B_E, +_{D\times E})$. Furthermore, when we fix a dimension set $D$ and we consider \textbf{dimension-preserving morphisms}, i.e. dimensional group morphisms $\Phi:A_D\to B_D$ for which the induced map on the dimension sets is the identity $\Id_D:D\to D$, the dimensional abelian dimensional groups over $D$ form a subcategory $\textsf{DimAb}_D\subset\textsf{DimAb}$ that, in addition to the notions of subgroup, kernel and quotient, also admits a \textbf{direct sum} defined as $A_D\oplus_D B_D:=(A\times B)_D$ with partial multiplication given in the obvious way
\begin{equation}
    (a_d,b_d)+_d (a_d',b_d'):=(a_d+_d a_d',b_d +_d b_d').
\end{equation}
It is easy to prove that this direct sum operation on $\textsf{DimAb}_D$ acts as a product and coproduct for which the notions of kernel and quotient identified in the general category $\textsf{DimAb}$ satisfy the axioms of an abelian category. We call $\textsf{DimAb}_D$ the \textbf{category of $D$-dimensional abelian groups}. The dimension structure of a dimensioned set $\delta: S\to D$ allows for an obvious generalization of the free abelian group construction: the \textbf{dimensional free abelian group} on a dimensioned set $S_D$ is defined as the dimensional group
\begin{equation}
    \Int[S_D]:=\bigcup_{d\in D} \Int[A_d]
\end{equation}
where $\Int[\,\,]$ denotes the usual free abelian group construction. The dimensioned version of the free abelian group can be applied to the cartesian product of dimensional abelian groups to obtain the natural generalization of the tensor product for dimensional abelian groups: let two dimensional abelian groups $(A_D,+_D)$ and $(B_E,+_D)$, their \textbf{tensor product} is defined simply as follows:
\begin{equation}
    A_D\otimes B_E := \bigcup_{(d,e)\in D\times E} A_d\otimes B_e
\end{equation}
where $\otimes$ denotes the ordinary tensor product of abelian groups.\newline

A \textbf{dimensioned binar} structure $(A_D,*^D)$ is a totally-defined binary operation $*$ on $A$ with
\begin{equation}
    \forall \, d,d'\in D, \, \exists!\, d''\in D \qquad \text{such that} \qquad A_d* A_{d'}\subset A_{d''}.
\end{equation}
This condition is equivalent to the set of dimensions carrying a binar structure $(D,\,)$ (denoted by juxtaposition) and the dimension projection being a morphism of binars $\delta: (A,*)\to (D,\,)$. We thus write $a_d*b_e=(a*b)_{de}$, where the juxtaposition $de$ denotes de binar structure of the set of dimensions. A \textbf{morphism of dimensioned binars} between $(A_D,*^D)$ and $(B_E,\circ^E)$ is simply a dimensioned map $\Phi_\varphi:A_D\to B_E$ such that $\Phi:(A,*)\to (B,\circ)$ is a morphism of binars, since the binary operations are totally-defined. Note that this condition on $\Phi_\varphi$ makes $\varphi: (D,\,) \to (E,\,)$ into a morphism of binars.\newline

A \textbf{dimensioned ring} $(R_D,+_D,\cdot^D)$ is simply a dimensional abelian group $(R_D,+_D)$ with a dimensioned (totally-defined) monoid multiplication $(R_D,\cdot^D)$ such that the distributivity axiom:
\begin{equation}
    (a+b)\cdot c = a\cdot c + b\cdot c \qquad c\cdot (a+b) = c\cdot a + c\cdot b
\end{equation}
holds whenever it is defined for $a,b,c\in R$. The dimension set $D$ of a dimensioned ring $(R_D,+_D,\cdot^D)$ carries an associative unital binary operation $(D,\, )$, denoted by juxtaposition, such that $\delta: (R,\cdot)\to (D,\, )$ is a monoid morphism. Dimensioned rings will be assumed to be \textbf{commutative} unless otherwise stated. A \textbf{dimensioned field} is a dimensioned ring whose non-zero elements have multiplicative inverses. Let $(R_D,+_D,\cdot^D)$ and $(P_E,+_E,\cdot^E)$ be two dimensioned rings, a dimensioned map $\Phi:R_D\to P_E$ is called a \textbf{morphism of dimensioned rings} when
\begin{equation}
    \Phi(a\cdot b)=\Phi(a)\cdot \Phi(b), \qquad \Phi(1_R)=1_P
\end{equation}
for all $a,b\in R_D$. The map between the dimension monoids $\phi:D\to E$ is thus necessarily a monoid morphism. Dimensioned rings with these morphisms form the \textbf{category of dimensioned rings}, denoted by \textsf{DimRing}. A \textbf{unit} or \textbf{choice of units} $u$ in a dimensioned ring $R_D$ is a section of the dimension projection
\begin{equation}
\begin{tikzcd}[row sep=small]
R  \arrow[d, "\delta"'] \\
D \arrow[u, "u"', bend right=60] 
\end{tikzcd}\quad \delta \circ u =\Id_D,
\qquad \text{ such that } \quad
    u_{de}=u_d\cdot u_e \quad \text{and} \quad u_d\neq 0_d
\end{equation}
for all $d,e\in D$. In other words, a unit is a splitting $u:(D,\,)\to (R,\cdot)$ of the monoid surjection $\delta: (R,\cdot)\to (D,\,)$ with non-zero image. Units can be regarded as the dimensioned generalization of the notion of non-zero element of a ring with the caveat that they may not exist due to the non-vanishing condition being required for all of $D$. A dimensioned subring $I\subset R_D$ is called a \textbf{dimensioned ideal} if for all elements $a_d\in R_D$ and $i_e\in I$ we have
\begin{equation}
    a_d\cdot i_e\in I\cap R_{de}.
\end{equation}
Note that the zero $0_D\subset R$ is an ideal since the dimensioned ring axioms imply that it acts as an absorbent set in the following sense:
\begin{equation}
    0_d\cdot a_e=0_{de}.
\end{equation}

\begin{prop}[Quotient Dimensioned Ring] \label{QuotDimRing}
Let $(R_D,+_D,\cdot^D)$ be a dimensioned ring and and $I\subset R_D$ a dimensioned ideal, then the quotient dimensioned group $R/I$ carries a canonical dimensioned ring structure such that the projection map:
\begin{equation}
    q: R\to R/I
\end{equation}
is a morphism of dimensioned rings. This construction is called the \textbf{quotient dimensioned ring}.
\end{prop}
\begin{proof}
Let us denote dimension slices of the ideal by $I_d:= I\cap R_d$. From the construction of quotient dimensioned group we see that the projection map $q: R\to R/I$ is explicitly given by
\begin{equation}
    a_d\mapsto a_d + I_d,
\end{equation}
which makes $R/I$ into a dimensioned abelian group with dimension set $\delta (I)\subset D$. The dimensioned ring multiplication on the quotient can be explicitly defined by:
\begin{equation}
    (a_d+_dI_d)\cdot (b_e+_eI_e)=a_d\cdot b_e +_{de} a_d\cdot I_e +_{de} b_e\cdot I_d +_{de} I_d\cdot I_e=a_d\cdot b_e +_{de} I_{de}.
\end{equation}
This is easily checked to be well-defined and to inherit all the dimensioned ring multiplication properties from $R_D$. The map $q$ is then a morphism of dimensioned rings by construction. Note that the quotient ring has dimension projection $\delta': R/I \to \delta(I)$ thus, in particular, $\delta(I)\subset D$ is a submonoid.
\end{proof}

\begin{prop}[Dimensionless Ring] \label{DimLessRing}
Let $(R_D,+_D,\cdot^D)$ be a dimensioned ring, then its dimensionless slice $R_1$ carries a natural ring structure induced from the partial addition defined on the slice $+_1$ and the restriction of the total multiplication $\cdot|_{R_1}$. This is called the \textbf{dimensionless ring} $(R_1,+_1,\cdot|_{R_1})$. Furthermore, morphisms, products and quotients of dimensioned rings induce their analogous counterparts for dimensionless rings.
\end{prop}
\begin{proof}
The fact that $(R_1,+_1,\cdot|_{R_1})$ is a ring stems simply from the fact that $R_1$ is closed under multiplication, since $1\in D$ is the monoid identity. From a similar reasoning we see that for any dimensioned ideal $I\subset R_D$ the dimensionless slice $I_1:= I\cap R_1$ is an ideal of the dimensionless ring. Let two dimensioned rings $(R_D,+_D,\cdot^D)$ and $(P_E,+_E,\cdot^E)$. Since a morphism $\Phi: R_D \to P_E$ preserves the multiplicative identities, it is clear that the restriction
\begin{equation}
    \Phi|_{R_1}:(R_1,+_1,\cdot|_{R_1})\to (P_{\phi(1)},+_{\phi(1)},\cdot|_{P_{\phi(1)}})
\end{equation}
is a morphism of rings. The dimensionless slice of the dimensioned product $R_D \times P_E$ is indeed $R_1\times P_1$ with the direct product construction applying to rings in a straightforward manner.
\end{proof}
This shows that dimensioned rings are, in fact, a strict generalization of ordinary rings since we recover them by considering trivial dimension monoids, i.e. singleton dimension sets. More precisely, the category of rings is a subcategory of the category of dimensioned rings $\textsf{Ring}\subset \textsf{DimRing}$.\newline

Once the theory of dimensioned rings is established, the notions of dimensioned modules and algebras follow easily. We summarise the main definitions in the remainder of this section, for more details see \cite[Sec. 5,6]{zapata2021dimensioned}. A dimensional abelian group $(A_D,+_D)$ is called a \textbf{dimensioned module} over the dimensioned ring $R_G$ if there is a map 
\begin{equation}
    \cdot :R_G\times A_D \to A_D
\end{equation}
that is compatible with the dimensioned structures via a monoid action $G\times D\to D$ (denoted by juxtaposition) in the following sense
\begin{equation}
    r_g\cdot a_d=(r \cdot a)_{gd}
\end{equation}
and that satisfies the following axioms
\begin{itemize}
    \item[1)] $r_g\cdot (a_d+b_d)=r_g\cdot a_d + r_g\cdot b_d$,
    \item[2)] $(r_g+p_g)\cdot a_d=r_g\cdot a_d + p_g \cdot a_d$,
    \item[3)] $(r_g p_h)\cdot a_d=r_g\cdot (p_h\cdot a_d)$,
    \item[4)] $1\cdot a_d=a_d$
\end{itemize}
for all $r_g,p_h\in R_G$ and $a_d,b_d\in A_D$. All the familiar notions of ordinary modules extend to the dimensioned context provided the obvious caveats. Particularly, the \textbf{tensor product} now requires some extra considerations on the monoid structures of the sets of dimensions: we construct the tensor product of the underlying dimensional abelian groups and quotient by the dimensioned ring action
\begin{equation}
    A_D\otimes_{R_G} B_E := (A_D \otimes B_E)/\sim_{R_G}
\end{equation}
where $\sim_{R_G}$ is defined in the obvious way:
\begin{equation}\label{tensorquorel}
    (r_g\cdot a_d,b_e)\sim_{R_G} (a_d,r_g\cdot b_e)
\end{equation}
for all $r_g\in R_G$, $a_d\in A_D$ and $b_e\in B_E$. By construction, $A_D\otimes_{R_G} B_E$ is a dimensional abelian group with dimension set given by a peculiar quotient of the product of dimension sets:
\begin{equation}
    D \times^G E := (D \times E)/\sim_G
\end{equation}
where $\sim_G$ is induced from $\sim_{R_G}$ above in the obvious way:
\begin{equation}
    (g d,e)\sim_G (d,g e),
\end{equation}
for all $g\in G$, $d\in D$, $e\in E$, making use of the monoid actions. It is also possible to define quotients of dimensioned modules.

\begin{prop}[Quotient Dimensioned Module] \label{QuotMod}
Let $(A_D,+_D)$ be a dimensioned $R_G$-module and $I\subset R_G$ an ideal, then, if $S\subset A_D$ is a submodule such that $I\cdot A\subset S$, the quotient $A_D/S$ inherits a dimensioned $R_G/I$-module structure such that the projection map
\begin{equation}
    Q: A_D\to A_D/S
\end{equation}
is a $q$-linear map, where $q:R_G\to R_G/I$ is the quotient ring projection.
\end{prop}
\begin{proof}
Since quotients of rings and modules are taken as dimensional abelian groups, this result follows from a simple computation showing the distributivity property of the ideal submodule $S$: let $r_g\in R_G$, $i_g\in I$, $a_d\in A_D$ and $s_d\in S$ then
\begin{equation}
    (r_g+i_g)\cdot (a_d+s_d) = r_g \cdot a_d + r_g \cdot s_d + i_g \cdot a_d + i_g \cdot s_d.
\end{equation}
The second and fourth terms are in $S$ from the fact that $S$ is a submodule and the third term is in $S$ from the ideal submodule condition $I\cdot A\subset S$, then the above expression defines the $R_G/I$-module structure on $A_D/S$ which, by construction, satisfies $Q(r_g\cdot a_d)=q(r_g)\cdot Q(a_d)$.
\end{proof}

Let $(A_D,+_D)$ be a dimensioned $R_G$-module, a map $M:A_D\times A_D\to A_D$ is called a \textbf{dimensioned bilinear multiplication} if it satisfies
\begin{align}
    M(a_d \, +_d\, b_d,c_e)&=M(a_d,c_e) \,+_{\mu(d,e)} \, M(b_d,c_e)\\
    M(a_d,b_e \, +_e \, c_e)&=M(a_d,b_e) \, +_{\mu(d,e)} \, M(a_d,c_e)\\
    M(r_g\cdot a_d,s_h\cdot b_e)&=r_g\cdot s_h\cdot M(a_d, b_e)
\end{align}
for all $a_d,b_d,b_e,c_e\in A_D$, $r_g,s_h\in R_G$ and for a \textbf{dimension map} $\mu:D\times D\to D$ which is $G$-equivariant in both entries, i.e.
\begin{equation}
    \mu(gd,he)=gh\mu(d,e)
\end{equation}
for all $g,h\in G$ and $d,e\in D$. When such a map $M$ is present in a dimensioned $R_G$-module $A_D$, the pair $(A_D,M)$ is called a \textbf{dimensioned $R_G$-algebra}. Note that the tensor product of dimensioned modules allows to reformulate the definition of a dimensioned bilinear multiplication as a dimensioned $R_G$-linear morphism
\begin{equation}
    M:A_D\otimes_{R_G} A_D \to A_D.
\end{equation}
A dimensioned algebra multiplication $M:A_G\times A_G\to A_G$ is said to be \textbf{homogeneous of dimension $m$} if the dimension map $\mu:G\times G\to G$ is given by monoid multiplication with the element $m\in G$, i.e. $\mu(g,h)=mgh$ for all $g,h\in G$. Assuming a monoid structure on the dimension set of a dimensioned module and considering dimensioned algebra multiplications of homogeneous dimension is particularly useful in order to study several algebra multiplications coexisting on the same set. Indeed, given two homogeneous dimensioned algebra multiplications $(A_G,M_1)$ and $(A_G,M_2)$ with dimensions $m_1\in G$ and $m_2\in G$, respectively, the fact that the monoid operation is assumed to be associative and commutative, allows us to postulate algebraic identities involving expressions of the form $M_1(M_2(a,b),c)$ without any further requirements. Examples of such properties are \textbf{symmetry}, \textbf{antisymmetry}, \textbf{associativity} and \textbf{Jacobi}, which in the usual combinations give the obvious definitions of \textbf{dimensioned commutative algebras} and \textbf{dimensioned Lie algebras}. Similar to the case of ordinary Lie algebras, we find the natural notion of \textbf{Casimir element} in a dimensioned Lie algebra as a distinguished element of the center.\newline

Note that a generic (commutative) dimensioned ring $(R_G,+_G,\cdot^G)$ provides an example of a dimensioned commutative algebra itself: the dimensional abelian group structure is the same $(R_G,+_G)$ and the module structure is given by multiplication with the dimensionless subring $R_1\subset R_G$, then the ring multiplication becomes clearly bilinear (from associativity and commutativity) and it has definite dimension $1\in G$ as a dimensioned algebra multiplication. Note that the $R_1$-module structure has dimension monoid action the trivial action of $\{1\}$ on $G$, thus we see that $R_G$ as a dimensioned commutative algebra has an ordinary multiplicative module structure over an ordinary ring.\newline

Another important source of examples of dimensioned algebras connected to dimensioned rings are derivations. By analogy with the case of ordinary ring derivations a \textbf{dimensioned ring derivation} is a dimensioned map $\Delta\in \Dim{R_G}$ covering a dimension map $\delta:G\to G$ satisfying the Leibniz identity with respect to the dimensioned ring multiplication
\begin{equation}
    \Delta(r_g\cdot s_h)= \Delta(r_g)\cdot s_h + r_g\cdot\Delta(s_h),
\end{equation}
for all $r_g,s_h\in R_G$, however, for the right-hand-side to be well-defined, both terms must be of homogeneous dimension, which means that the dimension map must satisfy
\begin{equation}
    \delta(gh)=\delta(g)h=g\delta(h)
\end{equation}
for all $g,h\in G$. Since $G$ is a monoid, this condition is equivalent to the dimension map being given by multiplication with a monoid element, i.e. $\delta=d$ for some element $d\in G$. It follows from this observation that there is a natural dimensioned submodule of the dimensioned module of dimensioned maps $\Dim{R_G}_G\subset \Dim{R_G}_{\text{Map(G)}}$. Recall that dimensioned rings are assumed to be commutative and, thus, the dimension monoid has a commutative binary operation. This allows for the definition the commutator of the associative dimensioned composition:
\begin{equation}\label{dimcomm}
    [\Delta,\Delta']:=\Delta \circ \Delta'-\Delta' \circ \Delta,
\end{equation}
is easily checked to be antisymmetric and Jacobi, thus making $(\Dim{R_G}_G,[,])$ into the \textbf{dimensioned Lie algebra of dimensioned maps} of a dimensioned ring $R_G$. Notice that this commutator bracket can only be defined on the dimensioned submodule $\Dim{R_G}_G\subset \Dim{R_G}_{\text{Map(G)}}$. This motivates the definition of the dimensioned Lie algebra of \textbf{derivations of a dimensioned ring} $R_G$ as the natural dimensioned Lie subalgebra of the dimensioned maps:
\begin{equation}
    \Dr{R_G}\subset (\Dim{R_G}_G,[,]).
\end{equation}
Derivations covering the identity dimension map $\Id_G:G\to G$ are called \textbf{dimensionless derivations} and it is clear by definition that they form an ordinary Lie algebra with the commutator bracket $(\Dr{R_G}_{\Id_G}, [,])$.

\subsection{Dimensioned Poisson Algebras} \label{dimPoisson}

Let us present here the definition and main results of the dimensioned analogue of Poisson algebras. Let $A_G$ be a dimensioned $R_H$-module and let two dimensioned algebra multiplications $*:A_G\times A_G\to A_G$ and $\{\,,\}:A_G\times A_G\to A_G$ with homogeneous dimensions $p\in G$ and $b\in G$, respectively, the triple $(A_G,*_p,\{\,,\}_b)$ is called a \textbf{dimensioned Poisson algebra} if
\begin{itemize}
    \item[1)] $(A_G,*_p)$ is a dimensioned commutative algebra,
    \item[2)] $(A_G,\{\,,\}_b)$ is a dimensioned Lie algebra,
    \item[3)] the two multiplications interact via the Leibniz identity
    \begin{equation}
        \{a,b*c\}=\{a,b\}*c+b*\{a,c\},
    \end{equation}
    for all $a,b,c\in A_G$.
\end{itemize}
Note that the Leibniz condition can be consistently demanded of the two dimensioned algebra multiplications since the dimension projections of each of the terms of the Leibniz identity for $\{a_g,b_h*c_k\}$ are: 
\begin{equation}
    bgphk,\quad pbghk,\quad phbgk,
\end{equation}
which are all indeed equal from the fact that the monoid binary operation is associative and commutative.\newline

A morphism of dimensioned modules between dimensioned Poisson algebras $\Phi:(A_G,*_p,\{\,,\}_b)\to (B_H,*_r,\{\,,\}_c)$ is called a \textbf{morphism of dimensioned Poisson algebras} if $\Phi:(A_G,*_p)\to (B_H,*_r)$ is a morphism of dimensioned commutative algebras and also $\Phi:(A_G,\{\,,\}_b)\to (B_H,\{\,,\}_c)$ is a morphism of dimensioned Lie algebras. A submodule $I\subset A_G$ that is a dimensioned ideal in $(A_G,*_p)$ and that is a dimensioned Lie subalgebra in $(A_G,\{\,,\}_b)$ is called a \textbf{dimensioned coisotrope}. The category of dimensioned Poisson algebras is denoted by $\textsf{DimPoissAlg}$.\newline

Dimensioned Poisson algebras admit the all-important constructions of reductions and products, which turn out to be direct generalizations of their ordinary counterparts for Poisson algebras.

\begin{prop}[Dimensioned Poisson Reduction] \label{DimensionedPoissonReduction}
Let $(A_G,*_p,\{\,,\}_b)$ be a dimensioned Poisson algebra and $I\subset A_G$ a coisotrope, then there is a dimensioned Poisson algebra structure induced in the subquotient
\begin{equation}
    (A'_G:=N(I)/I,*'_p,\{\,,\}'_b)
\end{equation}
where $N(I)$ denotes the dimensioned Lie idealizer of $I$ regarded as a submodule of the dimensioned Lie algebra.
\end{prop}
\begin{proof}
We assume without loss of generality that the dimension projection of $I$ is the whole of $G$, the intersections with the dimension slices are denoted by $I_g:=I\cap A_g$. The dimensioned Lie idealizer is defined in the obvious way:
\begin{equation}
    N(I):=\{n_g\in A_G|\quad \{n_g,i_h\}\in I_{bgh} \quad \forall i_h\in I\}.
\end{equation}
Note that $N(I)$ is the smallest dimensioned Lie subalgebra that contains $I$ as a dimensioned Lie ideal. The Leibniz identity implies that $N(I)$ is a dimensioned commutative subalgebra with respect to $*_p$ in which $I$ sits as a dimensioned commutative ideal, since it is a commutative ideal in the whole $A_G$. Following a construction analogous to the quotient dimensioned ring of Proposition \ref{QuotDimRing} we can form the dimensioned quotient commutative algebra $(N(I)/I,*')$ whose underlying module is the quotient dimensioned module of Proposition \ref{QuotMod}. Note that the only difference with the case of the quotient dimensioned ring construction is that the commutative multiplication $*_p$ covers a dimension map that is given by the monoid multiplication with a non-identity element $p\in G$, but this has no effect on the quotient construction itself. To obtain the desired quotient dimensioned Lie bracket we set:
\begin{equation}
    \{n_g+I_g,m_h+I_h\}':=\{n_g,m_h\}+I_{bgh}
\end{equation}
which is easily checked to be well-defined and that inherits the antisymmetry and Jacobi properties directly from dimensioned Lie bracket $\{\,,\}$ and the fact that $I\subset N(I)$ is a dimensioned Lie ideal.
\end{proof}

\begin{prop}[Heterogeneous Dimensioned Poisson Product] \label{DimensionedPoissonProductHetero}
Let $(A_G,*_g,\{\,,\}_g)$ and $(B_H,*_h,\{\,,\}_h)$ be dimensioned Poisson algebras over a dimensioned ring $R_P$, then the tensor product $A_G \otimes_{R_P} B_H$ carries a natural Poisson structure defined by linearly extending the following commutative product and Lie bracket:
\begin{align}
    (a\otimes b) *_{gh} (a' \otimes b) &:= a *_g a' \otimes b *_h b'\\
    \{a \otimes b, a' \otimes b'\}_{gh} &:= \{ a, a'\}_g \otimes b *_h b' + a*_g a'\otimes \{b,b'\}_h.
\end{align}
\end{prop}
\begin{proof}
From an algebraic point of view, this construction is entirely analogous to the ordinary product of Poisson algebras (see for instance \cite{fernandes2014lectures}), which is in fact explicitly recovered for trivial dimension monoids $G=H=\{1\}$. Then, it suffices to show that the two proposed expressions for are well-defined dimensioned algebras. First note that the tensor product of dimensioned monoids carries a natural monoid structure defined by:
\begin{equation}
    (g \otimes_P h ) \cdot (g' \otimes_P h') := gg' \otimes_P hh'.
\end{equation}
Taking the dimension projection of the commutative product we obtain:
\begin{equation}
    \delta(a_x *_g a_y' \otimes_{R_P} b_u *_h b_v') = gxy \otimes_P huv = ( g \otimes_P h ) \cdot (x \otimes_P u) \cdot (y \otimes_P v),
\end{equation}
hence $*_{gh}$ is a well-defined commutative multiplication of dimension $g \otimes_P h\in G \times^P H$. For the bracket $\{\,,\,\}_{gh}$ to be well-defined both terms need to have equal dimension projection, this is indeed the case since the dimension of the bracket and commutative products are assumed to be equal:
\begin{equation}
    \delta (\{ a_x, a_y'\}_g \otimes b_u *_h b_v') = (g \otimes_P h)\cdot (x \otimes_P u) \cdot (y \otimes_P v) = \delta (a_x*_g a_y'\otimes \{b_u,b_v'\}_h).
\end{equation}
\end{proof}

\begin{prop}[Homogeneous Dimensioned Poisson Product] \label{DimensionedPoissonProductHomo}
Let $(A_G,*_p,\{\,,\}_b)$ and $(B_G,*_q,\{\,,\}_c)$ be dimensioned Poisson algebras over a dimensioned ring $R_G$ such that $bq=pc$, then the tensor product $A_G \otimes_{R_G} B_G$ carries a natural Poisson structure defined by linearly extending the following commutative product and Lie bracket:
\begin{align}
    (a\otimes b) *_{pq} (a' \otimes b) &:= a *_p a' \otimes b *_q b'\\
    \{a \otimes b, a' \otimes b'\}_{pc} &:= \{ a, a'\}_b \otimes b *_q b' + a*_p a'\otimes \{b,b'\}_c.
\end{align}
\end{prop}
\begin{proof}
We proceed analogously to the proof of Proposition \ref{DimensionedPoissonProductHetero} above with the added simplicity that now all dimension sets are $G$ and the element-wise tensor product is given simply by monoid multiplication since $G\times^G G \cong G$. Taking the dimension projection of the commutative product we obtain:
\begin{equation}
    \delta(a_x *_p a_y' \otimes_{R_G} b_u *_q b_v') = pxy \otimes_G quv = pq xyuv = pq \delta(a_x \otimes_{R_G} b_u ) \delta(a_y' \otimes_{R_G} b_v'),
\end{equation}
hence $*_{pq}$ is a well-defined commutative multiplication of dimension $pq\in G $. For the bracket $\{\,,\,\}_{pc}$ to be well-defined both terms need to have equal dimension projection, this is indeed the case from the condition $bq=pc$ assumed of the pair of dimensioned Poisson algebras:
\begin{equation}
    \delta (\{ a_x, a_y'\}_b \otimes b_u *_q b_v') = bq xyuv = pc xyuv =  \delta (a_x*_p a_y'\otimes \{b_u,b_v'\}_c).
\end{equation}
\end{proof}

\subsection{The Power Functor} \label{PowFunct}

In this section we show how the lines and line-vector spaces of Section \ref{Lines} can be regarded as special cases of the dimensioned structures of Section \ref{dimStruc}. In doing so we will see how these notions lead to a natural dimensioned generalisation of vector spaces. This shall pave the way for our main results that associate dimensioned algebras to line bundles in Section \ref{LinesToDim}.\newline

The cornerstone of the correspondence between the unit-free and the dimensioned formalisms is the so-called \textbf{power functor} construction that we detail in what follows. Consider a line $L\in \Line$ and denote:
\begin{equation}
    \begin{cases} 
      L^n:=\otimes^nL & n>0 \\
      L^n:= \Real & n=0 \\
      L^n:=\otimes^nL^* & n<0 
   \end{cases}
\end{equation}
such that given two integers $n,m\in\Int$ the following equations hold
\begin{equation}
    (L^n)^* = L^{-n} \qquad L^{n}\otimes L^{m} = L^{n+m}.
\end{equation}
We define the \textbf{power} of a line $L\in\Line$ as the set of all tensor powers
\begin{equation}
    L^\odot:=\bigcup_{n\in \Int} L^n.
\end{equation}
This set has than an obvious dimension structure with dimension set $\Int$:
\begin{equation}
    \pi:L^\odot\to \Int.
\end{equation}
Since dimension slices are precisely the tensor powers $L^n$, they carry a natural $\Real$-vector space structure, thus making the power of $L$ into a dimensional abelian group $(L^\odot_\Int,+_\Int)$. The next proposition shows that the ordinary $\Real$-tensor product of vector spaces endows $L^\odot$ with a dimensioned field structure.

\begin{prop}[Dimensioned Ring Structure of the Power of a Line] \label{DimRingPower}
Let $\normalfont L\in\Line$ be a line and $(L^\odot_\Int,+_\Int)$ its power, then the $\Real$-tensor product of elements induces a dimensioned multiplication
\begin{equation}
    \odot: L^\odot \times L^\odot\to L^\odot
\end{equation}
such that $(L^\odot_\Int,+_\Int,\odot)$ becomes a dimensioned field.
\end{prop}
\begin{proof}
The construction of the dimensioned ring multiplication $\odot$ is done simply via the ordinary tensor product of ordinary vectors and taking advantage of the particular properties of 1-dimensional vector spaces. The two main facts that follow from the 1-dimensional nature of lines are: firstly, that linear endomorphisms are simply multiplications by field elements
\begin{equation}
    \text{End}(L)\cong L^*\otimes L\cong \Real
\end{equation}
which, at the level of elements, means that
\begin{equation}
    \text{End}(L)\ni\alpha \otimes a =\alpha(a)\cdot \Id_{L}
\end{equation}
as it can be easily shown by choosing a basis; and secondly, that the tensor product becomes canonically commutative, since, using the isomorphism above, we can directly check
\begin{equation}
    a\otimes b(\alpha,\beta)=\alpha(a)\beta(b)=\alpha(b)\beta(a)=b\otimes a(\alpha,\beta),
\end{equation}
thus showing
\begin{equation}
    a\otimes b=b\otimes a \in L\otimes L=L^2.
\end{equation}
The binary operation $\odot$ is then explicitly defined for elements $a,b\in L=L^1$, $\alpha,\beta\in L^*=L^{-1}$ and $r,s\in \Real=L^0$ by
\begin{align}
    a\odot b &:= a\otimes b\\
    \alpha\odot \beta &:= \alpha \otimes \beta\\
    r\odot s &:= r\otimes s=rs\\
    r\odot a &:= ra\\
    r\odot \alpha &:= r\alpha\\
    \alpha \odot a &:=\alpha(a) = a(\alpha) =: a \odot \alpha
\end{align}
Products of two positive power tensors $a_1\otimes \cdots \otimes a_q$, $b_1\otimes \cdots \otimes b_p$ and negative powers $\alpha_1\otimes \cdots \otimes \alpha_q$, $\beta_1\otimes \cdots \otimes \beta_p$ are defined by
\begin{align}
    (a_1\otimes \cdots \otimes a_q)\odot(b_1\otimes \cdots \otimes b_p) &:= a_1\otimes \cdots \otimes a_q\otimes b_1\otimes \cdots \otimes b_p\\
    (\alpha_1\otimes \cdots \otimes \alpha_q)\odot (\beta_1\otimes \cdots \otimes \beta_p) &:= \alpha_1\otimes \cdots \otimes \alpha_q \otimes \beta_1\otimes \cdots \otimes \beta_p
\end{align}
and extending by $\Real$-linearity. Let $q,p>0$, the dimensioned ring product satisfies:
\begin{equation}
    \odot: L^q \times L^p \to L^{q+p}, \qquad \odot: L^{-q} \times L^{-p} \to L^{-q-p}, \qquad \odot: L^0 \times L^0 \to L^0.
\end{equation}
For products combining positive power tensors $a_1\otimes \cdots \otimes a_q$ and negative power tensors $\alpha_1\otimes \cdots \otimes \alpha_p$ we critically make use of the isomorphism $L^*\otimes L\cong \Real$ to define without loss of generality:
\begin{equation}
    (a_1\otimes \cdots \otimes a_q) \odot (\alpha_1\otimes \cdots \otimes \alpha_p) := \alpha_1( a_1) \cdots \alpha_q( a_q) \alpha_{p-q}\otimes \cdots \otimes \alpha_p
\end{equation}
where $p>q>0$. It is then clear that the multiplication $\odot$ satisfies, for all $m,n\in \Int$,
\begin{equation}
    \odot: L^m \times L^n \to L^{m+n}
\end{equation}
and so it is compatible with the dimensioned structure of $L^\odot_\Int$. The multiplication $\odot$ is clearly associative and bilinear with respect to addition on each dimension slice from the fact that the ordinary tensor product is associative and $\Real$-bilinear. Then it follows that $(L^\odot_\Int,+_\Int,\odot)$ is a commutative dimensioned ring. It only remains to show that non-zero elements of $L^\odot$ have multiplicative inverses. Note that a non-zero element corresponds to some non-vanishing tensor $0\neq h\in L^n$, but, since $L^n$ is a 1-dimensional vector space for all $n\in \Int$, we can find a unique $\eta\in (L^n)^*=L^{-n}$ such that $\eta(h)=1$. It follows from the above formula for products of positive and negative tensor powers that, in terms of the dimensioned ring multiplication, this becomes
\begin{equation}
    h\odot \eta =1,
\end{equation}
thus showing that all non-zero elements have multiplicative inverses, making the dimensioned ring $(L^\odot_\Int,+_\Int,\odot)$ into a dimensioned field.
\end{proof}

The construction of the power dimensioned field of a line is, in fact, functorial.

\begin{prop}[The Power Functor for Lines] \label{PowerIsAFunctorLine}
The assignment of the power construction to a line is a functor
\begin{equation}
\normalfont
    \odot: \Line \to \textsf{DimRing}.
\end{equation}
Furthermore, a choice of unit in a line $L\in \Line$ induces a choice of units in the dimensioned field $(L^\odot_\Int,+_\Int,\odot)$ which, since $L^0=\Real$, then gives an isomorphism with the product dimensioned field
\begin{equation}
    L^\odot \cong \Real \times \Int.
\end{equation}
\end{prop}
\begin{proof}
To show functoriality we need to define the power of a factor of lines $B:L_1\to L_2$
\begin{equation}
    B^\odot:L_1^\odot \to L_2^\odot.
\end{equation}
This can be done explicitly in the obvious way, for $q>0$
\begin{align}
    B^\odot|_{L^q} &:= B\otimes \stackrel{q}{\cdots} \otimes B :L_1^q\to L_2^q\\
    B^\odot|_{L^0} &:= \Id_{\Real}:L_1^0\to L_2^0\\
    B^\odot|_{L^{-q}} &:= (B^{-1})^*\otimes \stackrel{q}{\cdots} \otimes (B^{-1})^* :L_1^{-q}\to L_2^{-q}
\end{align}
where we have crucially used the invertibility of the factor $B$. By construction, $B^\odot$ is compatible with the $\Int$-dimensioned structure and since $B$ is a linear map with linear inverse, all the tensor powers act as $\Real$-linear maps on the dimension slices, thus making $B^\odot:L_1^\odot \to L_2^\odot$ into a morphism of abelian dimensioned groups. Showing that $B^\odot$ is a dimensioned ring morphism follows easily by the explicit construction of the dimensioned ring multiplication $\odot$ given in proposition \ref{DimRingPower} above. This is checked directly for products that do not mix positive and negative tensor powers and for mixed products it suffices to note that
\begin{equation}
    B^\odot (\alpha)\odot B^\odot (a)=(B^{-1})^*(\alpha) \odot B(a)= \alpha (B^{-1}(B(a)))= \alpha(a) = \Id_{\Real}(\alpha(a)) =B^\odot (\alpha\odot a).
\end{equation}
It follows from the usual properties of tensor products in vector spaces that for another factor $C:L_2\to L_3$ we have
\begin{equation}
    (C\circ B)^\odot=C^\odot \circ B^\odot, \qquad (\Id_L)^\odot=\Id_{L^\odot},
\end{equation}
thus making the power assignment into a functor. Recall that a choice of unit in a line $L\in \Line$ is simply a choice of non-vanishing element $u\in L^\bullet$. In proposition \ref{DimRingPower} we saw that $L^\odot$ is a dimensioned field, so multiplicative inverses exist, let us denote them by $u^{-1}\in (L^*)^\bullet$. Using the notation for $q>0$
\begin{align}
    u^q &:= u\odot \stackrel{q}{\cdots} \odot u\\
    u^0 &:= 1\\
    u^{-q} &:= u^{-1}\odot \stackrel{q}{\cdots} \odot u^{-1},
\end{align}
it is clear that the map
\begin{align}
u: \Int & \to L^\odot\\
n & \mapsto u^n
\end{align}
satisfies
\begin{equation}
    u^{n+m}=u^n\odot u^m.
\end{equation}
By construction, all $u^n\in L^n$ are non-zero, so $u:\Int \to L^\odot$ is a choice of units in the dimensioned field $(L^\odot_\Int,+_\Int,\odot)$. Noting that $(L^\odot)_0=L^0=\Real$, the isomorphism $L^\odot \cong \Real \times \Int$ follows from proposition \ref{DimLessRing}.
\end{proof}

The power construction naturally extends to line-vector spaces. Since lines and vector spaces belong to the same category of generic vector spaces, we define the \textbf{power dimensioned vector space} as:
\begin{equation}
    V^{\odot L}:=\bigcup_{n\in \Int} L^n\otimes V.
\end{equation}
The ordinary linear structure of tensor powers gives a natural dimensional abelian group structure $(V^{\odot L}_\Int,+_\Int)$. The monoidal structure of the category of vector spaces allows us to define a natural dimensioned $L^\odot$-module multiplication:
\begin{equation}
    r_n\cdot (s_m\otimes v) := r_n \odot s_m \otimes v
\end{equation}
where $v\in V$ and $r_n,s_m\in L^\odot$ and whose distributivity property follows directly from the dimensioned ring structure of $L^\odot$. Note that on the dimensionless slice $V^{\odot L}_0=\Real \otimes V\cong V$ the dimensioned $L^\odot$-module multiplication restricts to the ordinary $\Real$-vector space multiplication. At the level of dimensions, the dimensioned module multiplication appears as the monoid action of $\Int$ on itself by addition. The power construction for lvector spaces turns out to be functorial and to generalise the power functor of lines.

\begin{prop}[The Power Functor for Line-Vector Spaces] \label{PowerIsAFunctorLVect}
The assignment of the power dimensioned vector space construction to a line-vector space
\begin{equation}
\normalfont
    \odot: \LVect \to \textsf{DimVect}
\end{equation}
is functorial.
\end{prop}
\begin{proof}
This result is a direct extension of Proposition \ref{PowerIsAFunctorLine} to line-vector spaces. The proof works analogously for the following assignment of morphisms: let a linear factor $\psi^B:V^L\to W^{L'}$ and define its power dimensioned morphism $\psi^{\odot B}:V^{\odot L}\to W^{\odot L'}$ as acting on a generic element
\begin{equation}
    \psi^{\odot B} (s\otimes v) := B^\odot(s) \otimes \psi (v)
\end{equation}
for all $s\in L^\odot$ and $v\in V$. This map is clearly additive, since it is slice-wise $\Real$-linear, and it satisfies the following multiplicativity condition
\begin{equation}
    \psi^{\odot B} (r\cdot (s\otimes v)) = B^\odot (r \odot s) \otimes \psi(v) = B^\odot (r) \odot B^\odot(s) \otimes \psi(v) = B^\odot (r) \cdot \psi^{\odot B}(s\otimes v)
\end{equation}
which is the general property of (twisted) dimensioned module morphisms generalising linearity.
\end{proof}

The power dimensioned field construction is not just limited to a single line but can also be defined for any ordered family of lines: given the ordered set of lines $L_1,\dots,L_k\in\Line$, we define their \textbf{power dimensioned field} as:
\begin{equation}
    (L_1,\dots,L_k)^\odot:=\bigcup_{n_1,\dots n_k\in \Int} L_1^{n_1}\otimes \cdots \otimes L_k^{n_k},
\end{equation}
which has a natural dimensional abelian group structure given by $\Real$-linear addition and has dimension group $\Int^k$. The dimensioned multiplicative structure is now defined via:
\begin{equation}
    (a_1\otimes \dots \otimes a_k) \odot (b_1\otimes \dots \otimes b_k):= a_1\odot b_1 \otimes \dots \otimes a_k\odot b_k
\end{equation}
which is clearly associative and commutative thus making $((L_1,\dots,L_k)^\odot_{\Int^k},+_{\Int^k},\odot)$ into a dimensioned field. The braiding isomorphisms present in the category of lines and the fact that zeroth tensor powers of lines are copies of $\Real$ allows to define inclusions of partial powers:
\begin{equation}
    L_i^{n_i}\ni s_i\mapsto 1 \otimes \cdots \otimes s_i \otimes \cdots 1\in L_1^0\otimes \cdots L_i^{n_i} \otimes \cdots \otimes L_k^0
\end{equation}
These maps are dimensioned ring homomorphisms by construction so they lead to a collection of natural injections of the power rings of partial powers, e.g.
\begin{align}
    L_i^\odot &\hookrightarrow (L_1,\dots,L_k)^\odot\\
    (L_i\otimes L_j)^\odot &\hookrightarrow (L_1,\dots,L_k)^\odot
\end{align}
Note that these are proper inclusions since the dimension set is $\Int$ for the dimensioned rings on the left but $\Int^k$ for the one on the right.\newline

Recall that dimensioned rings can be regarded as dimensioned commutative algebras over the (ordinary) ring of dimensionless elements. This is now readily apparent from the fact that powers of lines are defined from $\Real$-vector spaces and the fact that they carry a copy of $\Real$ as the ring of dimensionless elements. It is then possible to take tensor products of power rings, since they share the same dimensionless ring $\Real$, and it is straightforward to show that the power of a collection of lines corresponds to the tensor product of their powers regarded as dimensioned commutative $\Real$-algebras: for any two lines $L_1,L_2\in\Line$ there is a dimensioned commutative algebra isomorphism
\begin{equation}
    (L_1,L_2)^\odot \cong L_1^\odot \otimes_\Real L_2^\odot.
\end{equation}
Note that the dimension groups of both sides are $\Int^2$ from the fact that $\Real$ is regarded as dimensioned ring with trivial monoid ${0}$ as dimension set and so the tensor product on the right projects to the Cartesian product of dimension groups.\newline

When we consider power dimensioned vector spaces, however, the definition of tensor product of dimensioned modules in Section \ref{dimStruc} forces us to consider products of pairs of line-vector spaces with the same line component. In fact, given two power dimensioned vector spaces $V^{\odot L}$ and $W^{\odot L}$, the tensor product as dimensioned modules recovers the tensor product of ordinary vector spaces as we can easily check:
\begin{equation}
    V^{\odot L} \otimes_{L^\odot} W^{\odot L} \cong (V\otimes W)^{\odot L}.
\end{equation}
We thus see how the L-rooted view introduced in Section \ref{Lines}, a somewhat unmotivated approach to generalise standard linear algebra notions in which tensor products of lines were explicitly avoided, is now vindicated as the manifestation of the fact that tensor products of dimensioned modules can only be defined for a common underlying dimensioned ring of scalars -- just like in the ordinary theory of modules.\newline

\section{Dimensioned Rings from Line Bundles} \label{LinesToDim}

In this section we extend the dimensioned algebra constructions defined at the level of linear algebra in Section \ref{PowFunct} to the context of differential geometry. The results in this section shall be interpreted as the dimensioned-algebraic approach to smooth manifolds and line bundles.

\subsection{The Power Functor for Line Bundles} \label{PowFunctorLineBundle}

Let $\lambda:L\to M$ be a line bundle over a smooth manifold. Following a construction entirely analogous to the power of a line, we can define the \textbf{power} of the line bundle $L$ as:
\begin{equation}
    \Sec{L}^\odot:=\bigcup_{n\in \Int} \Sec{L^n}.
\end{equation}
This set carries an obvious dimensioned structure with dimension set $\Int$ and the usual module structure on sections for each power $(\Sec{L^n},+_n)$ clearly makes $\Sec{L}^\odot$ into an abelian dimensioned group. Furthermore, the construction of the dimensioned ring product $\odot$ in Proposition \ref{DimRingPower} can be reproduced mutatis mutandis for the modules of sections, thus making the power of a line bundle into a dimensioned ring $(\Sec{L}^\odot_\Int,+_\Int,\odot)$. Note that this dimensioned ring encapsulates the usual algebraic structures found in sections of line bundles: indeed, the dimensionless ring of $\Sec{L}^\odot$ is the ordinary ring of functions of the base manifold $\Sec{L^0}=\Sec{\Real_M}\cong\Cin{M}$ and, for $f\in\Sec{L^0}=\Cin{M}$, $s\in\Sec{L^1}=\Sec{L}$ and $\sigma\in\Sec{L^{-1}}=\Sec{L^*}$, the dimensioned products
\begin{equation}
    f \odot s = f\cdot s, \qquad f \odot \sigma = f\cdot \sigma, \qquad \sigma \odot s = \sigma(s)
\end{equation}
amount to the $\Cin{M}$-module maps and the duality pairing. We now show that, as was the case for lines, the power construction of line bundles is functorial.

\begin{prop}[The Power Functor for Line Bundles] \label{PowerFunctorLineBundles}
The assignment of the power construction to a line bundle is a contravariant functor
\begin{equation}
\normalfont
    \odot : \Line_\Man \to \textsf{DimRing}.
\end{equation}
\end{prop}
\begin{proof}
Let us first define the power of a factor between line bundles $B:L_1\to L_2$ covering a smooth map $b:M_1\to M_2$. We aim to define a dimensioned ring morphism of the form
\begin{equation}
    B^\odot: (\Sec{L_2}^\odot_\Int,+_\Int,\odot)\to (\Sec{L_1}^\odot_\Int,+_\Int,\odot),
\end{equation}
that is, furthermore, a dimensionless morphism, in the sense that it will cover the identity on the dimension group $\Id_\Int:\Int\to \Int$. It suffices to provide a collection of maps between the sections of all the tensor powers $B^\odot_n:\Sec{L_2^n}\to \Sec{L_1^n}$. The datum provided by the line bundle factor $B$ allows to define three maps
\begin{align}
    b^* &:\Cin{M_2}\to \Cin{M_1}\\
    B^* &:\Sec{L_2}\to \Sec{L_1}\\
    B^* &:\Sec{L_2^*} \to \Sec{L_1^*}
\end{align}
where the first is simply the pull-back of the smooth map between base manifolds, the second is the pull-back of sections induced by a factor of line bundles defined point wise by
\begin{equation}
    B^*(s_2)(x):=B_x^{-1}(s_2(b(x)))
\end{equation}
for all $s_2\in\Sec{L_2}$, and the third is the usual pull-back of dual forms on general vector bundles, defined point-wise for a general bundle map by
\begin{equation}
    B^*\sigma_2(s_1)(x):=\sigma_2(b(x))(B_x(s_1(x)))
\end{equation}
for all $\sigma_2\in \Sec{L_2^*}$, $s_1\in\Sec{L_1}$. The maps $B^\odot_n$ are then defined simply as the tensor powers of these pull-backs. Contravariance then follows directly from contravariance of the pull-backs. It is then clear by construction that $B^\odot$ so defined acts as a dimensioned ring morphism for products of positive or negative tensor powers, then it only remains to show that it also acts as such for mixed products of tensor powers. This is readily checked from the observation that for any two sections $s_2\in\Sec{L_2}$ and $\sigma_2\in\Sec{L_2^*}$ we have:
\begin{align}
    B^\odot\sigma_2\odot B^\odot s_2(x) &=B^*\sigma_2\odot B^*s_2(x)\\
    &=\sigma_2(b(x))(B_xB_x^{-1}(s_2(b(x))))\\
    &=\sigma_2(b(x))(s_2(b(x)))\\
    &=b^*(\sigma_2(s_2))(x)=B^\odot(\sigma_2\odot s_2)(x).
\end{align}
Functoriality thus follows directly from the basic properties of factor pull-backs.
\end{proof}

This last proposition provides the key result that legitimizes the interpretation of line bundles as unit-free manifolds since we notice the similarity of the power functor above with the ordinary contravariant functor given by the assignment of the ring of smooth functions to a manifold
\begin{equation}
    \text{C}^\infty:\Man \to \Ring.
\end{equation}
The power functor $\odot$ is, in fact, a direct generalization of $\text{C}^\infty$ since ordinary rings of functions can be recovered as from trivialised line bundles whose power ring carries the singleton as dimension set.\newline

Due to possible topological constraints, the notion of \textbf{unit} of a line, i.e. a non-vanishing element, can be recovered only locally in line bundles. Let $\lambda:L\to M$ be a line bundle and $U\subset M$ an open subset, the power construction is clearly natural with respect to restrictions since the same prescription used for global sections can be used to define $\Sec{L|_U}^\odot$. Defining the positive and negative powers of $u$ as it was done for the line case, it is clear that a local unit induces a choice of units for the local power
\begin{equation}
    u:\Int \to \Sec{L|_U}^\odot.
\end{equation}
It then follows from the second part of Proposition \ref{PowerIsAFunctorLine} that a local unit $u$ induces an isomorphism of the local power with the trivial dimensioned ring of local functions with dimension set $\Int$:
\begin{equation}
    \Sec{L|_U}^\odot\cong \Cin{U}\times \Int.
\end{equation}

To further establish the dimensioned formalism as the appropriate technology to describe the algebraic counterpart to line bundles, we shall give a dimensioned-algebraic account of submanifolds and products. Consider a submanifold $i:S\hookrightarrow M$ of a line bundle $\lambda:L\to M$. We saw in Section \ref{LineBundles} that a line bundle is induced on $S$ by pull-back with inclusion factor $\iota:L_S\to L$ covering the embedding. There, the set of vanishing sections on $S$, defined formally as the kernel of $\iota$, was shown to be a submodule of the sections of the ambient line bundle $\Gamma_S\subset \Sec{L}$ that can be seen as (locally) generated by the ideal of vanishing functions $I_S\subset\Cin{M}$. The following proposition shows that these two algebraic manifestations of a submanifold in a line bundle fit nicely into the dimensioned picture.

\begin{prop}[Vanishing Dimensioned Ideal of a Submanifold] \label{DimVanishingIdeal}
A submanifold $i:S\hookrightarrow M$ in a line bundle $\lambda:L\to M$ defines a dimensioned ideal $I_S\subset \Sec{L}^\odot$ that allows to characterize (perhaps only locally) the restricted power as a quotient of dimensioned rings
\begin{equation}
    \Sec{L_S}^\odot\cong \Sec{L}^\odot/I_S.
\end{equation}
\end{prop}
\begin{proof}
The vanishing  dimensioned ideal is simply defined as the set of sections of all the tensor powers that vanish when restricted to $S$, that is
\begin{equation}
    I_S:=\{a\in \Sec{L^n}|\quad a(x)=0\in L^n_x \quad \forall x\in S\}.
\end{equation}
Note that this is equivalent to the kernel of the power of the inclusion factor $I_S=\Ker{\iota^\odot}$, the dimensioned ideal condition $\Sec{L}^\odot\odot I_S\subset I_S$ then follows:
\begin{equation}
    \iota^\odot(r_n\odot a_m)= (\iota^*)^n r_n \odot (\iota^*)^m a_m=(\iota^*)^n r_n \odot 0_m=0_{n+m}
\end{equation}
for all $r_n\in \Sec{L^n}$, $a_m\in I_S$, where the functoriality of the power construction in Proposition \ref{PowerFunctorLineBundles} has been used. It follows by construction that the ordinary ideal of vanishing functions is the dimensionless component of $I_S$ and that
\begin{equation}
    \Gamma_S=I_S\cap \Sec{L^1}.
\end{equation}
Possibly restricting to a local neighbourhood, we can see that, similarly to the submodule of vanishing sections, the subsets of homogeneous dimension of $I_S$ can all be generated by elements in the dimensionless component. The quotient $\Sec{L}^\odot/I_S$ in Proposition \ref{QuotDimRing} identifies sections of the tensor powers of $L_S$ with extensions in $L$ that differ by a vanishing section of the corresponding tensor power, thus giving the desired result.
\end{proof}

Recall from the ordinary theory of manifolds that for any two manifolds $M_1$ and $M_2$ there are natural inclusions of the rings of functions $\Cin{M_1}$ and $\Cin{M_2}$ into the ring of functions of the product $\Cin{M_1 \times M_2}$. The following propositions show that the power dimensioned ring construction allows to recover the analogous result for the product of line bundles $L_1 \utimes L_2$ defined in Section \ref{LineBundles} as the direct analogue of the Cartesian product of manifolds.

\begin{prop}[Power Dimensioned Ring of Line Bundle Products I] \label{PowerLineBundleProduct1}
Let $\lambda_1:L_1\to M_1$ and $\lambda_2:L_2\to M_2$ be line bundles, then there is a map of dimensioned commutative $\Real$-algebras
\begin{equation}
    \Sec{L_1}^\odot \otimes_\Real \Sec{L_2}^\odot \to \Sec{L_1 \utimes L_2}^\odot.
\end{equation}
\end{prop}
\begin{proof}
This result follows from the line bundle product construction. Recall that there are canonical projections
\begin{equation}
\begin{tikzcd}
L_1  & L_1\utimes L_2 \arrow[l,"P_1"']\arrow[r,"P_2"] & L_2.
\end{tikzcd}
\end{equation}
Since these are factors, i.e. line bundle morphisms, the above diagram becomes
\begin{equation}
\begin{tikzcd}[row sep=tiny]
\Sec{L_1}^\odot \arrow[r,"P_1^\odot"] & \Sec{L_1\utimes L_2}^\odot  & \Sec{L_2}^\odot\arrow[l,"P_2^\odot"']
\end{tikzcd}
\end{equation}
under the contravariant power functor for line bundles. Then the map is simply defined as:
\begin{equation}
    r_1 \otimes_\Real r_2 \mapsto P_1^\odot r_1 \odot P_2^\odot r_2.
\end{equation}
Since the tensor product $\Sec{L_1}^\odot \otimes_\Real \Sec{L_2}^\odot$ is taken as a dimensioned commutative algebras over the (dimensionless) ring $\Real$, it only remains to check that the product multiplication maps homomorphically into the power of the line bundle product. This is trivially ensured by functoriality of the power construction again:
\begin{equation}
    (r_1 \otimes_\Real r_2)\cdot (s_1 \otimes_\Real s_2) = r_1\odot s_1 \otimes_\Real r_2\odot s_2 \mapsto P_1^\odot r_1 \odot P_2^\odot r_2 \odot P_1^\odot s_1 \odot P_2^\odot s_2.
\end{equation}
\end{proof}

Note that, in contrast with the case of ordinary rings of functions in product manifolds, the map in the proposition above is not an injection. This is signaling the fact that the line product $L_1 \utimes L_2$, although a well-defined Cartesian categorical product in $\Line_\Man$, does not fully capture the algebraic information that is relevant for the power functor. This goes back to the comment we made in Proposition \ref{ProductLineBundle} about the asymmetry of the definition of line product. Note that the base product $M_1 \dtimes M_2$ is built so that the canonical projections give two line bundles over it:
\begin{equation}\label{LineProductCommutativeDiagram}
\begin{tikzcd}
L_1 \arrow[d] & p_1^*L_1, p_2^*L_2 \arrow[l]\arrow[d, shift right=3]\arrow[d, shift left=3]\arrow[r] & L_2\arrow[d] \\
M_1 &  M_1 \dtimes M_2\arrow[l,"p_1"]\arrow[r,"p_2"'] & M_2
\end{tikzcd}
\end{equation}
This invites us to consider collections of line bundles $L_1,\dots,L_k$ with the same base manifold $M$. By extending the definition of power of a collection of lines introduced at the end of Section \ref{PowFunct} into the context of line bundles we define the \textbf{power dimensioned ring} as:
\begin{equation}
    \Sec{L_1,\dots,L_k}^\odot := \bigcup_{n_1,\dots,n_k\in\Int} \Sec{L_1^{n_1}\otimes \cdots \otimes L_k^{n_k}}
\end{equation}
which, with the slice-wise $\Real$-linear addition $+$ and ordered tensor product of elements $\odot$, clearly becomes a dimensioned ring $( \Sec{L_1,\dots,L_k}^\odot_{\Int^k},+_{\Int^k},\odot^{\Int^k})$. It is then natural to define the image of the power functor on line products as
\begin{equation}
    \odot: L_1 \utimes L_2 \mapsto \Sec{p_1^*L_1,p_2^*L_2}^\odot.
\end{equation}
Note that the symmetry isomorphisms for lines $L_1 \utimes L_2 \cong L_2 \utimes L_1$ translate as the braiding isomorphisms for power dimensioned rings:
\begin{equation}
    \Sec{p_1^*L_1,p_2^*L_2}^\odot \cong \Sec{p_2^*L_2,p_1^*L_1}^\odot.
\end{equation}
When there is no risk of confusion we will omit the pull-back maps in the power ring notation. We are now in the position to recover an inclusion of power dimensioned rings.
\begin{prop}[Power Dimensioned Ring of Line Bundle Products II] \label{PowerLineBundleProduct2}
Let $\lambda_1:L_1\to M_1$ and $\lambda_2:L_2\to M_2$ be line bundles, then there is a natural inclusion of dimensioned rings
\begin{equation}
    \Sec{L_1}^\odot \otimes_\Real \Sec{L_2}^\odot \hookrightarrow \Sec{L_1, L_2}^\odot
\end{equation}
where the tensor product is taken as dimensioned commutative $\Real$-algebras.
\end{prop}
\begin{proof}
Note that the construction of the power dimensioned ring is such that:
\begin{equation}
    \Sec{L_1, L_2}^\odot \cong \Sec{p_1^*L_1}^\odot \otimes_{\Cin{M_1 \dtimes M_2}} \Sec{p_2^*L_2}^\odot,
\end{equation}
then the inclusion map is given similarly to the one in Proposition \ref{PowerLineBundleProduct1}
\begin{equation}
    r_1 \otimes_\Real r_2 \mapsto P_1^\odot r_1 \odot P_2^\odot r_2
\end{equation}
but now the $\odot$ multiplication is taken in the product power dimensioned ring, which makes the map injective and covering the identity dimension map $\Id:\Int^2\to \Int^2$.
\end{proof}

\subsection{Derivations and Jets} \label{DerJet}

The derivations of a dimensioned ring were introduced at the end of Section \ref{dimStruc} in direct analogy with ordinary derivations of rings. It follows from our remarks there that the derivations of the power dimensioned ring of a line bundle $\Sec{L}^\odot$, denoted by $\Dr{L^\odot}$ for short, contain the derivations of functions on the base manifold:
\begin{equation}
    \Dr{\Cin{M}}\cong\Sec{\Tan M}\subset \Dr{L^\odot}.
\end{equation}
In Section \ref{LineBundles} we argued that line bundle derivations were the unit-free generalization of sections of the tangent bundle, the next proposition shows that the derivations of the power of a line bundle naturally include the line bundle derivations.

\begin{prop}[Dimensionless Power Derivations] \label{DimensionlessPowerDerivations}
Let $\lambda:L\to M$ be a line bundle and $\Sec{L}^\odot$ its power, then there is an isomorphism of Lie algebras
\begin{equation}
\normalfont
    \Dr{L}\cong \Dr{L^\odot}_0.
\end{equation}
\end{prop}
\begin{proof}
We give the isomorphism by explicitly specifying the two invertible maps. Firstly we assign a dimensionless derivation to its restriction on the slices of dimension $0$ and $1$
\begin{equation}
    \Dr{L^\odot}_0\ni P\mapsto (P|_{L^0},P|_{L^1})=:(X,D)
\end{equation}
Since $P$ is a $\odot$-derivation, these satisfy
\begin{align}
    X(fg) &:=X(f\odot g)=X(f)\odot g+f\odot X(g)=X(f)g+fX(g)\\
    D(f\cdot s) &:= D(f\odot s)=X(f)\odot s+f\odot D(s)=X(f)\cdot s + f\cdot D(s),
\end{align}
thus showing that $D$ is a line bundle derivation with symbol $X$, as desired. Conversely, given a line bundle derivation $D\in \Dr{L}$ with symbol $X$, we need to define a dimensionless derivation $\Dr{L^\odot}_0$. This is accomplished by extending $D$ as a $\odot$-derivation for non-negative tensor powers of $L$ the following basic identities
\begin{align}
    P(f\odot g) &:= X[fg]=X[f]g+fX[g]\\
    P(f\odot s) &:= D(f\cdot s)=X[f]\cdot s + f\cdot D(s)\\
    P(s\odot r) &:= P(s)\odot r + s\odot P(s).
\end{align}
To account for negative tensor powers we use the isomorphism of Lie algebras $\Dr{L}\cong \Dr{L^*}$ given by assigning a derivation $D\in \Dr{L}$ the derivation $\Delta\in\Dr{L^*}$ defined by the Leibniz identity relative to the dual pairing:
\begin{equation}
    (\Delta \sigma) (s) := X[\sigma(s)]-\sigma(Ds)
\end{equation}
for all $s\in\Sec{L}$ and $\sigma\in\Sec{L^*}$. The derivation $\Delta$ is then similarly extended as a $\odot$-derivation for non-positive tensor powers. To complete the extension of $D$ to $P$ as a $\odot$-derivation, it only remains to account for products mixing positive and negative tensor powers. This follows from the consistency formula in the definition of the dimensioned ring product:
\begin{align}
    P(\sigma \odot s) &= P(\sigma(s))=X(\sigma(s))=\Delta(\sigma)(s)+\sigma(D(s))=\\
    &=\Delta(\sigma)\odot s + \sigma \odot D(s)=P(\sigma)\odot s+ \sigma \odot P(s).
\end{align}
\end{proof}
Dimensionless derivations do not determine all the derivations of a dimensioned ring; in general, derivations of the power of a line bundle are given by (infinite) collections of differential operators between all the tensor powers satisfying some compatibility conditions with the $\Int$-dimensioned structure. The derivations of the power dimensioned ring, naturally a dimensioned $\Sec{L}^\odot$-module, form a dimensioned Lie algebra with the commutator bracket and so we are compelled to regard the \textbf{dimensioned derivations} $(\Dr{L^\odot},[\,,])$ as the dimensioned analogue of vector fields. This analogy can be pushed further still so we can easily define the \textbf{dimensioned multiderivations} by constructing the obvious dimensioned $\Sec{L}^\odot$-modules of antisymmetric tensor powers $(\text{Der}^\bullet(L^\odot),\wedge)$ with a Schouten bracket defined via:
\begin{align}
    \llbracket D,D' \rrbracket &= [D,D']\\
    \llbracket D,r \rrbracket &= D(r)
\end{align}
for all $D\in\Dr{L^\odot}$ and $r\in\Sec{L}^\odot$. We call $(\text{Der}^\bullet(L^\odot),\wedge,\llbracket\,,\rrbracket)$ the \textbf{dimensioned Gerstenhaber algebra} of the line bundle $L$.\newline

Note that we are now in the position to develop the dimensioned analogue of \textbf{Cartan calculus} simply by dualising dimensioned derivations. We define \textbf{dimensioned 1-forms} as the $\Sec{L}^\odot$-dual of derivations:
\begin{equation}
    \Omega^1(L^\odot):=\left \{\Dr{L^\odot} \to \Sec{L^\odot}, \, \Sec{L}^\odot\text{-linear}\right \}
\end{equation}
which carry a natural $\Sec{L}^\odot$-module structure. The \textbf{dimensioned differential}:
\begin{equation}
    \delta: \Sec{L}^\odot \to \Omega^1(L^\odot)
\end{equation}
is defined in the standard way
\begin{equation}
    \delta r (D): = D(r)
\end{equation}
for all $r\in \Sec{L}^\odot$ and $D\in \Dr{L^\odot}$. By constructing the dimensioned $\Sec{L}^\odot$-modules of antisymmetric $k$-forms on derivations $\Omega^k(L^\odot)$ and extending $\delta$ as an exterior derivation in the obvious way, we obtain the \textbf{dimensioned de Rham complex}:
\begin{equation}
    (\Omega^\bullet(L^\odot),\wedge,\delta).
\end{equation}
The following proposition shows that, as was the case for dimensioned derivations, ordinary jets are recovered from dimensioned 1-forms.
\begin{prop}[Power Forms] \label{JetsPowerJets}
Let $\lambda:L\to M$ be a line bundle and $\Sec{L}^\odot$ its power, then there are natural inclusions of $\normalfont\Cin{M}$-modules
\begin{align}
\normalfont
    \Sec{(\Der L)^*}&\hookrightarrow \Omega^1(L^\odot)_{0}\\
    \normalfont \Sec{\Jet^1 L}&\hookrightarrow \Omega^1(L^\odot)_{+1}
\end{align}
Furthermore, the Lie algebroid de Rham complex $\normalfont \Omega^\bullet(\Der L)$ and the complex of Atiyah forms $\normalfont\Omega^\bullet(\Der L; L)$ are contained in the dimensioned de Rham complex:
\begin{equation}
    \normalfont\Omega^\bullet(\Der L) \subset \Omega^\bullet(L^\odot) \supset \Omega^\bullet(\Der L;L).
\end{equation}
\end{prop}
\begin{proof}
Note that restricting dimensioned jets to dimensionless derivations gives:
\begin{align}
    \Omega^1(L^\odot)_{0}|_{\Dr{L^\odot}_0} &=\{\Dr{L^\odot}_0\to \Cin{M}\}\\
    \Omega^1(L^\odot)_{+1}|_{\Dr{L^\odot}_0} &=\{\Dr{L^\odot}_0\to \Sec{L}\},
\end{align}
then, since Proposition \ref{DimensionlessPowerDerivations} ensures that $\Dr{L^\odot}_0\cong \Dr{L}\cong \Sec{\Der L}$, the inclusions follow trivially. A similar argument applies for $k$-forms which gives the inclusions:
\begin{align}
    \Omega^k(\Der L)&\hookrightarrow \Omega^k(L^\odot)_{0}\\
    \Omega^k(\Der L;L)&\hookrightarrow \Omega^k(L^\odot)_{+1}.
\end{align}
Since the wedge product is defined identically for all three complexes, it only remains to check that the dimensioned differential $\delta$ restricts to the correct differential maps on each slice. This follows by construction: firstly, since the anchor map of the der bundle regarded as a Lie algebroid $\rho: \Der L\to \Tan M$ is precisely the map that sends derivations to symbols we have
\begin{equation}
    \delta|_{\Omega^0(L^\odot)_{0}}=d_{\Der L}=\rho^* \circ d: \Cin{M}\to \Sec{\Der L}
\end{equation}
and, secondly, since the action of jets on derivations is given by the dual pairing we have
\begin{equation}
    \delta|_{\Omega^0(L^\odot)_{+1}}=d_L=j^1: \Sec{L}\to \Sec{\Jet^1 L},
\end{equation}
thus showing that the differentials agree at degree $0$. Finally, since $\delta$, $d_{\Der L}$ and $d_L$ are all extended as $\wedge$-derivations for higher degrees, their agreement at degree $0$ implies agreement at all degrees.
\end{proof}

The results presented in the last couple of sections firmly establish dimensioned rings and modules as the natural algebraic counterparts to line bundles. Propositions \ref{PowerFunctorLineBundles} to \ref{JetsPowerJets} indicate that the oddities of the categories typically associated with line bundles, what we summarised in Section \ref{LineBundles} as the unit-free formalism, disappear when promoting them to the dimensioned setting. Broadly speaking, by paying the price that additive structures are now partially defined -- something that could be regarded as a positive if one is interested in recovering the dimensional analysis of physical quantities -- all the notions of ordinary manifold theory, i.e. products, submanifolds, functions, vector fields, cohomology, etc., promote into the context of line bundles mutatis mutandis via their dimensioned analogues.

\section{Dimensioned Algebras from Jacobi Manifolds} \label{JacToDim}

In this section we show how a Jacobi structure on a line bundle is mapped to a dimensioned Poisson algebra under the power functor, thus generalising the assignment of Poisson algebras to the rings of functions of ordinary Poisson manifolds. The analogous constructions for Poisson products and reductions are also recovered in the dimensioned picture, a clear improvement from the standard formulation using ordinary modules of sections. We shall build on the approach to Jacobi manifolds developed in the prequel paper \cite{zapata2020unitfree}, of which we summarise the main points here for convenience.\newline

A \textbf{Jacobi manifold} is a line bundle $\lambda:L\to M$ whose sections carry a local Lie algebra structure $(\Sec{L},\{\,,\})$. This means that there are symbol maps $X$ and $\Lambda$ such that:
\begin{equation}
    \{f\cdot s, g\cdot r\} = fg\cdot\{s,r\}+f X_s[g]\cdot r - g X_r[f]\cdot s +\Lambda(df\otimes s, dg \otimes r)
\end{equation}
for all $s,r\in \Sec{L}$ and $f,g\in\Cin{M}$. The Jacobi identity of the Lie bracket implies some compatibility conditions between the symbol maps.
\begin{prop}[Symbols of a Jacobi Bracket]\label{JacobiSymbol}
Let $\lambda:L\to M$ be a line bundle and let $\Sigma\subset \Sec{L}$ be a subspace of spanning sections such that $\normalfont\Sec{L}\cong \Cin{M}\cdot \Sigma$. The datum of a Jacobi structure $(\Sec{L},\{,\})$ is equivalent to a triple $((\Sigma,[\,,]),X,\Lambda)$ where
\begin{itemize}
\normalfont
    \item $(\Sigma,[,])$ is a $\Real$-linear Lie bracket,
    \item $X:\Sec{L}\to \Sec{\Tan M}$ is a $\Real$-linear map and,
    \item $\Lambda\in\Sec{\wedge^2\Tan^{L} M\otimes L}$ inducing a vector bundle morphism $\Lambda^\sharp:\Tan^{*L} M\to \Tan M$
\end{itemize}
satisfying the compatibility conditions
\begin{align*}
    1.\quad & [X_s,X_{r}]=X_{[s,r]},\\
    2.\quad & X_{f\cdot s}=f\cdot X_s+\Lambda^\sharp(df\otimes s), \\
    3.\quad & [X_s,\Lambda^\sharp(df\otimes r)]=\Lambda^\sharp(dX_{s}[f]\otimes r + df\otimes [s,r]),\\
    4.\quad & \Lambda(df\otimes s, \Lambda(dg \otimes r , dh\otimes t))+\Lambda(dg\otimes r, \Lambda(dh \otimes t , df\otimes s))+\Lambda(dh\otimes t, \Lambda(df \otimes s , dg\otimes r))\\
    =&X_r[f]\cdot \Lambda( dh\otimes t, dg\otimes s)+X_t[g]\cdot \Lambda(df\otimes s, dh\otimes r)+X_s[h]\cdot \Lambda(dg\otimes r,df\otimes t),
\end{align*}
for all $\normalfont f,g,h\in\Cin{M}$ and $s,r,t\in\Sigma$.
\end{prop}
\begin{proof}
\cite[Prop. 3.1.1]{zapata2020unitfree}.
\end{proof}

The fact that the Lie bracket $\{\,,\}$ is a derivation on each argument allows us to write
\begin{equation}
    \{s,r\}=\Pi(j^1s,j^1r) \quad \text{ with }\quad  \Pi\in \text{Der}^2(L).
\end{equation}
The bilinear form $\Pi$, called the \textbf{Jacobi biderivation} in analogy with the the Poisson bivector in ordinary Poisson manifolds, naturally induces a musical map $\Pi^\sharp:\Jet^1 L \to \Der L$. Let $(\Sec{L_1},\{\,,\}_1)$ and $(\Sec{L_2},\{\,,\}_2)$ be two Jacobi manifolds, a line bundle morphism $B:L_1\to L_2$ is called a \textbf{Jacobi map} when its pull back
\begin{equation}
    B^*:(\Sec{L_2},\{\,,\}_2)\to (\Sec{L_1},\{\,,\}_1)
\end{equation}
is a Lie algebra morphism. We can define coisotropic submanifolds of Jacobi manifolds in direct analogy with coisotropic submanifolds of Poisson manifolds. A submanifold $i:S\hookrightarrow M$ is called \textbf{coisotropic} when the annihilator of its der bundle $(\Der L_S)^{0L}\subset \Jet^1L$ is an isotropic subbundle with respect to the biderivation $\Pi$, i.e.
\begin{equation}
    \Pi^\sharp_x((\Der_xL_S)^{0L})\subset \Der_xL_S \qquad \forall x\in S.
\end{equation}
Coisotropic submanifolds are of vital importance for the physical applications of Jacobi geometry since they provide the prime example of a reduction scheme. Let a line bundle $\lambda:L\to M$ with a Jacobi structure $(\Sec{L},\{,\})$ and consider a submanifold $i:C\hookrightarrow M$ with its corresponding embedding factor $\iota: L_C\to L$. For another line bundle $\lambda':L'\to M'$, assume there exists a submersion factor $\pi:L_C\to L'$ covering a surjective submersion $p:C\twoheadrightarrow M'$. Then, we say that a Jacobi structure $(\Sec{L},\{,\})$ \textbf{reduces} to the Jacobi structure $(\Sec{L'},\{\,\}')$ via $\pi:L_C\to L'$ when for all pairs of sections $s_1,s_2\in\Sec{L'}$ the identity
\begin{equation*}
    \pi^*\{s_1,s_2\}'=\iota^*\{S_1,S_2\}
\end{equation*}
holds for all choices of extensions $S_1,S_2$, i.e all choices of sections $S_1,S_2\in\Sec{L}$ satisfying 
\begin{equation*}
    \pi^*s_1=\iota^*S_1 \qquad \pi^*s_2=\iota^*S_2.
\end{equation*}
The following proposition shows that, in analogy with the case of reduction of Poisson manifolds, coistropic submanifolds play a distinguished role in the reduction of Jacobi manifolds.

\begin{prop}[Coisotropic Reduction of Jacobi Manifolds] \label{CoisotropicReductionJacobi}
Let $\lambda:L\to M$ and $\lambda':L'\to M'$ be line bundles,  $(\Sec{L},\{,\})$ be a Jacobi structure, $i:C\hookrightarrow M$ a closed coisotropic submanifold and $\pi:L_C\to L'$ a submersion factor covering a surjective submersion $p:C\twoheadrightarrow M'$ so that we have the reduction diagram:
\begin{equation*}
    \begin{tikzcd}[sep=tiny]
    L_C \arrow[rr,"\iota"] \arrow[dd,"\pi"'] \arrow[dr]& & L \arrow[dr] & \\
    & C \arrow[rr,"i"', hook] \arrow[dd, "p",twoheadrightarrow] & & M \\
    L' \arrow[dr] & & & \\
    & M' &  & 
    \end{tikzcd}
\end{equation*}
Assume the following compatibility condition between the coisotropic submanifold and the submersion factor:
\begin{equation*}
\normalfont
    \delta(\Ker{\Der \pi})=\Lambda^\sharp((\Tan C)^{\text{0}L}),
\end{equation*}
where $\normalfont\delta:\Der L \to \Tan M$ is the anchor of the der bundle, $\normalfont(\Tan C)^{0L}\subset \Tan^{*L}M$ is the annihilator and $\normalfont\Lambda^\sharp: \Tan^{*L}M\to \Tan M$ is the squiggle, then there exists a unique Jacobi structure $(\Sec{L'},\{,\}')$ such that $(\Sec{L},\{,\})$ reduces to it via the submersion factor $\pi:L_C\to L'$.
\end{prop}
\begin{proof}
\cite[Prop. 3.2.4]{zapata2020unitfree}.
\end{proof}

\subsection{The Dimensioned Poisson Algebra of a Jacobi Manifold} \label{dimPoiss}

Firstly we identify a dimensioned Poisson bracket on the power ring of a Jacobi manifold.

\begin{thm}[Dimensioned Poisson Algebra of a Jacobi Manifold] \label{DimPoissonAlgebraJacobi}
Let $\lambda:L\to M$ be a line bundle and $(\Sec{L},\{\,,\})$ a Jacobi structure, then there exists a unique dimensioned Poisson algebra of dimension $-1$ on the power dimensioned ring $(\Sec{L}^\odot_\Int,+_\Int,\odot_0,\{\,,\}_{-1})$ such that the brackets combining elements of dimensions $+1$, $0$, and $-1$ are determined by the Jacobi bracket $\{\,,\}$ and its symbols $X$, $\Lambda$.
\end{thm}
\begin{proof}
We give an explicit construction of the dimensioned Poisson algebra on the power $(\Sec{L}^\odot_\Int,+_\Int,\odot_0)$, that we regard here as a dimensioned commutative algebra over the real numbers with dimension set $\Int$ and dimensionless commutative multiplication $\odot_0$. Since the Jacobi bracket maps pairs of sections into sections $\{\,,\}:\Sec{L}\times \Sec{L}\to \Sec{L}$, we aim to extend it to all the tensor powers of the power as a dimensioned algebra bracket of dimension $-1\in\Int$:
\begin{equation}
    \{\,,\}_{-1}:\Sec{L^n}\times \Sec{L^m}\to \Sec{L^{n+m-1}}.
\end{equation}
It is clear that we obtain a partial Lie bracket for all positive tensor powers simply by extending the Jacobi bracket as $\odot$-derivations in each argument, i.e. setting $\{a,b\}_{-1}:=\{a,b\}$ and generating all the brackets between higher powers from the basic identity:
\begin{equation}
    \{a,b\odot c\}_{-1}:=\{a,b\}\odot c +b\odot \{a,c\}
\end{equation}
for all $a,b,c\in \Sec{L^1}=\Sec{L}$, note that this is analogous to the isomorphism $\Dr{L}\cong\Dr{L^\odot}_0$ in Proposition \ref{DimensionlessPowerDerivations} which makes the Hamiltonian derivation of the Jacobi bracket $D_a=\{a,-\}$ into a dimensionless derivation of the power. The symbol identity of the Jacobi bracket written in terms of the power dimensioned multiplication reads
\begin{equation}
    \{f\odot a ,g\odot b\}_{-1}=f\odot g \odot \{a,b\} + f\odot X_a[g]\odot b - g\odot X_b[f]\odot a + \Lambda(df\otimes a, dg \otimes b)
\end{equation}
for $f,g\in \Sec{L^0}\cong\Cin{M}$, $a,b\in \Sec{L^1}=\Sec{L}$, then we can extract the definition of the dimensioned Poisson bracket for non-negative tensor powers by reading off the above formula interpreted as a Leibniz rule of the $\odot$ multiplication:
\begin{align}
    \{a,f\}_{-1} &:= X_a[f] = -\{f,a\}_{-1}\\
    \{f,g\}_{-1}(a) &:= \Lambda^\sharp( df\otimes a)[g] = -\{g,f\}_{-1}(a)
\end{align}
Note that the second bracket has been defined on a generic argument since
\begin{equation}
    \{\,,\}_{-1}:\Sec{L^0}\times \Sec{L^0}\to \Sec{L^{-1}}=\Sec{L^*}.
\end{equation}
To define the bracket on negative powers we exploit the isomorphism $\Dr{L}\cong\Dr{L^*}$ to define the Hamiltonian derivation $D_a$ on dual sections $\Delta_a\in\Dr{L^*}$ and set
\begin{equation}
    \{a,\alpha\}_{-1}:= \Delta_a(\alpha)=-\{\alpha,a\}_{-1}.
\end{equation}
Note that this definition is consistent with the previous definitions of brackets of non-negative tensor powers as we readily check that it acts as a $\odot$-derivation in both arguments:
\begin{equation}
    \{a,\alpha \odot b\}_{-1}=\{a,\alpha(b)\}_{-1}=X_a[\alpha(b)]=\Delta_a(\alpha)(b)+\alpha(D_a(b))=\{a,\alpha\}_{-1}\odot b+\alpha \odot \{a,b\}_{-1}.
\end{equation}
With the brackets defined so far for non-negative tensor powers and the mixed bracket above, we can expand the expression $\{f\odot a,\alpha \odot b\}$ by $\odot$-derivations (full details of the computation shown in appendix \ref{CalculationsDimensionedPoissonBracket}) to find the only bracket not yet defined:
\begin{equation}
    \{f,\alpha,\}_{-1}(a,b):=\Lambda^\sharp(df\otimes a)[\alpha(b)]+X_b[f]\alpha(a).
\end{equation}
Similarly, expanding the bracket $\{\alpha \odot a,\beta \odot b\}$ (again, full details in appendix \ref{CalculationsDimensionedPoissonBracket}) we find:
\begin{equation}
    \{\alpha,\beta\}_{-1}(a,b,c):=\Lambda^\sharp(d\alpha(a)\otimes b)[\beta(c)]+X_c[\alpha(a)]\beta(b)-\alpha(b)X_a[\beta(c)]+\alpha(b)\beta(\{a,c\}).
\end{equation}
With these partial brackets we can now define the brackets of combinations of positive and negative tensor powers via extension as $\odot$-derivations. The Jacobi identity of the bracket for the negative tensor powers directly depends on the Jacobi identity for the bracket at dimensions $+1$ and $0$. In appendix \ref{CalculationsDimensionedPoissonBracket} the Jacobi identities for brackets of all the combinations of the tensor powers $+1$ and $0$ are shown to follow directly from the basic identities satisfied by the symbols $X$, $\Lambda$ of the Jacobi bracket $\{\,,\}$ listed in Proposition \ref{JacobiSymbol}.
\end{proof}

Recall that the functor $\text{C}^\infty:\Man \to \Ring$ assigns ordinary Poisson manifolds to Poisson algebras, the next theorem shows that the power construction provides an analogous functor of Jacobi manifolds into dimensioned Poisson algebras.

\begin{thm}[The Power Functor for Jacobi Manifolds] \label{PowerFunctorJacobi}
The assignment of the power of a line bundle restricted to the category of Jacobi manifolds with Jacobi maps gives a contravariant functor
\begin{equation}
\normalfont
    \odot: \Jac_\Man\to \textsf{DimPoissAlg}.
\end{equation}
\end{thm}
\begin{proof}
In proposition \ref{PowerFunctorJacobi} it was shown that a line bundle factor $B:L\to L'$ covering a smooth map $b:M\to M'$ is assigned to a dimensioned ring morphism $B^\odot :\Sec{L'}^\odot\to \Sec{L}^\odot$ under the power contravariant functor, then it suffices to show that when $B$ is a Jacobi map the power map is a dimensioned Lie algebra morphism, i.e. 
\begin{equation}
    B^\odot\{s,r\}_{-1}'=\{B^\odot s,B^\odot r\}_{-1}
\end{equation}
for all $s,r\in \Sec{L'}^\odot$. Note that $B^\odot$ was defined in proposition \ref{PowerFunctorLineBundles} on the tensor powers of $\Sec{L'}$, $\Sec{L'^*}$ and $\Cin{M'}$. It then follows from the fact that the dimensioned Lie brackets $\{\,,\}_{-1}$, $\{\,,\}_{-1}'$ are defined by extension as $\odot$-derivations that the dimensioned Lie algebra morphism condition for brackets of positive and negative tensor powers suffices to hold for all the bracket combinations of elements in dimensions $+1$, $0$ and $-1$. These conditions are checked directly using the definitions. For the bracket of a pair of elements of dimension $+1$, the condition of dimensioned Lie algebra morphism for $B^\odot$ is precisely the condition that $B$ is a Jacobi map. By considering a bracket of the form $\{a,f\cdot b\}'$ we can see that the fact that $B$ is a Jacobi map and the basic properties of the pull-backs of line bundle factors imply the morphism condition for the bracket of elements of dimension $+1$ and $0$:
\begin{equation}
    B^\odot\{a,f\}_{-1}'=b^*X_a'[f]= X_{B^*a}[b^*f]=\{B^*a,b^*f\}_{-1}=\{B^\odot a,B^\odot f\}_{-1}.
\end{equation}
From similar considerations for a bracket of the form $\{f\cdot a,g\cdot b\}$, it follows that
\begin{align}
    B^\odot\{f,g\}_{-1}'(c,d)&=B^*\Lambda'(df \otimes c,dg\otimes d)\\
    &=\Lambda(db^*f \otimes c,db^*g\otimes d)\\
    &=\{b^*f,b^*g\}_{-1}(c,d)\\
    &=\{B^\odot f,B^\odot g\}_{-1}(c,d)
\end{align}
for all $c,d\in \Sec{L}$. To account for brackets containing elements of dimension $-1$ we first consider the defining formula of the isomorphism $\Dr{L}\cong \Dr{L^*}$ under pull-back
\begin{align}
B^*(\Delta_a'(\alpha)(c))&= b^*X_a'[\alpha(c)]- b^*\alpha(\{a,c\}')\\
&= X_{B^*a}[b^*\alpha(c)]-B^*\alpha(B^*\{a,c\})\\
&= X_{B^*a}[B^*\alpha(B^*c)]-B^*\alpha(\{B^*a,B^*c\})\\
&= \Delta_{B^*a}(B^*\alpha)(B^*c).
\end{align}
Which, in particular, implies the dimensioned Lie morphism condition for the bracket of mixed tensor powers
\begin{equation}
    B^\odot\{a,\alpha\}_{-1}'=B^*\Delta_a'(\alpha)=\Delta_{B^*a}(B^*\alpha)=\{B^*a,B^*\alpha\}_{-1}=\{B^\odot a,B^\odot\alpha\}_{-1}.
\end{equation}
In the proof of theorem \ref{DimPoissonAlgebraJacobi} it was shown that the brackets $\{f,\alpha\}_{-1}$ and $\{\alpha,\beta\}_{-1}$ were determined by extending the previously defined brackets between elements of dimensions $+1$, $0$ and $-1$ as $\odot$-derivations, thus the dimensioned Lie algebra morphism condition for these follows from the fact that $B^\odot$ is defined as the tensor powers of $B^*$.
\end{proof}

The dimensioned multiderivation and de Rham complex technologies for line bundles were introduced in Section \ref{DerJet} as direct analogues of their ordinary smooth manifold counterparts, in light of the two theorems above, we should expect to find dimensioned counterparts of Poisson bivectors and Poisson cohomology. Indeed, we define the \textbf{dimensioned Poisson biderivation} $\pi\in\text{Der}^2(L^\odot)$ of a Jacobi manifold $(\Sec{L},\{\,,\})$ via the usual identity
\begin{equation}
    \pi(\delta s,\delta r):=\delta \{s , r\}_{-1}
\end{equation}
for all $s,r\in\Sec{L}^\odot$. Proposition \ref{JetsPowerJets} implies that the dimensioned Poisson biderivation $\pi$ recovers the ordinary Jacobi biderivation $\Pi$ and symbol $\Lambda$ when restricted to elements of dimension $+1$ and $0$, respectively. Furthermore, it is easy to show -- in fact following an argument entirely analogous to the one for ordinary Poisson manifolds -- that the Jacobi identity of the dimensioned Poisson bracket can be captured as the vanishing condition of the dimensioned Schouten bracket:
\begin{equation}
    \llbracket \pi,\pi \rrbracket = 0.
\end{equation}
It is then straightforward to form the \textbf{dimensioned Lichnerowicz complex} $(\text{Der}^\bullet(L^\odot),\wedge,\llbracket \pi, - \rrbracket)$ whose \textbf{dimensioned Poisson cohomology} $\text{H}^\bullet_\pi(L^\odot)$ captures the ordinary Jacobi algebroid cohomology of the Jacobi manifold $(\Sec{L},\{\,,\})$ and is, of course, a direct generalisation of the ordinary Poisson cohomology.

\subsection{Algebraic Quotients of Jacobi Manifolds} \label{Quotients}

The language of coisotropic calculus, in which morphisms are regarded as coisotropic submanifolds of products, allows for the translation between geometric and algebraic pictures in ordinary Poisson geometry. As introduced at the start of Section \ref{JacToDim}, coisotropic submanifolds can be defined in Jacobi geometry analogously to the ordinary Poisson case by taking the unit-free approach. In this section we show that the dimensioned technology provides, once again, the right algebraic counterparts to coisotropic submanifolds. 

\begin{prop} [Coisotropic Submanifolds induce Dimensioned Coisotropes] \label{CoisotropicSubmanifoldsDimensionedCoisotropes}
Let $\lambda:L\to M$ be a line bundle and $(\Sec{L},\{\,,\})$ a Jacobi structure, then the vanishing dimensioned ideal of a coistropic submanifold $i:S \hookrightarrow M$ is a dimensioned coisotrope of the dimensioned Poisson algebra on the power
\begin{equation}
    I_S\subset (\Sec{L}^\odot_\Int,+_\Int,\odot_0,\{\,,\}_{-1}).
\end{equation}
\end{prop}
\begin{proof}
Proposition \ref{DimVanishingIdeal} shows that $I_S\subset \Sec{L}^\odot$ is a dimensioned $\odot$-ideal for any submanifold $S$, then it suffices to show that $I_S$ is a dimensioned Lie subalgebra. The vanishing ideal is generated by the $\odot$-products of elements of dimension $+1$, $0$ and $-1$, then, by the Leibniz identity of the dimensioned Poisson bracket, it suffices to check the dimensioned Lie subalgebra conditions for elements of those dimensions. For this we will use characterizations 2, 3 and 4 of coisotropic submanifolds of a Jacobi manifold given in \cite[Prop. 3.2.1]{zapata2020unitfree}. Clearly the condition on the brackets of positive powers $\{a,b\}_{-1}$ for $a,b\in\Sec{L}$ is the fact that the vanishing submodule of sections of $S$ forms a Lie subalgebra of the Jacobi structure, characterization 3. For the bracket $\{a,f\}_{-1}=X_a[f]$ it is the fact that Hamiltonian vector fields of vanishing sections are tangent to the submanifold, i.e. $X_{\Gamma_S}[I_S]\subset I_S$, characterization 4. For the bracket $\{f,g\}_{-1}$ it is the condition $\Lambda^\sharp(dI_S\otimes L)\subset \Tan S$, characterization 2. From the observation that
\begin{equation}
    a\in I_S\cap \Sec{L^1} \Rightarrow \alpha \odot a=\alpha(a)\in I_S
\end{equation}
for any $\alpha\in\Sec{L^*}$, we check the dimensioned Lie subalgebra condition for the brackets $\{f,\alpha\}_{-1}$ and $\{\alpha,\beta\}_{-1}$ by writing the explicit defining formulas presented in the proof of theorem \ref{DimPoissonAlgebraJacobi} and using characterizations 2 and 4 combined. 
\end{proof}

In Section \ref{dimPoisson} dimensioned Poisson algebras were shown to reduce via coisotropes, again, just like in the case of ordinary Poisson algebras. The next theorem establishes the dimensioned-algebraic analogue of coisotropic reduction in the form of dimensioned Poisson reduction.

\begin{thm}[Coisotropic Reduction induces Dimensioned Poisson Reduction] \label{CoisotropicReductionDimPoissonReduction}
Let a Jacobi manifold $(\Sec{L},\{\,,\})$ and let $i:S \hookrightarrow M$ be a coisotropic submanifold satisfying the assumptions of proposition \ref{CoisotropicReductionJacobi} so that there is a a reduced Jacobi structure $(\Sec{L'},\{\,,\}')$ fitting in the reduction diagram:
\begin{equation}
    \begin{tikzcd}[sep=tiny]
    L_S \arrow[rr,"\iota"] \arrow[dd,"\pi"'] \arrow[dr]& & L \arrow[dr] & \\
    & S \arrow[rr,"i"', hook] \arrow[dd, "p",twoheadrightarrow] & & M \\
    L' \arrow[dr] & & & \\
    & M' &  & 
    \end{tikzcd}
\end{equation}
then there is an isomorphism of dimensioned Poisson algebras between the power of the reduced Jacobi structure and the algebraic dimensioned Poisson reduction by the vanishing dimensioned coisotrope:
\begin{equation}
    \Sec{L'}^\odot\cong N(I_S)/I_S.
\end{equation}
\end{thm}
\begin{proof}
The Jacobi reducibility condition was given explicitly in terms of the brackets as
\begin{equation}
    \pi^*\{a_1,a_2\}'=\iota^*\{A_1,A_2\}
\end{equation}
for all $a_i\in \Sec{L'}$ and $A_i\in\Sec{L}$ extensions satisfying $\pi^*a_i=\iota^*A_i$. The definition of the power of a line bundle factor of proposition \ref{PowerFunctorLineBundles} and the explicit definition of the dimensioned bracket $\{\,,\}_{-1}$ clearly show that the reducibility condition translates into the power setting mutatis mutandis as one finds that the dimensioned Poisson brackets on the powers of $L$ and $L'$ are related by the following condition
\begin{equation}
    \pi^\odot\{a_1,a_2\}'_{-1}=\iota^\odot\{A_1,A_2\}_{-1}
\end{equation}
for all $a_i\in \Sec{L'}^\odot$ and $A_i\in\Sec{L}^\odot$ extensions satisfying $\pi^\odot a_i=\iota^\odot A_i$. We aim to relate the dimensioned Lie idealizer of the vanishing dimensioned ideal $N(I_S)$ to the submersion factor $\pi:L_S\to L'$ in a natural way. This will follow by the compatibility condition assumed in proposition \ref{CoisotropicReductionJacobi} for the coistropic submanifold:
\begin{equation}
    \delta(\Ker{\Der \pi})=\Lambda^\sharp((\Tan S)^{\text{0}L})
\end{equation}
which, exploiting the jet sequence of the Jacobi structure, can be rewritten as
\begin{equation}
    \Ker{\Tan p}=(\tilde{\Lambda}^\sharp\circ i)(\Tan^0S \otimes L_S).
\end{equation}
This equation gives the point-wise condition that the $p$-fibration on $S$ is a foliation integrating the tangent distribution of Hamiltonian vector fields of the vanishing sections. It follows from the observation made at the end of the proof of proposition \ref{CoisotropicSubmanifoldsDimensionedCoisotropes} that the Lie idealizer $N(I_S)$ is generated by the brackets of elements of dimension $+1$ and $0$, then it suffices to identify the elements satisfying the dimensioned Lie idealizer defining condition of these dimensions. These will be $f\in\Sec{L^0}\cong \Cin{M}$ and $s\in\Sec{L^1}=\Sec{L}$ such that
\begin{equation}
    \{f,g\}_{-1}\in I_S, \qquad \{s,g\}_{-1}\in I_S, \qquad \{s,a\}_{-1}\in I_S
\end{equation}
for all $g\in I_S\cap \Sec{L^0}$ and $a\in I_S\cap \Sec{L^1}$. From the explicit formulas given for the dimensioned bracket $\{\,,\}_{-1}$ given in the proof of theorem \ref{DimPoissonAlgebraJacobi} we clearly see that the compatibility condition of the coisotropic submanifold with the submersion factor gives a point-wise identification of elements in the idealizer and the infinitesimal description of the $p$-fibration with the restricted line bundles on $S$. Since, by assumption, the submersion factor fits in a reduction scheme of (smooth) Jacobi manifolds this infinitesimal identification carries over globally to allow the identification of $N(I_S)$ with the line bundle $\pi$-fibration. Quotienting by $I_S$, as seen in proposition \ref{DimVanishingIdeal}, amounts to restricting to the submanifold $S$, thus we find the equivalent description of the reduced bracket $\{\,,\}_{-1}'$ as the canonical dimensioned Poisson bracket on the quotient $N(I_S)/I_S$.
\end{proof}

\subsection{Algebraic Products of Jacobi Manifolds} \label{Products}

The unit-free approach allows for a geometric definition of Jacobi products based on the product of line bundles in Proposition \ref{ProductLineBundle}. The proposition below shows that this construction is entirely analogous to the definition of Poisson products on ordinary Cartesian products of smooth manifolds.

\begin{prop}[Product of Jacobi Manifolds I]\label{ProductJacobi1}
Let Jacobi structures $(\Sec{L_1},\{\,,\}_1)$ and $(\Sec{L_2},\{\,,\}_2)$, then there exists a unique Jacobi structure in the line product $(\Sec{L_1\utimes L_2},\{\,,\}_{12})$ such that the canonical projection factors
\begin{equation}
\begin{tikzcd}
L_1 & L_1\utimes L_2 \arrow[l,"P_1"']\arrow[r,"P_2"] & L_2
\end{tikzcd}
\end{equation}
are Jacobi maps defining the product bracket $\{\,,\}_{12}$ via:
\begin{align}
    \{P_1^*-,P_1^*-\}_{12}&:=P_1^*\{-,-\}_1\\ \{P_2^*-,P_2^*-\}_{12}&:=P_2^*\{-,-\}_2\\ \{P_1^*-,P_2^*-\}_{12}&:=0
\end{align}
together with analogous conditions for the symbols $X_{12}$ and $\Lambda_{12}$. We call $(\Sec{L_1\utimes L_2},\{\,,\}_{12})$ the \textbf{product Jacobi structure}.
\end{prop}
\begin{proof}
\cite[Prop. 3.2.2]{zapata2020unitfree}.
\end{proof}
This approach to Jacobi products has some clear shortcomings that we set out to remedy. First note that the construction of the product Jacobi bracket $\{\,,\}_{12}$ above is considerably more subtle than that of the ordinary Poisson product, despite both operating analogously at a structural level within their respective categories. The subtlety stems from the fact that the base product $M_1 \dtimes M_2$ carries information from the line bundles $L_1$ and $L_2$ and, consequently, its cotangent space cannot be fully spanned by pairs of functions on each factor $M_1$ and $M_2$. In \cite{zapata2020unitfree} this is resolved by introducing a class of spanning functions defined from pairs of locally non-vanishing sections (see \cite[Prop. 2.2.5]{zapata2020unitfree}): given non-vanishing sections $u_1\in \Sec{L_1}$ and $u_2\in\Sec{L_2}$ is it possible to define functions on the base product
\begin{equation}
    \frac{u_1}{u_2} \in \Cin{M_1\dtimes M_2}
\end{equation}
such that their differentials span the cotangent bundle and can thus be exploited together with the local nature of the brackets to define the product Jacobi bracket on general arguments. In fact, this manner of quotient operation for sections somewhat foreshadowed our definition of the power dimensioned ring since it is easy to check that
\begin{equation}
    \frac{u_1}{u_2} = P_1^\odot u_1 \odot P_2^\odot u_2^{-1}
\end{equation}
where $u_2^{-1}\in\Sec{L_2^*}$ stands for the unique dual locally non-vanishing section such that $u_2^{-1}(u_2)=1\in\Cin{M_2}$.\newline

These added technical complexities at the geometric level turn into proper obstacles when considering the algebraic counterpart to the product of Jacobi manifolds. Recall that the product of ordinary Poisson manifolds $(M_1,\{\,,\}_1)$ and $(M_2,\{\,,\}_2)$ can be recovered algebraically from the tensor product of Poisson algebras since the usual inclusion of products of functions turns into an inclusion of Poisson algebras
\begin{equation}
    \Cin{M_1}\otimes_\Real \Cin{M_2} \hookrightarrow \Cin{M_1 \times M_2}
\end{equation}
where the product bracket satisfies
\begin{equation}
    \{f_1f_2,g_1g_2\}=\{f_1,g_1\}f_2g_2 + f_1g_1\{f_2,g_2\}
\end{equation}
for all $f_1,g_1\in \Proj_1^*\Cin{M_1}$ and $f_2,g_2\in\Proj_2^*\Cin{M_2}$. With the power functor technology at hand, we are now ready to give the dimensioned-algebraic analogue of this construction.

\begin{thm}[Product of Jacobi Manifolds II]\label{ProductJacobi2}
Let the dimensioned Poisson algebras of two Jacobi manifolds $(\Sec{L_1}^\odot,\{\,,\}_1)$ and $(\Sec{L_2}^\odot,\{\,,\}_2)$, then there exists a unique dimensioned Poisson structure $(\Sec{L_1, L_2}^\odot,\{\,,\}_{12})$ such that the powers of the canonical projections
\begin{equation}
\begin{tikzcd}
\Sec{L_1}^\odot \arrow[r,"P_1^\odot"] & \Sec{L_1, L_2}^\odot & \Sec{L_2}^\odot \arrow[l,"P_2^\odot"']
\end{tikzcd}
\end{equation}
are dimensioned Poisson algebra homomorphisms and the product bracket $\{\,,\}_{12}$ satisfies:
\begin{align}
    \{P_1^\odot-,P_1^\odot-\}_{12} &:=P_1^\odot\{-,-\}_1\\ \{P_2^\odot-,P_2^\odot-\}_{12} &:=P_2^\odot\{-,-\}_2\\ \{P_1^\odot-,P_2^\odot-\}_{12} &:=0.
\end{align}
Furthermore, there is an inclusion of dimensioned Poisson algebras
\begin{equation}
    \Sec{L_1}^\odot \otimes_\Real \Sec{L_2}^\odot \hookrightarrow \Sec{L_1,L_2}^\odot
\end{equation}
where the tensor product is defined as in Proposition \ref{DimensionedPoissonProductHomo}.
\end{thm}
\begin{proof}
We have set things up in such a way that this proof is now entirely analogous to the one for ordinary Poisson products (see for instance \cite{fernandes2014lectures}). The defining formulas for the bracket $\{\,,\}_{12}$ restricted to each factor $P_i^\odot\Sec{L_i}^\odot$ can be used in combination with the Leibniz identity to extend the bracket to arbitrary elements of the power dimensioned ring $\Sec{L_1,L_2}^\odot$. The bracket $\{\,,\}_{12}$ then inherits the antisymmetry and Jacobi identity directly from each of the factor brackets $\{\,,\}_i$ and gets the product dimension $(-1,0)\in\Int^2$ as a homogeneous dimensioned algebra multiplication. It only remains to show the inclusion. First note that Proposition \ref{DimensionedPoissonProductHomo} requires the dimensioned Poisson algebras to be over the same dimension monoid and for their multiplication dimensions to satisfy a multiplicativity condition. This is indeed the case since the dimension set is $\Int$ for both and their dimensions are $0$ and $-1$ which trivially satisfy $0+(-1)=(-1)+0$, therefore the left-hand-side of the inclusion is a well-defined dimensioned Poisson algebra. The inclusion map is the same as in Proposition \ref{PowerLineBundleProduct2} and a simple computation gives
\begin{equation}
    \{s_1\odot s_2,r_1\odot r_2\}_{12}=\{s_1,r_1\}_{12} \odot s_2 \odot r_2 + s_1 \odot r_1 \odot \{s_2,r_2\}_{12}
\end{equation}
for all $s_i,r_i\in P_i^\odot\Sec{L_i}^\odot$. Using the defining equations of the bracket $\{\,,\}_{12}$ we recover the expression of the tensor product in Proposition \ref{DimensionedPoissonProductHomo}, hence showing that the inclusion is a dimensioned Poisson algebra homomorphism.
\end{proof}

In our search for algebraic closure we have been pushed into considering dimensioned Poisson algebras that do not strictly correspond to Jacobi manifolds in a conventional sense, i.e. $(\Sec{L_1, L_2}^\odot,\{\,,\}_{12})$ doesn't correspond uniquely to a single Jacobi manifold despite there being projections to the $P_i^\odot\Sec{L_i}^\odot$ subalgebras that do. The structure resulting from the product of two Jacobi manifolds is the motivating example of what will be called a \textbf{poly-Jacobi manifold} in Section \ref{polyJacobiMan} below.

\section{Poly-Jacobi Geometry} \label{polyJacobiGeo}

The dimensioned technology has proven very effective at capturing the algebraic counterparts of line bundles and Jacobi manifolds. However, in the process of seeking for algebraic closure and functoriality, we have been pushed to consider objects slightly more general than ordinary line bundles and Jacobi manifolds. This was first noted when the power dimensioned ring construction was naturally extended to collections of lines or collections of line bundles over the same base. Collections of line bundles were also necessary to develop the correct theory of algebraic products of line bundles, as seen in Proposition \ref{PowerLineBundleProduct2}, and Theorem \ref{ProductJacobi2} uses a dimensioned Poisson structure defined from a collection of line bundles that does not correspond to a single Jacobi manifold in any conventional sense. It seems that collections of line bundles are being forced onto us if we want to give a parsimonious algebraic formulation Jacobi geometry.\newline 

This is perhaps hinting at the physical significance of line bundles as phase spaces whose observables have not been given a choice of unit. Indeed, if one is to recover observables of a physical system  behaving according to the rules of dimensional analysis, the power dimensioned ring of a collection of lines, each representing a fundamental measurable quantity, must be used. This was discussed in detail in \cite{zapata2021dimensioned} and, although the question of introducing physical quantities more explicitly into differential geometry is not strictly in the scope of the present work, these considerations further legitimise the investigation of collections of line bundles as the natural generalisations of ordinary smooth manifolds.

\subsection{Poly-Line Bundles} \label{polyLine}

Our aim is to define a category whose objects generalise line bundles while retaining the functorial assignment of dimensioned algebras. We begin by defining a \textbf{poly-line bundle} simply as an ordered collection of line bundles over a common base manifold
\begin{equation}
\begin{tikzcd}
L_1 \cdots L_m \arrow[d,"...",shift right=2.5]\arrow[d, shift left=2.3]\\
M
\end{tikzcd}
\end{equation}
We may also denote a poly-line bundle by $\overline{L}_M^m$, $\overline{L}_M$ or $\overline{L}$ depending on context. A \textbf{morphism of poly-line bundles} or, simply, a \textbf{factor} is an ordered collection of pair-wise line bundle morphisms covering the same smooth map between the base manifolds
\begin{equation}
\begin{tikzcd}
L_1 \cdots L_m \arrow[d,"...",shift right=2.5]\arrow[d, shift left=2.3] \arrow[r, "\overline{B}"] & L_1' \cdots L_n' \arrow[d,"...",shift right=2.5]\arrow[d, shift left=2.3]\\
M \arrow[r,"\varphi"'] & N
\end{tikzcd}
\end{equation}
here $\overline{B}$ stands for the collection of line bundle morphisms $B^i:L_i\to L_j'$ with $i\in\{1,\dots,m\}$, all covering the same smooth map $\varphi: M \to N$. By defining composition of factors in the obvious way we obtain the \textbf{category of poly-line bundles} $\textsf{polyLine}_\Man$. It follows by construction that the ordinary category of line bundles is recovered as a full subcategory of poly-line bundles $\Line_\Man\subset \textsf{polyLine}_\Man$.\newline

Since we are dealing with collections of ordinary line bundles, many of the notions introduced in Section \ref{LineBundles} extend trivially to the category of poly-line bundles. \textbf{Submersion}, \textbf{embedding} and \textbf{diffeomorphic factors} are defined in the obvious way and it is easy to see that Proposition \ref{SubManLineBundle}, stating that a submanifold $i:S\hookrightarrow M$ induces and embedding factor of line bundles, extends trivially to poly-line bundles.\newline

Similarly to the line-bundle case, Cartesian products of poly-line bundles are defined via the space of all fibre-wise invertible maps. The \textbf{base product} of two poly-line bundles $\overline{L}_{M_1}^m$ and $\overline{L}_{M_2}^n$ is defined as the set of all fibre-wise pair-wise line morphisms
\begin{equation}
    M_1 \dtimes M_2 := \left\{ B^{ij}_{xy}:(L_x)_i \to (L_y)_j, \, (x,y)\in M_1\times M_2\,, i\in\{1,\dots,m\}\,, j\in\{1,\dots, n\}\right\}.
\end{equation}
This definition clearly recovers the base product of ordinary line bundles, in particular, the projection maps $p_i:M_1 \dtimes M_2 \to M_i$ are defined similarly. We also find the generalisation of Proposition \ref{BaseProductLineBundles}.

\begin{prop}[Base Product of poly-Line Bundles]\label{BaseProductpolyLineBundles}
The base product of poly-line bundles $M_1\dtimes M_2$ is a smooth manifold and the natural $\Real^\times$-actions given by fibre-wise multiplication of factors make $M_1\dtimes M_2$ into a principal bundle
\begin{equation}
    \begin{tikzcd}[column sep=0.1em]
        (\Real^\times)^{nm} & \Acts & M_1\dtimes M_2  \arrow[d,"p_1\times p_2"]\\
         & & M_1\times M_2.
    \end{tikzcd}
\end{equation}
\end{prop}
\begin{proof}
This result follows directly from Proposition \ref{BaseProductLineBundles} and the observation that 
\begin{equation}
    M_1\dtimes M_2 = \bigcup_{\substack{i\in\{1,\dots,m\}\\j\in\{1,\dots, n\}}} M_1^i\dtimes M_2^j
\end{equation}
where $M_1^i\dtimes M_2^j$ denotes the base product of the ordinary line bundles $(L_i)_{M_1}$ and $(L_j)_{M_2}$.
\end{proof}

The first indication that the base product of poly-line bundles is the correct generalisation of the Cartesian product of ordinary manifolds comes in the form of the poly-line bundle version of the graph of a map: morphisms of poly-line bundles $\overline{B}:\overline{L}_{M_1}^m \to \overline{L}_{M_2}^n$ covering a smooth map $\varphi: M_1\to M_2$ can be equivalently characterised as special submanifolds of the base product $M_1\dtimes M_2$, the correspondence is given by defining the following subset:
\begin{equation}
    \text{graph}(\overline{B}):=\{C_{xy}:(L_i)_x\to (L_j)_y \,|\quad \varphi(x)=y, \quad C_{xy}=B^i_x\} \subset M_1\dtimes M_2
\end{equation}
which is easily shown to be a submanifold and is then called the \textbf{graph} of the factor $\overline{B}$. We confirm that the base product construction generalises Cartesian products appropriately from the fact that the projections 
\begin{equation}
\begin{tikzcd}
M_1 & M_1 \dtimes M_2\arrow[l,"p_1"']\arrow[r,"p_2"] & M_2
\end{tikzcd}
\end{equation}
allow us to pull-back the line bundles onto the base product to form a poly-line bundle:
\begin{equation}
    \overline{L} \utimes \overline{L}' := p_1^*L_1 \cdots p_1^*L_m \, p_2^*L_1' \cdots p_2^*L_n'
\end{equation}
which turns out to be a well-defined categorical product.

\begin{prop}[Product of poly-Line Bundles]\label{ProductpolyLineBundle}
The line product construction $\utimes$ is a categorical product
\begin{equation}
    \normalfont\utimes:\textsf{polyLine}_\Man \times \textsf{polyLine}_\Man \to \textsf{polyLine}_\Man
\end{equation}
with the corresponding projection factors fitting in the commutative diagram:
\begin{equation}\label{polyLineProductCommutativeDiagram}
\begin{tikzcd}
\overline{L}_1 \arrow[d] & \overline{L}_1\utimes \overline{L}_2 \arrow[l,"\overline{P}_1"']\arrow[d]\arrow[r,"\overline{P}_2"] & \overline{L}_2\arrow[d] \\
M_1 & M_1 \dtimes M_2\arrow[l,"p_1"]\arrow[r,"p_2"'] & M_2
\end{tikzcd}
\end{equation}
\end{prop}
\begin{proof}
This result follows simply by keeping track of the order of collections of line bundles and the ordinary construction of pull-back line bundles. Note that, in contrast with Proposition \ref{ProductLineBundle}, the product of poly-line bundles is defined by concatenation of line bundle pull-backs, consequently, the projection factors $\overline{P}_i$ are simply defined as the canonical projections to each of the ordered lists. This gives the obvious product of poly-line bundle morphisms as simply the ordered union of all the partial pair-wise line bundle morphisms.
\end{proof}

The internal \textbf{symmetries} of a poly-line bundle $\overline{L}_M$ are encapsulated by its base self-product $M \dtimes M$, which becomes a groupoid with the obvious composition of fibre-wise maps and the projections $p_1,p_2: M \dtimes M\to M$ acting as source and target maps. This is, indeed, a generalisation of the general linear groupoid of a line bundle $\text{GL}(L)$. We call $M \dtimes M$ the \textbf{general linear groupoid} of the poly-line bundle $\overline{L}_M$ and denote it by $\text{GL}(\overline{L})$. The group of bisections of this groupoid is the poly-line bundle analogue of the group of diffeomorphisms of ordinary manifolds and the group of automorphisms of line bundles.\newline

Although it seems natural to consider a collection of modules of sections as the algebraic counterpart to a poly-line bundle, we shall ignore this intermediate object and proceed to define the \textbf{power dimensioned ring} of a poly-line bundle $(L_1\cdots L_m)_M$ as
\begin{equation}
    \Sec{L_1\cdots L_m} := \bigcup_{k_1,\dots,k_m\in\Int} \Sec{L_1^{k_1} \otimes \cdots \otimes L_m^{k_m}}
\end{equation}
which carries a $\Int^m$-dimensional addition and a dimensioned multiplication $\odot$ given by the ordered tensor product of power ring products of each of the factors. As in the case of lines, the power of a poly-line bundle can be reconstructed from the tensor product of the power dimensioned rings regarded as dimensioned $\Cin{M}$-algebras:
\begin{equation}\label{dimpoweriso}
    \Sec{L_1\cdots L_m} \cong \Sec{L_1}^\odot \otimes_{\Cin{M}} \cdots \otimes_{\Cin{M}} \Sec{L_m}^\odot.
\end{equation}
This construction generalises that of the power ring of a line bundle in Section \ref{LinesToDim} and it retains the same functoriality property.

\begin{prop}[The Power Functor for poly-Line Bundles] \label{PowerFunctorpolyLineBundles}
The assignment of the power dimensioned ring to a poly-line bundle is a contravariant functor
\begin{equation}
\normalfont
    \odot : \textsf{polyLine}_\Man \to \textsf{DimRing}.
\end{equation}
\end{prop}
\begin{proof}
To prove this statement we follow Proposition \ref{PowerFunctorLineBundles} closely and we give some observations about the power of a factor $\overline{B}:\overline{L}_{M_1}^m \to \overline{L'}_{M_2}^n$. We just need to do some bookkeeping to account for the multiple pair-wise line morphisms between the poly-line bundles. Firstly note that the each of the pair-wise line bundle maps has arbitrary tensor powers $(B^{ij})^k:(L_i)^k\to (L_j')^k$ but, crucially, the zeroth tensor power, corresponding to the pull-back of the smooth map between the base manifolds $b:M_1\to M_2$, is common to all of them. It is then clear that the collection of all the pair-wise powers gives a single dimensioned map between the power rings defined on each dimension slice by:
\begin{equation}
    \overline{B}^\odot|_{(k_1,\dots,k_m)}:= (B^1)^{k_1}\otimes \cdots \otimes (B^m)^{k_m}.
\end{equation}
The images of these powers of $\overline{B}$ are, by construction, subsets of some tensor powers of the image poly-line bundle, hence, using the braiding isomorphisms of tensor products of lines, we see that the union of all the possible tensor powers gives a well-defined map of sets:
\begin{equation}
    \overline{B}^\odot: \Sec{L_1\cdots L_m}\to \Sec{L_1'\cdots L_n'}.
\end{equation}
This map is slice-wise $\Real$-linear and its dimension map $\beta: \Int^m \to \Int^n$ operates as projected addition determined by the assignment between lines given by the collection of pair-wise line bundle maps $\{B^i\,, i=1,\dots,m\}$. This means that $\beta$ acts as the identity on the $j$-th component if there is only one factor $B^i:L_i\to L_j'$ and additively when there are more factors mapping into the same $j$-th target line, i.e. suppose both $B^{i_1}:L_{i_1}\to L_j$ and $B^{i_2}:L_{i_2}\to L_j$, then
\begin{equation}
    \beta: (\dots, n_{i_1}, \dots, n_{i_2},\dots) \mapsto (\dots, n_{i_1} + n_{i_2}, \dots).
\end{equation}
This shows that $\overline{B}^\odot$ is a morphism of dimensional abelian groups, it remains to show that it is a morphism of dimensioned rings. This follows easily from a similar reasoning to the proof of Proposition \ref{PowerFunctorLineBundles} and the key observation that dimensionless ring elements, i.e. functions on the base, are always mapped under pull-back of the smooth map between base manifolds, thus all the multilinear arrangements $\Cin{M}$-linear factors are mapped consistently under arbitrary tensor powers. Functoriality follows from the basic properties of factor pull-backs presented in Section \ref{LineBundles}.
\end{proof}

We can easily extend the notion of unit to poly-line bundles: let a poly-line bundle $(L_1\cdots L_m)_M$, a \textbf{local choice of units} on an open subset $U\subset M$ is a collection of local non-vanishing sections $u_i\in\Sec{L_i}$. The argument used in Proposition \ref{PowerIsAFunctorLine} for the linear case extends easily to the case of sections so it is straightforward to show that a choice of units on $U\subset M$ induces an isomorphism:
\begin{equation}
    \Sec{L_1\cdots L_m}|_U\cong \Cin{U} \times \Int^m
\end{equation}
thus showing that poly-line bundles can be locally described by trivial dimensioned rings.\newline

An embedded submanifold $i:S\hookrightarrow M$ in a poly-line bundle $\overline{L}_M$ induces a collection of restricted subbundles simply by taking the restriction of each individual line bundle, this gives the restricted poly-line bundle with a canonical inclusion factor $\overline{\iota}:\overline{L}_S\to \overline{L}_M$. Similarly to the case of ordinary line bundles, restrictions to submanifolds appear as quotients of dimensioned rings under the power functor.

\begin{prop}[Vanishing Dimensioned Ideal of a Submanifold in a poly-Line Bundle] \label{PolyDimVanishingIdeal}
An embedded submanifold $i:S\hookrightarrow M$ in a poly-line bundle $\overline{L}_M$ defines a dimensioned ideal $I_S\subset \Sec{\overline{L}}$ that allows to characterize (perhaps only locally) the restricted power as a quotient of dimensioned rings
\begin{equation}
    \Sec{\overline{L}_S}\cong \Sec{\overline{L}}/I_S.
\end{equation}
\end{prop}
\begin{proof}
The vanishing ideal $I_S$ is defined in the obvious way:
\begin{equation}
    I_S:=\{s\in \Sec{\overline{L}} \, | \quad \overline{\iota}^\odot s=0\}.
\end{equation}
By trivially regarding $\Real\subset \Cin{M}$, isomorphism (\ref{dimpoweriso}) can be specialised for commutative $\Real$-algebras:
\begin{equation}
    \Sec{L_1\cdots L_m} \cong \Sec{L_1}^\odot \otimes_\Real \cdots \otimes_\Real \Sec{L_m}^\odot,
\end{equation}
then it follows that $I_S$ is generated by $\odot$-multiplication of the dimensioned vanishing ideals of each line bundle $I^i_S$. The result then follows as a direct consequence of Proposition \ref{DimVanishingIdeal}.
\end{proof}

The definition of poly-line bundles was motivated by the interaction of the product of ordinary line bundles with the power functor, it is then unsurprising that the product of poly-line bundles has a natural dimensioned correlate under the power functor.

\begin{prop}[Power Dimensioned Ring of poly-Line Bundle Products] \label{PowerpolyLineBundleProduct}
Let $\overline{L}_1$ and $\overline{L}_2$ be poly-line bundles, then there is a natural inclusion of dimensioned rings
\begin{equation}
    \Sec{\overline{L}_1} \otimes_\Real \Sec{\overline{L}_2} \hookrightarrow \Sec{\overline{L}_1 \utimes \overline{L}_2}
\end{equation}
where the tensor product is taken as commutative dimensioned $\Real$-algebras.
\end{prop}
\begin{proof}
This result follows by applying the argument used in the proof of Proposition \ref{PowerLineBundleProduct2} to a pair of collections of lines. Recall that we have canonical projection factors 
\begin{equation}\label{polyLineProductCommutativeDiagram}
\begin{tikzcd}
\overline{L}_1 & \overline{L}_1\utimes \overline{L}_2 \arrow[l,"\overline{P}_1"']\arrow[r,"\overline{P}_2"] & \overline{L}_2
\end{tikzcd}
\end{equation}
then power functor for poly-line bundles allows us to write the inclusion map explicitly:
\begin{equation}
    s_1 \otimes s_2 \mapsto \overline{P_1}^\odot s_1 \odot \overline{P_2}^\odot s_2
\end{equation}
for all $s_1\in\Sec{\overline{L_1}}$ and $s_2\in\Sec{\overline{L_2}}$.
\end{proof}

The results presented so far in this section show that dimensioned rings are indeed the right algebraic counterparts to poly-line bundles when regarded as generalisations of smooth manifolds. In fact, at the algebraic level, there is no essential difference between the dimensioned structures associated to a single line bundle or to a collection of line bundles. Therefore, we can extend the standard formalisms of multiderivations and differential forms to poly-line bundles in a straightforward manner. Let a poly-line bundle $\overline{L}_M$, in  analogy with vector fields on ordinary manifolds, we define the \textbf{dimensioned derivations} as
\begin{equation}
    \Dr{L_1\cdots L_m}:= \Dr{\Sec{L_1\cdots L_m}}.
\end{equation}
By taking the commutator bracket and constructing the obvious dimensioned $\Sec{L_1\cdots L_m}$-modules of antisymmetric tensor powers, we find the \textbf{dimensioned multiderivations} $(\text{Der}^\bullet(L_1\cdots L_m),\wedge)$ with a Schouten bracket defined via:
\begin{align}
    \llbracket D,D' \rrbracket &= [D,D']\\
    \llbracket D,r \rrbracket &= D(r)
\end{align}
for all $D\in\Dr{L_1\cdots L_m}$ and $r\in\Sec{L_1\cdots L_m}$. We call $(\text{Der}^\bullet(L_1\cdots L_m),\wedge,\llbracket\,,\rrbracket)$ the \textbf{dimensioned Gerstenhaber algebra} of the poly-line bundle $\overline{L}_M$. Dually to this construction we find the \textbf{Cartan calculus} of a poly-line bundle: \textbf{dimensioned 1-forms} are defined as the $\Sec{L_1\cdots L_m}$-dual of derivations:
\begin{equation}
    \Omega^1(L_1\cdots L_m):=\left \{\Dr{L_1\cdots L_m} \to \Sec{L_1\cdots L_m}, \, \Sec{L_1\cdots L_m}\text{-linear}\right \}
\end{equation}
which carry a natural $\Sec{L_1\cdots L_m}$-module structure. The \textbf{dimensioned differential}:
\begin{equation}
    \delta: \Sec{L_1\cdots L_m} \to \Omega^1(L_1\cdots L_m)
\end{equation}
is defined in the standard way
\begin{equation}
    \delta r (D): = D(r)
\end{equation}
for all $r\in \Sec{L_1\cdots L_m}$ and $D\in \Dr{L_1\cdots L_m}$. By constructing the dimensioned $\Sec{L_1\cdots L_m}$-modules of antisymmetric $k$-forms on derivations $\Omega^k(L_1\cdots L_m)$ and extending $\delta$ as an exterior derivation in the obvious way, we obtain the \textbf{dimensioned de Rham complex}:
\begin{equation}
    (\Omega^\bullet(L_1\cdots L_m),\wedge,\delta).
\end{equation}

\subsection{Generalising Jacobi Manifolds} \label{GeneralisingJacobi}

In Section \ref{JacToDim} we introduced Jacobi manifolds as line bundles carrying additional structure on their modules of sections. Since poly-line bundles have been shown to be the natural generalisations of line bundles, an obvious way to generalise Jacobi manifolds is to consider additional structure on the associated dimensioned rings. Before we propose some tentative definitions in Section \ref{polyJacobiMan} below, we briefly discuss equivalent characterisations of Jacobi manifolds that may lead to different generalisations when extended to poly-line bundles.\newline

\textbf{1. Local Lie Algebras} \cite{tortorella2017deformations}\textbf{.} The definition of a Jacobi structure most commonly found in the literature uses the notion local Lie brackets, i.e. brackets acting as differential operators in each argument, either on real functions or on sections of a line bundle. As a reminder, recall that the ring structure of functions on a manifold $\Cin{M}$ or the $\Cin{M}$-module structure of the sections of a line bundle $\Sec{L}$ can be used to define differential operators as the generalisation of $\Cin{M}$-linear maps; for instance, zeroth and first degree differential operators on $\Sec{L}$ are defined as $\Real$-linear maps commuting with nested products of functions
\begin{align}
    \text{Diff}_0(L) &:=\left \{ \Delta: \Sec{L}\to \Sec{L}\, |  \quad [\Delta, (f\cdot)]=0, \quad \forall f\in\Cin{M} \right\}\\
    \text{Diff}_1(L) &:=\left \{ \Delta: \Sec{L}\to \Sec{L}\, |  \quad [[\Delta, (f\cdot)],(g\cdot)]=0, \quad \forall f,g\in\Cin{M} \right\}.
\end{align}
Differential operators of arbitrary degree form a filtered (non-commutative) associative algebra with composition and, consequently, a Lie algebra with the commutator bracket. In this picture, a Jacobi structure on a line bundle $L$ is defined as a Lie bracket $\{\,,\}:\Sec{L}\times \Sec{L}\to \Sec{L}$ that acts as a first degree differential operator in each argument or, equivalently, whose adjoint map is a morphism of Lie algebras
\begin{equation}
    \text{ad}_{\{\,,\,\}}:\Sec{L}\to \text{Diff}_1(L).
\end{equation}
Since a poly-line bundle $(L_1\cdots L_m)$ was argued to be the generalisation of a smooth manifold, with its power dimensioned ring $\Sec{L_1\cdots L_m}$ being the analogue of the smooth functions $\Cin{M}$, the direct analogue of a module of sections would be a dimensioned module $\mathcal{S}$ over the dimensioned ring $\Sec{L_1\cdots L_m}$. We could interpret this as the algebraic counterpart of a `line bundle over a poly-line bundle'. The properties of dimensioned modules discussed in Section \ref{dimStruc} ensure that the theory of \textbf{dimensioned differential operators} can be developed in direct analogy with the ordinary case -- in particular, it is easy to see that we can reproduce the above definition of first degree differential operators $\text{Diff}_1(\mathcal{S})$ by demanding commutativity properties with the dimensioned module structure and the $\odot$-multiplication. This presents a clear route to define \textbf{dimensioned Jacobi structures} as a dimensioned Lie algebra structure $(\mathcal{S},\{,\})$ such that the bracket acts as a dimensioned differential operator of first degree on each argument. Ordinary Jacobi manifolds are then recovered as a limit case where the underlying poly-line bundle has only one line $L$ and the dimensioned module is just the power dimensioned ring $\mathcal{S}=\Sec{L}^\odot$ regarded as a dimensioned module over itself.\newline

\textbf{2. Unit-Free Poisson Manifolds} \cite{zapata2020unitfree}\textbf{.} The unit-free approach of Section \ref{LineBundles}, in which sections of line bundles generalised ordinary smooth functions, together with the dimensioned technology of Proposition \ref{DimPoissonAlgebraJacobi}, allows to regard Jacobi manifolds as the direct analogues of Poisson manifolds. In this analogy, the Jacobi bracket is characterised as a dimensioned Poisson algebra on the power ring of the underlying line bundle. Extending to poly-line bundles, it is then straightforward to consider dimensioned Poisson algebras on the associated power dimensioned rings as the natural generalisation of Jacobi manifolds. This route leads to the definition of a \textbf{poly-Jacobi structure} on a poly-line bundle $L_1\cdots L_m$ simply as a dimensioned Poisson algebra on its power dimensioned ring $(\Sec{L_1\cdots L_m},\{\,,\})$. This notion encompasses ordinary Jacobi structures as a limit case where the poly-line bundle is a single line and the dimensioned Poisson bracket has dimension $-1\in\Int$.\newline

\textbf{3. Jacobi Algebroids and Dirac-Jacobi Bundles} \cite{vitagliano2015dirac}\textbf{.} It is well-known that the theory of Lie algebroids and Dirac structures can be extended to the context of Jacobi geometry. Similarly to how a Poisson manifold $(M,\pi)$ can be regarded as a Lie algebroid structure on the cotangent bundle $(\Cot M,\pi^\sharp,[\,,]_\pi)$ or as a class of transversal Dirac structures on the standard Courant algebroid $\Tan M\oplus \Cot M$, a Jacobi manifold $(L,\{\,,\})$ can be characterised by Jacobi algebroid structure on the jet bundle $(\Sec{\Jet^1 L},[\,,])$ or a class of Dirac-Jacobi bundles on $\Der L\oplus \Jet^1 L$. When seen through the dimensioned algebra lens, Lie algebroids and Dirac structures admit a much more direct generalisation: a \textbf{dimensioned Lie-Rinehart algebra} is a pair $(\mathcal{S},\mathcal{A})$ where $\mathcal{S}$ is a dimensioned commutative algebra and $\mathcal{A}$ is a dimensioned Lie algebra such that the usual Lie-Rinehart algebra axioms hold. A \textbf{poly-Lie algebroid} is then the geometric counterpart of a dimensioned Lie-Rinehart algebra on the power dimensioned ring of a poly-line bundle. A prime example of a poly-Lie algebroid is given by the dimensioned derivations of a poly-line bundle with the commutator bracket $(\Dr{\overline{L}},[\,,])$. A theory of \textbf{poly-Dirac structures} follows analogously to the ordinary case. It is easy to see that a dimensioned Poisson structure on a poly-line bundle, e.g. one given by an ordinary Jacobi manifold, induces a dimensioned Lie-Rinehart algebra structure on dimensioned 1-forms by means of a construction similar to the cotangent Lie algebroid of a Poisson manifold. These considerations then suggest poly-Dirac structures as generalisations of Jacobi manifolds.

\subsection{Poly-Jacobi Manifolds}\label{polyJacobiMan}

The notion of poly-Jacobi structure deserves our special attention among the multiple generalisations discussed in Section \ref{GeneralisingJacobi} above. Not only is it the most direct generalisation of ordinary Jacobi manifolds into the dimensioned setting but it also encompasses Poisson structures on an equal basis. In this section we give precise definitions of poly-Jacobi manifolds, present some preliminary results and discuss the subtleties resulting from the dimensioned nature of the algebras associated with them.\newline

A general \textbf{poly-Jacobi manifold} is defined as a poly-line bundle $\overline{L}_M$ whose power dimensioned ring carries a dimensioned Poisson algebra structure $\left (\Sec{\overline{L}},\odot,\{\,,\}\right )$. A poly-line bundle morphism is called a \textbf{poly-Jacobi map} when the induced map between the power dimensioned rings is a dimensioned Poisson algebra homomorphism -- in fact, it suffices to require a dimensioned Lie algebra homomorphism, since powers are always dimensioned commutative algebra homomorphisms by construction. Note that this definition is identical to that of Poisson structures on ordinary smooth manifolds but transposed to the category of poly-line bundles $\textsf{polyLine}_\Man$. In Section \ref{polyLine} we showed that the categorical structure of $\textsf{polyLine}_\Man$ is analogous to that of ordinary smooth manifolds and thus we should expect most of the familiar results in ordinary Poisson geometry to appear in the context of poly-Jacobi manifolds.\newline

The dimensioned Poisson bracket of a poly-Jacobi manifold $\left (\Sec{\overline{L}},\odot,\{\,,\}\right )$ acts as a $\odot$-derivation on each argument, hence, using the exterior calculus on poly-line bundles presented at the end of Section \ref{polyLine}, we can write the bracket as
\begin{equation}
    \{r,s\}=\pi(\delta r, \delta s)
\end{equation}
for all $r,s\in\Sec{\overline{L}}$. The dimensioned $\Sec{\overline{L}}$-bilinear form $\pi\in\text{Der}^2(\overline{L})$ is called the \textbf{dimensioned Poisson biderivation} and the Jacobi identity of the dimensioned bracket $\{\,,\}$ can be shown to be equivalent to the following condition in the dimensioned Gerstenhaber algebra:
\begin{equation}
    \llbracket \pi, \pi \rrbracket=0,
\end{equation}
inducing the obvious \textbf{dimensioned Poisson cohomology} $\text{H}_\pi(\overline{L})$. The Cartan calculus on poly-line bundles allows to express the dimensioned Poisson structure as a dimensioned Lie-Rinehart algebra on the space of dimensioned 1-forms $\Omega^1(\overline{L})$ by setting:
\begin{equation}
    [\alpha,\beta]_\pi:=\mathcal{L}_{\pi^\sharp(\alpha)} \beta -\mathcal{L}_{\pi^\sharp(\beta)}\alpha -\delta (\pi(\alpha,\beta))
\end{equation}
for $\alpha,\beta\in\Omega^1(\overline{L})$ and where $\pi^\sharp: \Omega^1(\overline{L})\to \Dr{\overline{L}}$ is the dimensioned musical map induced from $\pi$ as a bilinear form. This is, of course, the analogue of the ordinary Lie-Poisson algebroid on the cotangent bundle and so we call $(\Omega^1(\overline{L}),\pi^\sharp,[\,,]_\pi)$ the \textbf{poly-Lie-Poisson algebroid}.\newline

An embedded submanifold $i:C\hookrightarrow M$ in a poly-Jacobi manifold $\overline{L}_M$ is called \textbf{coisotropic} if its dimensioned vanishing ideal $I_C\subset \Sec{\overline{L}}$ is closed under the dimensioned Poisson bracket
\begin{equation}
    \{I_C,I_C\}\subset I_C.
\end{equation}
Algebraically, this is simply the condition that the vanishing ideal becomes a coisotrope in the dimensioned Poisson algebra $(\Sec{\overline{L}},\odot,\{\,,\})$.\newline

Let $i:C\hookrightarrow M$ an embedded submanifold, a poly-Jacobi manifold $(\Sec{\overline{L}_M},\{,\})$ \textbf{reduces} to the poly-Jacobi structure $(\Sec{L'},\{\,\}')$ via $\overline{\pi}:\overline{L}_C\to\overline{L}'$ when for all pairs $s_1,s_2\in\Sec{\overline{L}'}$ the identity
\begin{equation*}
    \overline{\pi}^\odot\{s_1,s_2\}'=\overline{\iota}^\odot\{S_1,S_2\}
\end{equation*}
holds for all choices of extensions $S_1,S_2$, i.e all choices of sections $S_1,S_2\in\Sec{\overline{L}}$ satisfying 
\begin{equation*}
    \overline{\pi}^\odot s_1=\overline{\iota}^\odot S_1 \qquad \overline{\pi}^\odot s_2=\overline{\iota}^\odot S_2.
\end{equation*}
This notion is directly analogous to ordinary Poisson reduction in the sense of \cite{marsden1986reduction}. When $C$ is a coisotropic manifold satisfying some technical compatibility conditions with the projection factor $\overline{\pi}$, a \textbf{coisotropic reduction} result will likely follow of which Proposition \ref{CoisotropicReductionDimPoissonReduction} would become a limit case for poly-line bundles with a single line.\newline

It is remarkable how, beyond some elementary bookkeeping of ordered lists of tensor products, the theory of poly-line bundles and poly-Jacobi manifolds has not required any significant update from the standard definitions and constructions in ordinary Poisson geometry. This is yet more evidence supporting our claim that dimensioned algebra provides the right language in which to formulate Jacobi and Poisson geometry on an equal footing. Somewhat surprisingly, this is no longer the case when we aim to construct products of general poly-Jacobi manifolds.\newline

Poly-Jacobi manifolds were anticipated in Section \ref{Products} after Proposition \ref{ProductJacobi2} established the power dimensioned ring $\Sec{L_1,L_2}^\odot$ as the appropriate algebraic counterpart to the line bundle product of two ordinary Jacobi manifolds $L_1$ and $L_2$. Although our definition of poly-Jacobi manifold is indeed directly inspired by this example, it is important to note that the dimensioned Poisson bracket $\{\,,\}_{12}$ in Proposition \ref{ProductJacobi2} is very particular and far from a generic poly-Jacobi structure on $\Sec{L_1,L_2}^\odot$. Recall that our definition of dimensioned Poisson algebra (see Section \ref{dimPoisson}) assumes some fixed homogeneous dimensions (see end of Section \ref{dimStruc}) for the commutative and Lie multiplications. In all our definitions so far, the commutative multiplication has been $\odot$ in the power dimensioned ring which, by construction, is a dimensionless multiplication, i.e. it operates as multiplication by the identity element in the dimension monoid. When we defined poly-Jacobi manifolds above we omitted the fixed homogeneous dimension of the bracket since it has not been relevant for any of the constructions introduced so far. Let us fix some notation: a poly-Jacobi manifold $(\Sec{L_1\cdots L_m},\odot,\{\,,\}_{\overline{k}})$ is said to be of \textbf{dimension} $\overline{k}\in\Int^m$ when the dimension map of the bracket $\beta: \Int^m\times \Int^m\to \Int^m$ is given by:
\begin{equation}
    (\overline{n},\overline{m})\mapsto \overline{n}+\overline{m}+\overline{k}.
\end{equation}
The product bracket of two Jacobi manifolds $(\Sec{L_1L_2},\{\,,\}_{12})$ constructed in Proposition $\ref{ProductJacobi2}$ has dimension $(-1,0)\in \Int^2$ which corresponds to the projected dimensions $-1\in\Int$ of each of the two Jacobi structures onto the first component. This poly-Jacobi structure is special since it is entirely captured by the line bundle product construction of Section \ref{LineBundles}. The two possible brackets that can be defined under the symmetry isomorphism $L_1\utimes L_2\cong L_2\utimes L_1$ correspond simply to the trivial inclusions of $\Sec{L_1}^\odot$ and $\Sec{L_2}^\odot$ into the product power dimensioned ring $\Sec{L_1L_2}\cong \Sec{L_1}^\odot\otimes \Sec{L_2}^\odot$.\newline

We are interested in the general problem of defining products of poly-Jacobi manifolds: given two poly-Jacobi manifolds $(\Sec{\overline{L}},\{\,,\}_{\overline{k}})$ and $(\Sec{\overline{L}'},\{\,,\}_{\overline{q}})$, can we find a natural dimensioned Poisson algebra on $\Sec{\overline{L}\utimes \overline{L'}}$? Since all of our development of poly-Jacobi theory so far has mirrored the constructions of ordinary Poisson geometry, it is reasonable to attempt the analogous construction to the product of Poisson manifolds (see the remark before Proposition \ref{ProductJacobi2}). This seems promising at first in light of Proposition \ref{PowerpolyLineBundleProduct}, which ensures that the dimensioned commutative algebras $(\Sec{\overline{L}},\odot)$ and $(\Sec{\overline{L}'},\odot)$ appear in $(\Sec{\overline{L}\utimes\overline{L}'},\odot)$ as the tensor product. We are then tempted to define the product dimensioned Poisson bracket via the usual formula:
\begin{equation}\label{generaldimPoissproduct}
    \{s_1\otimes s_2,r_1\otimes r_2\}\stackrel{?}{:=}\{s_1,r_1\}_{\overline{k}}\otimes s_2 \odot r_2 + s_1 \odot r_1 \otimes \{s_2,r_2\}_{\overline{q}}
\end{equation}
for all $s_i,r_i\in\Sec{p_i^*\overline{L}_i}$. There is an immediate issue with this formula: the dimensions of the two terms on the right don't match unless $\overline{k}=\overline{0}$ and $\overline{q}=\overline{0}$. Hence, the product bracket cannot be defined in general. This limitation of the dimensioned setting is to be expected from the conditions on the multiplication dimensions of Propositions \ref{DimensionedPoissonProductHetero} and \ref{DimensionedPoissonProductHomo}, which gave the algebraic constructions of tensor products of dimensioned Poisson algebras. The construction of the product of two ordinary Jacobi manifolds discussed above uses the homogeneous tensor product (Proposition \ref{DimensionedPoissonProductHomo}), so-called from the fact that all the modules and rings have the same underlying dimension monoid. The more general situation -- when the two dimensioned Poisson algebras have different underlying dimension monoids -- calls for the use of the heterogeneous tensor product (Proposition \ref{DimensionedPoissonProductHetero}), which will impose some restrictions on the admissible dimensions of the dimensioned Poisson brackets.\newline

The first obvious way to construct a product is to assume that formula (\ref{generaldimPoissproduct}) can be used as is. This forces us to consider \textbf{dimensionless poly-Jacobi manifolds} $(\overline{L}_M,\{\,,\}_{\overline{0}})$ whose brackets have zero dimension. Interestingly, dimensionless poly-Jacobi manifolds recover ordinary Poisson manifolds on their dimensionless slice: since the bracket has dimension $\overline{0}$, the restriction to the dimensionless slice $\Sec{\overline{L}}^{\overline{0}}=\Cin{M}$ gives a map:
\begin{equation}
    \{\,,\}_0:=\{\,,\}_{\overline{0}}|_{\Sec{\overline{L}}^{\overline{0}}}:\Cin{M}\times \Cin{M}\to \Cin{M}
\end{equation}
which satisfies the antisymmetry and Jacobi identities from the fact that $\{\,,\}_{\overline{0}}$ does. Recall that the $\odot$ product restricts to the ordinary point-wise product on functions, then the general Leibniz condition for the dimensioned Poisson bracket $\{\,,\}_{\overline{0}}$ implies the Leibniz identity of the bracket $\{\,,\}_0$ with respect to the ordinary product of functions. It follows that the manifold of a dimensionless poly-Jacobi manifold carries an ordinary Poisson structure.In accordance with this observation we shall rename dimensionless poly-Jacobi manifolds as \textbf{poly-Poisson manifolds}. Products then work as expected.

\begin{prop}[Product of poly-Poisson Manifolds]\label{ProductpolyPoiss}
Let two poly-Poisson manifolds $(\Sec{\overline{L}_1},\{\,,\}^1_{\overline{0}})$ and $(\Sec{\overline{L}_2},\{\,,\}^2_{\overline{0}})$, then there exists a unique poly-Poisson structure on the poly-line bundle product $(\Sec{\overline{L}_1 \utimes \overline{L}_2},\{\,,\}^{12}_{\overline{0}})$ such that the power maps of the canonical projections
\begin{equation}
\begin{tikzcd}
\Sec{\overline{L}_1} \arrow[r,"\overline{P}_1^\odot"] & \Sec{\overline{L}_1 \utimes \overline{L}_2} & \Sec{\overline{L}_2} \arrow[l,"\overline{P}_2^\odot"']
\end{tikzcd}
\end{equation}
are dimensioned Poisson algebra morphisms that induce an inclusion of dimensioned Poisson algebras
\begin{equation}
    \Sec{{\overline{L}}_1} \otimes_\Real \Sec{{\overline{L}}_2} \hookrightarrow \Sec{\overline{L}_1 \utimes \overline{L}_2}
\end{equation}
where the tensor product is defined as in Proposition \ref{DimensionedPoissonProductHetero}.
\end{prop}
\begin{proof}
This result follows directly from Proposition \ref{ProductpolyLineBundle} by defining the product bracket via formula (\ref{generaldimPoissproduct}). It is straightforward to check that this is indeed equivalent to the usual defining conditions of product brackets by means of commuting subalgebras:
\begin{align}
    \{\overline{P}_1^\odot-,\overline{P}_1^\odot-\}^{12}_{\overline{0}} &:=\overline{P}_1^\odot\{-,-\}^1_{\overline{0}}\\ \{\overline{P}_2^\odot-,\overline{P}_2^\odot-\}^{12}_{\overline{0}} &:=\overline{P}_2^\odot\{-,-\}^2_{\overline{0}}\\ \{\overline{P}_1^\odot-,\overline{P}_2^\odot-\}^{12}_{\overline{0}} &:=0.
\end{align}
\end{proof}
It is important to emphasise that a poly-Poisson manifold $(\overline{L}_M,\{\,,\}_{\overline{0}})$ carries a great deal more information than simply its dimensionless Poisson algebra $(\Cin{M},\{\,,\}_0)$. To illustrate this consider a poly-Poisson structure on a single line bundle $L_M$, i.e. a dimensioned Poisson algebra $(\Sec{L}^\odot,\odot,\{\,,\})$ of dimension $0\in\Int$. As argued above, the dimensionless bracket gives a Poisson structure on the base manifold $(\Cin{M},\{\,,\}_0)$, however, for non-zero powers, the bracket is given by a collection differential operators between tensor powers of the line bundle. For instance, in the dimension 1 slice we have the $\Real$-bilinear form:
\begin{equation}
    \{\,,\}_1: \Sec{L}\times \Sec{L}\to \Sec{L\otimes L}
\end{equation}
which interacts with the $\Cin{M}$-module structure via the Leibniz identity
\begin{equation}
    \{s,f\cdot r\}_1=X_s(f)\otimes r + f\cdot \{s,r\}_1
\end{equation}
where the action of a section s on a function is given by the dimensioned bracket between the slices of dimension 0 and 1:
\begin{equation}
    X_s(f):=\{s,f\}\in \Sec{L}.
\end{equation}
Another manifestation of the extra data in poly-Poisson manifolds strictly beyond the Poisson manifold data appears in the product construction. Consider, again for simplicity, two poly-Poisson manifolds on single line bundles $(\Sec{L_1}^\odot,\odot,\{\,,\}^1)$ and $(\Sec{L_2}^\odot,\odot,\{\,,\}^2)$. Proposition \ref{ProductpolyPoiss} gives the product poly-Poisson structure $(\Sec{L_1 \utimes L_2}^\odot,\odot,\{\,,\}^{12})$ and, consequently, the dimensionless bracket gives a Poisson structure on the base product $M_1 \dtimes M_2$ such that the canonical projections
\begin{equation}
\begin{tikzcd}
M_1 & M_1\dtimes M_2 \arrow[l,"p_1"']\arrow[r,"p_2"] & M_2
\end{tikzcd}
\end{equation}
are Poisson maps with respect to the corresponding dimensionless Poisson brackets on each factor. Recall that Proposition \ref{BaseProductLineBundles} implies that $M_1 \dtimes M_2\ncong M_1 \times M_2$, therefore we see that the product dimensionless Poisson manifold $(\Cin{M_1\dtimes M_2}, \{\,,\}^{12}_0)$ carries strictly more information than the ordinary product of Poisson manifolds $(\Cin{M_1\times M_2}, \{\,,\}^1_0+\{\,,\}^2_0)$. These two Poisson structures can be related non-canonically by choosing local trivialisations.\newline

In \cite[Sec. 3.3]{zapata2020unitfree} unit-free Poisson manifolds were identified as a special class of Jacobi manifolds $(\Sec{L_M},\{\,,\})$ characterised by the existence of global Poisson units, i.e. non-vanishing sections $u\in\Sec{L}$ with vanishing symbol $X_u=0$ that induce ordinary Poisson structures on the base manifold $(\Cin{M},\{\,,\}_u)$. Since a global unit induces the isomorphism $\Sec{L}\cong \Cin{M}$, a Poisson unit can be equivalently characterised as a Casimir element of the power dimensioned Poisson algebra $(\Sec{L}^\odot,\odot,\{\,,\})$. This observation allows us to easily prove the following statement relating the Poisson-like structures in the unit-free context to the Poisson-like structures in the dimensioned context: a poly-Poisson structure on a single line bundle $(\Sec{L},\{\,,\})$ corresponds to the power of a unit-free Poisson if and only if
\begin{equation}
    \{-,-\}_1=u\odot \{-,-\}_0
\end{equation}
for some global unit Casimir $u\in\Sec{L}$.\newline

In light of these results, poly-Poisson manifolds appear as good candidates for a generalised category of Poisson manifolds. Indeed, it is straightforward to generalise the notions of \textbf{coisotropic calculus} to poly-Poisson manifolds and the analogues to the classic results in Poisson geometry should follow easily. Importantly, however, similar subtleties and limitations appear when one seeks for strict categorical structures in the class of poly-Poisson category. See Section \ref{commentary} for more comments on this.\newline

Going beyond poly-Poisson manifolds, we can see that the product Poisson bracket formula (\ref{generaldimPoissproduct}) may be slightly modified -- supplying some additional information -- to give a well-defined product of two general poly-Jacobi manifolds. To see this, consider two general poly-Jacobi manifolds $(\Sec{\overline{L}},\{\,,\}_k)$ and $(\Sec{\overline{L}'},\{\,,\}_q)$, and note that the dimensions of each term of (\ref{generaldimPoissproduct}) are:
\begin{equation}
    (n_1+m_1+k,n_2+m_2),\qquad (n_1+m_1,n_2+m_2+q)
\end{equation}
where $n_i\in\Int^n$ and $m_i\in\Int^m$ are the dimensions of the generic arguments and $k\in\Int^n$ and $q\in\Int^m$ are the dimensions of the brackets. The idea is to compensate both terms so that they match and addition is well-defined. Since we have already used $\odot$-multiplication in the defining formula of the product, a simple way to do this is by $\odot$-multiplication of elements of dimension $-k$ and $-q$ in each bracket component. Indeed, consider $u_1\in\Sec{\overline{L}_1}^k$ and $u_2\in\Sec{\overline{L}_2}^q$, then the terms of the modified expression
\begin{equation}\label{moddimPoissproduct}
    \{s_1\otimes s_2,r_1\otimes r_2\}:=u_1\odot\{s_1,r_1\}_k\otimes s_2 \odot r_2 + s_1 \odot r_1 \otimes u_2\odot\{s_2,r_2\}_q
\end{equation}
have matching dimension and thus the bracket is well-defined. As a dimensioned multiplication on the tensor product of dimensioned modules, this is a $\Real$-bilinear antisymmetric bracket with dimension $(\overline{0},\overline{0})\in\Int^n\times\Int^m$. This bracket acts as a $\odot$-derivation by construction but the Jacobi identity, however, is obstructed by the behaviour of the chosen elements $u_1$ and $u_2$ within their respective dimensioned Lie algebras. The following theorem shows that the existence of general poly-Jacobi products depends on the existence of Casimir elements.

\begin{thm}[Product of poly-Jacobi Manifolds]\label{ProductpolyJacobi}
Let two poly-Jacobi manifolds $(\Sec{\overline{L}_1},\{\,,\}^1_{\overline{k}})$ and $(\Sec{\overline{L}_2},\{\,,\}^2_{\overline{q}})$ each with a choice of Casimir $u_1\in\Sec{\overline{L}_1}^{\overline{k}}$ and $u_2\in\Sec{\overline{L}_2}^{\overline{q}}$, then there exists a unique poly-Jacobi structure on the poly-line bundle product $(\Sec{\overline{L}_1 \utimes \overline{L}_2},\{\,,\}^{12}_{\overline{0}})$ such that the power maps of the canonical projections
\begin{equation}
\begin{tikzcd}
\Sec{\overline{L}_1} \arrow[r,"\overline{P}_1^\odot"] & \Sec{\overline{L}_1 \utimes \overline{L}_2} & \Sec{\overline{L}_2} \arrow[l,"\overline{P}_2^\odot"']
\end{tikzcd}
\end{equation}
are dimensioned Poisson algebra morphisms that induce an inclusion of dimensioned Poisson algebras
\begin{equation}
    \Sec{{\overline{L}}_1} \otimes_\Real \Sec{{\overline{L}}_2} \hookrightarrow \Sec{\overline{L}_1 \utimes \overline{L}_2}
\end{equation}
where the tensor product is defined with the bracket (\ref{moddimPoissproduct}).
\end{thm}
\begin{proof}
This construction is identical to the product of poly-Poisson manifolds of Proposition \ref{ProductpolyPoiss} except for the presence of the Casimir elements in the product bracket. Since the Casimir elements enter the bracket by $\odot$-multiplication, the $\Real$-linearity and derivation properties are not affected. In showing the Jacobi identity, however, we find terms involving $\odot$-linear combinations with elements of the form
\begin{equation}
    \{u_1,r\}^1_{\overline{k}}\qquad \{u_2,r\}^2_{\overline{q}}.
\end{equation}
Since $u_1$ and $u_2$ are Casimir elements, i.e. elements of the center of their respective dimensioned Lie algebras, these terms vanish and we thus recover the usual expression for the Jacobi identity of the tensor product of Poisson brackets.
\end{proof}

Note that this product of poly-Jacobi manifolds generalises the product of poly-Poisson manifolds of Proposition \ref{ProductpolyPoiss} since for a pair of dimensionless brackets there is always a canonical choice of dimensionless Casimir elements given by the constant functions $1\in\Cin{M_1}$ and $1\in\Cin{M_2}$.

\section{Commentary} \label{commentary}

\subsection{Product Subtleties}

The product constructions given in Propositions \ref{ProductJacobi1}, \ref{ProductJacobi2} and \ref{ProductpolyJacobi} are all generalisations of the ordinary product of Poisson manifolds and, as such, they involve similar subtleties \cite{weinstein2009symplectic}; in particular, geometric products of Poisson-like structures cannot be promoted to full categorical products. When we consider dimensioned Poisson structures in the poly-line bundle category these issues appear in conjunction with the more explicit obstruction of the existence of Casimirs. Even if we decide to restrict to a subcategory of poly-Jacobi manifolds admitting (global) Casimirs, the product construction doesn't capture the full generality of the possible dimensioned brackets since the resulting bracket of the product of two arbitrary poly-Jacobi manifolds has dimension $0$, i.e. it is a poly-Poisson manifold. We thus see that, even under the necessary technical assumptions, products of general poly-Jacobi manifolds result in poly-Poisson manifolds.

\subsection{Poisson Flavour vs Jacobi Flavour}

Our comments on the product subtleties above are just one aspect of the broader Poisson-vs-Jacobi theme that was already introduced in Section \ref{GeneralisingJacobi} when discussing the ways in which Jacobi manifolds can be generalised. Proposition \ref{PowerFunctorJacobi} shows that Jacobi manifolds can be algebraically characterised as the dimensioned analogue of Poisson algebras, so one would expect that by identifying the right geometric category to support them (poly-line bundles) Poisson and Jacobi manifolds will be recovered on equal footing. This is only true at a very coarse degree of generality -- indeed, both Poisson and Jacobi manifolds are examples of poly-Jacobi manifolds -- however, at the level of dimensions, Poisson and Jacobi brackets are always separate, i.e. Poisson brackets are of dimension $0$ and Jacobi brackets are of dimension $-1$. This separation results in their respective product constructions being fundamentally distinct at the algebraic level. This is understood from the fact that there are two distinct notions of tensor product of dimensioned Poisson algebras, which we proved in Section \ref{dimPoisson}. This is perhaps further evidence supporting the view that the Poisson-vs-Jacobi distinction is essential -- not merely an artifact of the traditional formalisms -- and somewhat analogous to the general ring-vs-module distinction in commutative algebra.

\subsection{Dimensioned Geometry}

In Section \ref{polyLine} we defined the category of poly-line bundles as a dimensioned generalisation of the category of ordinary smooth manifolds, an interpretation that is supported by the several structural results regarding submanifolds, products and the power functor. Alternatively, however, we could take the dimensioned formalism introduced in \cite{zapata2021dimensioned} and develop the theory of dimensioned differential geometry from first principles. The key idea to do so is to replace any instance of the basic field $\Real$ by some dimensioned field $F_G$; in this manner, a dimensioned manifold would be a topological space that is locally modelled by a dimensioned vector space over the dimensioned field $F_G$. Using the obvious definitions suggested by this approach, it is easy to see that poly-line bundles can be recovered as dimensioned manifolds whose base dimensioned field is $\Real\times \Int^k$. This suggests the interesting question of what geometry would look like if dimensioned rings with dimensioned sets other than $\Int^k$ are used in the constructions of dimensioned manifolds.

\subsection{Marrying Dimensional Analysis with Geometric Mechanics}

Recall that beyond the purely mathematical aspects of Jacobi geometry, we further motivated tour research with the following question: \emph{Is it possible to formulate a theory of Hamiltonian mechanics where observables carry the algebraic structure of the standard dimensional analysis of physical quantities?} We are now in the position to offer a definitive positive answer: a physical theory that poses poly-Jacobi manifolds as phase spaces will recover all the desirable notions of classical Hamiltonian mechanics (derivations as dynamics, reductions, products) and the observable assignment (the power functor for poly-line bundles) will indeed map spaces of physical states into algebraic structures that manifestly display the properties of the dimensional analysis of physical quantities. Based on our results in Section \ref{polyJacobiMan} it may be appropriate to restrict to poly-Poisson manifolds as the dimensioned generalisation of the phase spaces of classical Hamiltonian mechanics since they explicitly recover Poisson structures on their dimensionless slices.

\subsection{Topological Metrology?}

All our treatment of poly-line bundles and poly-Jacobi manifolds has emphasised a trivialisation-independent approach; indeed, the (global) trivialisability condition for line bundles has not entered the discussion as it was not necessary in any of the proofs or constructions. Within the `dimensioned phase space' framework described in the section above, the question of trivialisability is directly related with the choice of units for a set of physical dimensions. Units are only defined locally in general, since there may be topological obstructions to the existence of global non-vanishing sections of line bundles, hence implying an explicit interaction between the topology of a phase space and the units of measurement that can be defined on it. This presents a plausible route to describing effects where observables (at the algebraic level) are sensitive to the structure of phase space (at the topological level), an effect that is well-documented for quantum systems but that has never been identified in the context of classical mechanics. More broadly, these considerations open up a path towards `topological metrology', where the structural features of phase spaces directly affect the possible metrological paradigms that can be developed. An interesting question is whether there is phenomena observed in the natural sciences that would be effectively described by this sort of `classical topological metrology'.

\section*{Acknowledgements}

I would like to thank Jos\'e Figueroa-O'Farrill for all the helpful conversations and the generous lunches in the unlikely sun of Edinburgh.

\printbibliography

\appendix

\section{Calculations for the Proof of Theorem \ref{DimPoissonAlgebraJacobi}} \label{CalculationsDimensionedPoissonBracket}

Let $\lambda: L\to M$ be a line bundle with a Jacobi structure $(\Sec{L},\{,\})$, the subindex on the dimensioned Poisson bracket $(\Sec{L}^\odot,\odot,\{,\}_{-1})$ will be omitted for simplicity. In what follows we take $a,b,c\in\Sec{L^1}=\Sec{L}$, $f,g,h\in\Sec{L^0}\cong \Cin{M}$ and $\alpha\in\Sec{L^{-1}}=\Sec{L^*}$.\newline

Consider the bracket $\{f\odot a,\alpha \odot b\}$, expanding as $\odot$-derivations we find
\begin{align}
    \{f\odot a,\alpha \odot b\} &= a \odot \{f,\alpha\}\odot b + f\odot \{a,\alpha\}\odot b + f\odot \{a,b\}\odot \alpha + a\odot \{f,b\}\odot \alpha=\\
    &= \{f,\alpha\}(a,b) + fX_a[\alpha(b)]-f\alpha(\{a,b\})+f\alpha(\{a,b\})-X_b[f]\alpha(a).
\end{align}
But from the definition of the $\odot$ dimensioned multiplication and the symbol-squiggle identity we have:
\begin{align}
    \{f\odot a,\alpha \odot b\} &= \{f\cdot a,\alpha(b)\}=X_{f\cdot a}[\alpha(b)]\\
    &= fX_a[\alpha(b)]+\Lambda^\sharp(df \otimes a)[\alpha(b)],
\end{align}
thus giving the desired bracket after simplifications:
\begin{equation}
    \{f,\alpha\}=\Lambda^\sharp(df \otimes a)[\alpha(b)] + X_b[f]\alpha(a).
\end{equation}

Consider the bracket $\{\alpha\odot a,\beta \odot b\}$, expanding as $\odot$-derivations we find
\begin{align}
    \{\alpha\odot a,\beta \odot b\}(c) &= a \odot \{\alpha,\beta\}\odot b \odot c + \alpha\odot \{a,\beta\}\odot b \odot c \\
    &+ \alpha\odot \{a,b\}\odot \beta \odot c + a\odot \{\alpha,b\}\odot \beta \odot c\\
    &= \{\alpha,\beta\}(a,b,c) +\alpha(b)\Delta_a(\beta)(c) +\alpha(\{a,b\})\beta(c)-\beta(a)\Delta_b(\alpha)(c)\\
    &= \{\alpha,\beta\}(a,b,c) - X_b[\alpha(a)]\beta(c)+\alpha(b)X_a[\beta(c)]-\alpha(b)\beta(\{a,c\}).
\end{align}
But from the definition of the $\odot$ dimensioned multiplication and the definition of the bracket $\{f,g\}$ we have:
\begin{equation}
    \{\alpha\odot a,\beta \odot b\}(c)=\{\alpha (a),\beta(b)\}(c)=\Lambda^\sharp(d\alpha(a)\otimes c)[\beta(b)],
\end{equation}
thus giving the desired bracket after simplifications:
\begin{equation}
     \{\alpha,\beta\}(a,b,c) = \Lambda^\sharp(d\alpha(a)\otimes c)[\beta(b)] + X_b[\alpha(a)]\beta(c)-\alpha(b)X_a[\beta(c)]+\alpha(b)\beta(\{a,c\}).
\end{equation}

We explicitly show that the dimensioned Poisson bracket satisfies the Jacobi identity for all combinations of elements of dimension $+1$ and $0$, which, as argued in the proof of theorem \ref{DimPoissonAlgebraJacobi}, is enough to guarantee the Jacobi identity for all the other dimensions via extension as $\odot$-derivations. Firstly,
\begin{equation}
    \{a,\{b,c\}\}=\{b,\{a,c\}\}+\{\{a,b\},c\}
\end{equation}
follows from the Jacobi identity to the Jacobi bracket $(\Sec{L},\{,\})$ itself. Now 
\begin{align}
    \{a,\{b,f\}\} = \{a,X_b[f]\} = X_a[X_b[f]] &= X_b[X_a[f]] + [X_a,X_b][f] \\
    &= X_b[\{a,f\}] +X_{\{a,b\}}[f]\\
    &=\{b,\{a,f\}\}+\{\{a,b\},f\},
\end{align}
where we have used the fact that the symbol map $X$ is a Lie algebra morphism. Also
\begin{align}
    \{a,\{f,g\}\}(b) = \Delta_a(\{f,g\})(b) &= X_a[\{f,g\}(b)]-\{f,g\}(\{a,b\})\\
    &= X_a[\Lambda^\sharp(df\otimes b)[g]]- \Lambda^\sharp(df\otimes \{a,b\})[g]\\
    &=\Lambda^\sharp(df\otimes b)[X_a[g]] + [X_a,\Lambda^\sharp(df\otimes b)][g] - \Lambda^\sharp(df\otimes \{a,b\})[g]\\
    &=\Lambda^\sharp(df\otimes b)[X_a[g]] + \Lambda^\sharp(dX_a[f]\otimes b)[g]\\
    &+\Lambda^\sharp(df\otimes \{a,b\})[g] - \Lambda^\sharp(df\otimes \{a,b\})[g]\\
    &=\Lambda^\sharp(df\otimes b)[\{a,g\}] + \Lambda^\sharp(d\{a,f\}\otimes b)[g]=\\
    &=\{f,\{a,g\}\}(b)+\{\{a,f\},g\}(b)
\end{align}
where the symbol compatibility condition 3 of proposition \ref{JacobiSymbol} has been used. Lastly
\begin{align}
    \{f,\{g,h\}\}(a,b,c) &= \Lambda(df\otimes a, \Lambda(dg \otimes b , dh\otimes c)) + X_b[f]\cdot \Lambda(dg\otimes a, dh\otimes c)\\
    &= \Lambda(dg\otimes a, \Lambda(df \otimes b , dh\otimes c)) + X_b[g]\cdot \Lambda(df\otimes a, dh\otimes c)\\
    & - \Lambda(dh\otimes a, \Lambda(df \otimes b , dg\otimes c)) - X_b[h]\cdot \Lambda(df\otimes a, dg\otimes c)\\
    &=\Lambda(dg\otimes a, \Lambda(df \otimes b , dh\otimes c)) + X_b[g]\cdot \Lambda(df\otimes a, dh\otimes c)\\
    &+ \Lambda(\Lambda(df \otimes b , dg\otimes c), dh\otimes a) + X_b[h]\cdot \Lambda(dg\otimes a, df\otimes c)\\
    &=\{g,\{f,h\}\}(a,b,c) + \{\{f,g\},h\}(a,b,c)
\end{align}
where the symbol compatibility condition 4 of proposition \ref{JacobiSymbol} has been used.

\end{document}